\documentclass[11pt,reqno]{amsart}

\usepackage{amsmath, amsthm, amssymb}
\usepackage{amsfonts}

\usepackage{hyperref}

\usepackage[ansinew]{inputenc}
\usepackage[dvips]{epsfig}
\usepackage{graphicx}
\usepackage[english]{babel}

\usepackage{thmtools}
\theoremstyle{plain}
\declaretheorem[title=Theorem, parent=section]{theorem}
\declaretheorem[title=Lemma,sibling=theorem]{lemma}
\declaretheorem[title=Proposition,sibling=theorem]{proposition}

\theoremstyle{definition}
\declaretheorem[title=Definition,sibling=theorem]{definition}
\declaretheorem[title=Remark,sibling=theorem]{remark}
\declaretheorem[title=Remark, numbered=no]{remark*}

\declaretheorem[title=Assumption, numbered=no]{assumption*}

\numberwithin{equation}{section}

\usepackage[backgroundcolor=white, bordercolor=blue,
linecolor=blue]{todonotes}

\usepackage{dsfont}
\usepackage{bbm}

\newcommand{\N}{\mathbb{N}}
\newcommand{\R}{\mathbb{R}}

\newcommand{\cD}{\mathcal{D}}

\newcommand{\cE}{\mathcal{E}}

\newcommand{\cI}{\mathcal{I}}

\newcommand{\eps}{\varepsilon}

\newcommand{\1}{\mathbbm{1}}

\DeclareMathOperator{\dist}{dist}

\DeclareMathOperator{\diam}{diam}
\DeclareMathOperator{\supp}{supp}

\DeclareMathOperator{\tail}{Tail}

\renewcommand{\d}{\textnormal{\,d}}

\newcommand{\average}{{\mathchoice {\kern1ex\vcenter{\hrule height.4pt
width 6pt depth0pt} \kern-9.7pt} {\kern1ex\vcenter{\hrule
height.4pt width 4.3pt depth0pt} \kern-7pt} {} {} }}
\newcommand{\dashint}{\average\int}

\begin{document}
\allowdisplaybreaks
\title{Improvement of flatness for nonlocal free boundary problems}

\author{Xavier Ros-Oton}
\author{Marvin Weidner}

\address{ICREA, Pg. Llu\'is Companys 23, 08010 Barcelona, Spain \& Universitat de Barcelona, Departament de Matem\`atiques i Inform\`atica, Gran Via de les Corts Catalanes 585, 08007 Barcelona, Spain \& Centre de Recerca Matem\`atica, Barcelona, Spain}
\email{xros@icrea.cat}

\address{Departament de Matem\`atiques i Inform\`atica, Universitat de Barcelona, Gran Via de les Corts Catalanes 585, 08007 Barcelona, Spain}
\email{mweidner@ub.edu}

\keywords{nonlocal, regularity, one-phase problem, free boundary, flatness}

\subjclass[2020]{47G20, 35B65, 31B05, 35R35}

\allowdisplaybreaks

\begin{abstract}
In this article we study for the first time the regularity of the free boundary in the one-phase free boundary problem driven by a general nonlocal operator. Our main results establish that the free boundary is $C^{1,\alpha}$ near regular points, and that the set of regular free boundary points is open and dense. Moreover, in 2D we classify all blow-up limits and prove that the free boundary is $C^{1,\alpha}$ everywhere. 
The main technical tool of our proof is an improvement of flatness scheme, which we establish in the general framework of viscosity solutions, and which is of independent interest. All of these results were only known for the fractional Laplacian, and are completely new for general nonlocal operators. In contrast to previous works on the fractional Laplacian, our method of proof is purely nonlocal in nature.
\end{abstract}

\allowdisplaybreaks

\maketitle
\section{Introduction}  

Free boundary problems arise in several areas of applied mathematics, such as in probability, finance, and control theory, but also in elasticity theory, combustion theory, material sciences, and fluid dynamics. Moreover, they have constituted a central topic of research in pure mathematics and especially in PDE theory for the last fifty years. The most intriguing and challenging question in this area is the study of the regularity of free boundaries, which was initiated by the pioneering work of Caffarelli \cite{Caf77} on the obstacle problem. Subsequently, numerous techniques have been developed for different kinds of free boundary problems, and they are illustrated for example in \cite{Fri82, CaSa05, PSU12, FeRo22b}. See also \cite{Caf98}, \cite{CSV18}, \cite{FiSe19} for further results on the free boundary in the obstacle problem.

An important class that has received an increasing amount of attention in the last 20 years is the class of \textit{nonlocal free boundary problems}, which arises as a natural model whenever long range interactions need to be taken into account.
Let us give a short overview of the literature on the nonlocal obstacle problem, which has been studied extensively. In comparison to the classical obstacle problem, here the Laplacian is replaced by a general stable integro-differential operator $L$ of order $2s$ for some $s \in (0,1)$. In case $L = (-\Delta)^s$ is the fractional Laplacian, the regularity theory for this problem has been developed in the articles \cite{ACS08, Sil07, CSS08}. A key tool in the study is the Caffarelli-Silvestre extension (see \cite{CaSi07}) which allows to identify the fractional obstacle problem with a local problem (``thin obstacle problem''), where the obstacle, and therefore also the free boundary, is contained in a hyper-plane. We refer to \cite{Fer22} for a survey on the thin obstacle problem.\\
After the seminal works \cite{ACS08, Sil07, CSS08} the case of nonlocal obstacle problems driven by more general nonlocal operators than the fractional Laplacian has remained an open problem for almost a decade. Since in this case no identification with a local problem is possible, the study of this question is particularly challenging. Finally, in \cite{CRS17,FRS23} the problem has been solved, and entirely new techniques have been developed therein to establish the regularity of solutions and of the free boundary. See also \cite{AbRo20}, \cite{RTW25}, \cite{RoWe23}, \cite{RoTo24}, \cite{RoWe24b} for further results in this direction.

Another classical free boundary problem that is widely studied in the literature is the so-called one-phase free boundary problem (``Bernoulli problem''). This problem deals with the analysis of minimizers of the energy functional
\begin{align*}
\int_{B_1} |\nabla u|^2 \d x + |\{ u > 0 \} \cap B_1|.
\end{align*}
The one-phase problem was introduced by Alt and Caffarelli in \cite{AlCa81}. It arises as a model for flame propagation and jet flows, and it is also related to shape optimization. The study of the free boundary $\partial \{ u > 0 \}$ has been initiated in \cite{AlCa81} and the series of papers \cite{Caf87,Caf89,Caf88}. Further landmark contributions on this topic are \cite{CJK04}, \cite{DeJe09}, \cite{JeSa15}, \cite{DeS11}, and we refer to \cite{CaSa05,Vel23} for comprehensive overviews of the theory.

A nonlocal version of the one-phase free boundary problem has been introduced in \cite{CRS10}, replacing the $H^1(B_1)$ seminorm in the energy functional by the $H^s(B_1)$ seminorm. This model is particularly relevant in case turbulence or long range interactions are taken into account. While in \cite{CRS10} the authors establish basic properties of minimizers, such as the optimal $C^s$ regularity and non-degeneracy, the study of the free boundary for the fractional one-phase free boundary problem is carried out in a series of works \cite{DeRo12,DeSa12,DeSa15b,DeSa15} in case $s= \frac{1}{2}$, and in \cite{DSS14,EKPSS21} for general $s \in (0,1)$. See also \cite{DeSa20,AlSm24} for results on almost minimizers and \cite{CSV15} for regularity properties in another fractional version of the Bernoulli problem involving the $s$-perimeter. All of the proofs in the aforementioned articles heavily rely on the Caffarelli-Silvestre extension, which reduces the fractional one-phase problem to a local one-phase problem with a ``thin'' free boundary.

As in the case of the nonlocal obstacle problem (see \cite{CRS17}), a natural research question is to analyze the one-phase free boundary problem for a general $2s$-stable integro-differential operator. However, as opposed to the obstacle problem, apart from our recent work \cite{RoWe24b}, where we establish the optimal $C^s$ regularity and non-degeneracy of minimizers (see also \cite{SnTe24}), there are currently no results available in the literature. In particular, nothing is known about the regularity properties of the free boundary. In parallel to the narrative for the nonlocal obstacle problem (see \cite{CRS17}), the lack of an extension formula calls for the development of purely nonlocal techniques in order to tackle the question of regularity for free boundaries. In this spirit, in \cite[p.1974]{EKPSS21} the authors write that \textit{``[...] at the moment, it seems to be impossible to tackle one-phase problems involving more general operators than the fractional Laplacian. The main point is we do not know how to prove any kind of monotonicity for general integral operators.''}

The goal of this work is precisely to establish for the first time fine regularity results for the free boundary of minimizers to the nonlocal one-phase problem for general nonlocal operators. To be precise, we consider minimizers of the functional
\begin{align}
\label{eq:onephase-intro}
\cI_{\Omega}(u) := \iint_{\R^n \times \R^n} \big(u(x)-u(y) \big)^2 K(x-y) \d y \d x + \big|\{ u > 0 \} \cap \Omega \big|
\end{align}
in a bounded domain $\Omega \subset \R^n$ with prescribed exterior condition $u \equiv g \ge 0$ in $\R^n \setminus \Omega$. The kernel $K : \R^n \to [0,\infty]$ is assumed to satisfy
\begin{align}
\label{eq:Kcomp}
\lambda |h|^{-n-2s} \le K(h) \le \Lambda |h|^{-n-2s}, \qquad K(h) = K(-h), \qquad K(h) = \frac{K(h/|h|)}{|h|^{n+2s}} 
\end{align}
for some $0 < \lambda \le \Lambda$ and $s \in (0,1)$. This class of kernels \eqref{eq:Kcomp} gives rise to integro-differential operators
\begin{align}
\label{eq:L}
L u(x) = 2 ~ \text{p.v} \int_{\R^n} (u(x) - u(y)) K(x-y) \d y.
\end{align}
This family is the natural class of symmetric $2s$-stable integro-differential operators and contains as a special case the fractional Laplacian $(-\Delta)^s$ which corresponds to $K(h) = c(n,s)|h|^{-n-2s}$.

\subsection{Main results}

Our main results establish the regularity of the free boundary for minimizers of the nonlocal one-phase problem \eqref{eq:onephase-intro} governed by a general kernel $K$ satisfying \eqref{eq:Kcomp}.

Our first result shows that the free boundary is of class $C^{1,\alpha}$ for any $\alpha \in (0,\frac{s}{2})$ near any point $x_0 \in \partial \{ u > 0 \}$, where the free boundary is sufficiently flat, i.e., trapped between two parallel  hyper-planes that are close enough. Such result is well-known for the fractional Laplacian (see \cite{DeRo12}, \cite{DSS14}), but completely new for general kernels \eqref{eq:Kcomp}.

\begin{theorem}
\label{thm:main-C1alpha}
Let $K \in C^{1-2s+\beta}(\mathbb{S}^{n-1})$ for some $\beta > \max\{0,2s-1\}$ and assume \eqref{eq:Kcomp}. Let $u$ be a minimizer of $\mathcal{I}_{\Omega}$ with $B_2 \subset \Omega$. Then, there exists $\delta \in (0,1)$, depending only on $n,s,K$, such that if $0 \in \partial \{ u > 0 \}$ and for some $\nu \in \mathbb{S}^{n-1}$ it holds
\begin{align}
\label{eq:fb-flat}
\{ x \cdot \nu \le - \delta \} \cap B_1 \subset \{ u = 0 \} \cap B_1 \subset \{ x \cdot \nu \le \delta \} \cap B_1 ,
\end{align} 
then, $\partial \{ u > 0\} \in C^{1,\alpha}$ in $B_{\rho}$ for any $\alpha \in (0,\frac{s}{2})$, and moreover, 
\begin{align*}
\left\Vert \frac{u}{d^s} \right\Vert_{C^{\alpha}\left( \overline{ \{ u > 0 \} } \cap B_{\rho} \right)} \le C \Vert u \Vert_{L^{1}_{2s}(\R^n)}
\end{align*}
for some $C,\rho > 0$, depending only on $n,s,K,\alpha$.
\end{theorem}

\begin{remark}
As in the case of the one-phase problem for the fractional Laplacian (see \cite{DeSa12,DeSa15b}), we believe that the $C^{1,\alpha}$ regularity of the free boundary near points satisfying \eqref{eq:fb-flat} can be improved, at least for sufficiently smooth kernels $K$. Establishing higher regularity of the free boundary for general nonlocal operators is an interesting question that certainly requires new ideas (see also \cite{AbRo20} for the nonlocal obstacle problem). We plan to investigate this question in the future.
\end{remark}

Our main result \autoref{thm:main-C1alpha} can be interpreted as an analog to the main results in \cite{DeRo12,DSS14} for the fractional Laplacian. This theorem is crucial to the understanding of the free boundary for minimizers of $\cI_{\Omega}$, as it \textit{reduces the question of regularity} of the free boundary near a point $x_0 \in \partial \{ u > 0 \}$ to determining whether the free boundary is \textit{flat near $x_0$}. 

We will show in \autoref{thm:close-to-halfspace} that the free boundary is flat near $x_0 \in \partial \{ u > 0 \}$, in the sense that for any $\delta > 0$ there exists $r > 0$ such that \eqref{eq:fb-flat} holds true for the rescaling $u_{r,x_0}$, if 
\begin{align}
\label{eq:regular-point}
u_{r,x_0}(x) := \frac{u(x_0 + rx)}{r^s} \quad \text{ satisfies } \quad u_{r,x_0} \xrightarrow{r \to 0} A(\nu)(x \cdot \nu)_+^s ~~ \text{ locally uniformly in } \R^n
\end{align}
for some $\nu \in \mathbb{S}^{n-1}$, where
\begin{align}
\label{eq:A-def-intro}
A(\nu) =  c_{n,s} \left( \int_{\mathbb{S}^{n-1}} K(\theta) |\theta \cdot \nu|^{2s} \d \theta \right)^{-\frac{1}{2}}.
\end{align}
$A$ is chosen in such a way that $A(\nu)(x \cdot \nu)_+^s$ is a solution to the nonlocal one-phase problem in the half-space ${\{x \cdot \nu > 0 \}}$ (see \autoref{prop:free-bound-cond}, and also \cite{CRS10,FeRo22}). 

In the light of this observation we can define the set of \textit{regular free boundary points} to consist of all points $x_0 \in \partial \{ u > 0 \}$ for which \eqref{eq:regular-point} holds true. A natural question is then to determine the size of the set of regular points, in order the quantify the portion of the free boundary that is smooth. Let us now present our two main results in this direction.

Our first result establishes that the set of regular points is an open and dense subset of the free boundary. We show that any free boundary point $x_0 \in \partial \{ u > 0 \}$ admitting a tangent ball inside $\{ u > 0 \}$ is a regular point (see \autoref{thm:close-to-halfspace}(iv)). As a consequence, we have the following result, which is, again, completely new for general nonlocal operators, and was only known for the fractional Laplacian (see \cite{DeRo12,DSS14}):

\begin{theorem}
\label{cor:open-dense}
Let $K \in C^{1-2s+\beta}(\mathbb{S}^{n-1})$ for some $\beta > \max\{0,2s-1\}$ and assume \eqref{eq:Kcomp}. Let $u$ be a minimizer of $\mathcal{I}_{\Omega}$ with $\Omega \subset \R^n$. Then, there exists an open, dense set $\mathcal{O} \subset \partial \{ u > 0 \} \cap \Omega$ such that for any $x_0 \in \mathcal{O}$ there exists $\rho > 0$ such that $\partial \{ u > 0\}$ is $C^{1,\alpha}$ in $B_{\rho}(x_0)$.
\end{theorem}

Another way to understand the set of regular points is to classify all possible blow-up limits $\lim_{r \to 0} u_{r,x_0}$. In fact, for the classical one-phase problem ($s = 1$), and for the one-phase problem for the fractional Laplacian, one can show with the aid of a monotonicity formula that all blow-up limits must be homogeneous of degree $s \in (0,1]$. Therefore, the classification of blow-ups reduces to the study of minimal cones. In case $s = 1$, it was shown in the celebrated works \cite{CJK04,DeJe09,JeSa15} that all blow-ups in dimensions $n \le 4$ are half-space solutions, and therefore every free boundary point is regular. For the fractional Laplacian, this property is only known in case $n = 2$ (see \cite{DeSa15,EKPSS21}), and the higher dimensional case is wide open (see \cite{FeRo22} for the classification of axially symmetric cones in case $n \le 5$).

In case of general kernels \eqref{eq:Kcomp}, no monotonicity formulas are available, and therefore establishing homogeneity of blow-ups seems to be out of reach with current techniques. 
Still, in this paper we show that all blow-ups are of the form \eqref{eq:regular-point} when $n=2$ (see \autoref{thm:two-dimensional-classification}). Our proof is a nonlocal version of the competitor argument for the thin energies in \cite[Theorem 5.5]{DeSa15}, \cite[Theorem 6.1]{EKPSS21}. Instead of homogeneity, we make crucial use of a purely nonlocal term appearing in the corresponding nonlocal energy estimate, which was already employed in \cite{CSV19,FiSe19} in a different context.

Since all free boundary points are regular in case $n = 2$, we can establish that free boundaries are everywhere $C^{1,\alpha}$ in two dimensions:

\begin{theorem}
\label{cor:two-dim}
Let $n=2$. Let $K \in C^{2}(\mathbb{S}^{1})$ and assume \eqref{eq:Kcomp}. Let $u$ be a minimizer of $\cI_{\Omega}$ with $\Omega \subset \R^n$. Then, $\partial \{ u > 0\}$ is $C^{1,\alpha}$ in $\Omega$.
\end{theorem}

As was mentioned above, \autoref{cor:two-dim} was only known for the fractional Laplacian (see \cite{DeSa15,EKPSS21}). Our proof is completely independent of previous ones, since it does not rely on the homogeneity of blow-ups, or on the identification with a thin problem. Note that already in dimension $n = 3$ it is not known whether the same result holds true, even for the fractional Laplacian.

\subsection{Strategy of proof}

The overall strategy to prove our main result \autoref{thm:main-C1alpha} follows the one for the classical local one-phase problem, as it is presented in \cite{Vel23}. However due to the nonlocality of \eqref{eq:onephase-intro}, we encounter several significant challenges, and new ideas are required to overcome them. In the following, we give a brief overview of the main steps of our proof.

First of all, we prove that minimizers of $\cI_{B_1}$ are solutions to the following nonlocal Bernoulli-type problem (see \autoref{lemma:min-visc}), which arises as the first variation of \eqref{eq:onephase-intro}:
\begin{align}
\label{eq:first-variation}
\left\{\begin{array}{rcl}
L u &=& 0 ~ \qquad \text{ in } B_1 \cap \{ u > 0 \},\\
u &=& 0 ~ \qquad \text{ in } B_1 \setminus \{ u > 0 \},\\  \displaystyle
\frac{u}{d^s} &=& A(\nu) ~~ \text{ on } B_1 \cap \partial\{ u > 0 \}, 
\end{array}\right.
\end{align}
where $\nu = \nu_x \in \mathbb{S}^{n-1}$ denotes the normal vector to $\partial\{ u > 0 \}$ at $x$, and $A$ is defined as in \eqref{eq:A-def-intro}. Moreover, $d := \dist(\cdot , \partial \{ u > 0 \} )$, and we understand $\frac{u}{d^s}(x_0) :=\lim_{\{ u > 0 \} \ni x \to x_0} \frac{u}{d^s}(x)$.\\
It is worth emphasizing that the anisotropy of $L$ is mirrored in the free boundary condition since the value of $\frac{u}{d^s}$ depends on the normal vector of the free boundary. Since we do not know a priori whether the free boundary is smooth, the anisotropy of the free boundary condition complicates the analysis of the problem. We refer to \cite{DeSa21}, where an anisotropic local Bernoulli problem is analyzed.

The natural setting in which \eqref{eq:first-variation} should be interpreted is  the framework of viscosity solutions (see \autoref{def:viscosity}). Although viscosity solutions to \eqref{eq:first-variation} share several properties with minimizers to \eqref{eq:onephase-intro}, such as the $C^s$ regularity (see \autoref{lemma:visc-Cs}), clearly not every viscosity solution to  \eqref{eq:first-variation} is a minimizer of \eqref{eq:onephase-intro}. However, in this paper (see \autoref{thm:free-boundary-regularity}) we still show that flatness implies $C^{1,\alpha}$ regularity of the free boundary $\partial \{ u > 0 \}$ in the more general realm of viscosity solutions to \eqref{eq:first-variation} (see \autoref{thm:free-boundary-regularity}).

The main idea to prove \autoref{thm:main-C1alpha} is to establish a so-called $\eps$-regularity theory for viscosity solutions to \eqref{eq:first-variation}.
In fact, first, by a compactness argument we prove that if the free boundary is flat near $0 \in \partial \{ u > 0 \}$ in the sense of \eqref{eq:fb-flat} with $\nu := e_n$, then a rescaling $u_r := u_{0,r}$ of $u$ is bounded from above and below by translations of the half-space solution to \eqref{eq:first-variation} introduced in \eqref{eq:regular-point}, namely for $\eps > 0$:
\begin{align}
\label{eq:intro-flat-1}
A(e_n)(x \cdot e_n -\eps)_+^s \le u_r(x) \le A(e_n)(x \cdot e_n + \eps)_+^s ~~ \forall x \in B_1.
\end{align}
Moreover, we have an integral control of the deviation of $u_r$ from the translated half-space solution outside $B_1$ in the following sense for $\delta_0 > 0$:
\begin{align}
\label{eq:intro-flat-2}
\tail \Big([u_r-A(e_n)(x \cdot e_n -\eps)_+^s]_- ; 1 \Big) + \tail \Big([A(e_n)(x \cdot e_n + \eps)_+^s - u_r]_- ; 1 \Big) \le \eps \delta_0.
\end{align}
We refer to Section \ref{sec:prelim} for a definition of the tail term.

The main work consists in proving a so-called improvement of flatness scheme, i.e, to show that when a viscosity solution $u$ to \eqref{eq:first-variation} satisfies \eqref{eq:intro-flat-1}, \eqref{eq:intro-flat-2} for some $\eps,\delta_0 \in (0,1)$ small enough, then a further rescaling $u_{r \rho_0}$ for some uniform $\rho_0 \in (0,1)$ satisfies
\begin{align}
\label{eq:intro-flat-3}
A(\nu)(x \cdot \nu - \sigma \eps)_+^s \le u_{\rho_0 r}(x) \le A(\nu)(x \cdot \nu + \sigma \eps)_+^s ~~ \forall x \in B_1.
\end{align}
and
\begin{align}
\label{eq:intro-flat-4}
\tail \Big([u_{\rho_0 r} -A(\nu)(x \cdot \nu - \sigma \eps)_+^s]_- ; 1 \Big) + \tail\Big([A(\nu)(x \cdot \nu + \sigma \eps)_+^s - u_{\rho_0 r}]_- ; 1 \Big) \le \sigma \eps \delta_0.
\end{align}
for some $\sigma \in (0,1)$ and $\nu \in \mathbb{S}^{n-1}$ with $|e_n - \nu| \le  C \eps$. 

Such improvement of flatness schemes are standard in the context of one-phase free boundary problems since the work of \cite{AlCa81} (see also \cite{DeS11,Vel23}), and they have also been established for the thin one-phase problem in \cite{DeRo12}, \cite{DeSa12}, \cite{DSS14}. In our purely nonlocal framework, a central difficulty comes from long range interactions which need to be included in the iteration scheme through corresponding tail terms in \eqref{eq:intro-flat-2}, \eqref{eq:intro-flat-4}. 

We prove that \eqref{eq:intro-flat-1}, \eqref{eq:intro-flat-2} imply \eqref{eq:intro-flat-3}, \eqref{eq:intro-flat-4} via a contradiction compactness argument, inspired by \cite{DeS11}. First, we establish a certain growth lemma on a fixed scale (see \autoref{lemma:partial-BH-scaled}), which allows us to establish compactness for sequences of viscosity solutions $(u_k)$ to \eqref{eq:first-variation} satisfying \eqref{eq:intro-flat-1}, \eqref{eq:intro-flat-2} with $\eps_k \searrow 0$. Then, the main work is to prove that for such sequence it holds
\begin{align}
\label{eq:lin-convergence-intro}
v_k(x) := \frac{u_k(x) - A(e_n)(x_n)_+^s}{\eps_k} \to s A(e_n)(x_n)_+^{s-1} u(x) ~~ \text{ as } k \to \infty,
\end{align}
where $u$ solves the so-called ``linearized problem'' for some $\omega \in \mathbb{S}^{n-1}$ with $\omega_n \ge c > 0$:  
\begin{align}
\label{eq:linearized-intro}
\left\{\begin{array}{rcl}
L ((x_n)_+^{s-1} u) &=& f ~~ \text{ in } \{ x_n > 0 \} \cap B_1,\\
\partial_{\omega} u &=& 0 ~~ \text{ on } \{ x_n = 0 \} \cap B_1.
\end{array}\right.
\end{align}

Both of these steps are very delicate since, due to the free boundary condition, all derivatives of solutions explode at the free boundary. This is a central difference to the local case, where a key observation is that the corresponding half-space solution $(x_n)_+$ can be smoothly extended to a global harmonic function, making the corresponding sequence $v_k$ in \eqref{eq:lin-convergence-intro} harmonic in $\{ u_k > 0 \}$ (harmonicity with respect to $L$). In the nonlocal case, we overcome this issue by employing \textit{domain variations} (see \eqref{eq:dom-variation}), which allows us to rewrite $v_k$ as a first-order difference quotient of $u_k$ of the following form for some implicit $\tilde{u}_k(x) \in [-1,1]$.
\begin{align*}
v_k(x) = \frac{u_k(x) - u_k(x - \eps_k \tilde{u}_k(x) e_n )}{\eps_k \tilde{u}_k(x)} \tilde{u}_k(x).
\end{align*}
This tool has already appeared in the thin case (see \cite{DeRo12}, \cite{DeSa12}, \cite{DSS14}), however, there, due to the locality of the problem, it is possible to compute domain variations in certain cases. This allows us to construct explicit barrier functions (see \cite[Proposition 4.5]{DSS14}), something that does not seem to be possible in our purely nonlocal setting, and causes our proofs to be significantly more involved.

Once the linearized problem is identified, the properties \eqref{eq:intro-flat-3}, \eqref{eq:intro-flat-4} follow by using that as a solution to \eqref{eq:linearized-intro}, it holds $u \in C^{1,\gamma}(B_{1/2} \cap \{ x_n = 0 \})$ for some $\gamma > 0$. Note that \eqref{eq:linearized-intro} is a nonlocal equation with an oblique local boundary condition. In case $\omega = e_n$ (which is what happens for the fractional Laplacian), \eqref{eq:linearized-intro} becomes a Neumann boundary condition, and the boundary regularity theory for such nonlocal problems has been established by the authors in the recent paper \cite{RoWe24}. In our anisotropic setting, we need to apply a certain change of variables to transform \eqref{eq:linearized-intro} into a nonlocal problem with a local Neumann boundary condition, so that we can apply the results in \cite{RoWe24}.

The improvement of flatness scheme -- namely that \eqref{eq:intro-flat-1}, \eqref{eq:intro-flat-2} imply \eqref{eq:intro-flat-3}, \eqref{eq:intro-flat-4} for viscosity solutions to \eqref{eq:first-variation} -- can be iterated in a relatively standard way (see Subsection \ref{subsec:iterate-improvement-of-flatness}), thereby implying uniqueness of the blow-up limits and a uniform rate of convergence. Together, these properties imply that the free boundary can be parametrized by a $C^{1,\alpha}$ graph in $B_{\rho}$ for some $\rho \in (0,1)$. 

Altogether, these findings yield the following flatness implies $C^{1,\alpha}$ result for viscosity solutions to \eqref{eq:first-variation}:

\begin{theorem}
\label{thm:free-boundary-regularity}
Let $K \in C^{1-2s+\beta}(\mathbb{S}^{n-1})$ for some $\beta > \max\{0,2s-1\}$. Let $u$ be a viscosity solution to the nonlocal one-phase problem for $K$ in $B_2$ and $0 \in \partial \{ u > 0 \}$. Then, there are $\eps, \delta_0 \in (0,1)$, depending only on $n,s,K$, such that if
\begin{align*}
A(e_n)(x_n -\eps)_+^s \le u(x) \le A(e_n)(x_n + \eps)_+^s ~~ \forall x \in B_1,
\end{align*}
and
\begin{align*}
T_{\eps} := \tail \Big([u-A(e_n)(x_n -\eps)_+^s]_- ; 1 \Big) + \tail \Big([A(e_n)(x_n + \eps)_+^s - u]_- ; 1 \Big) \le \eps \delta_0,
\end{align*}
then, $\partial \{ u > 0 \}$ is $C^{1,\alpha}$ in $B_{\rho}$ for any $\alpha \in (0,\frac{s}{2})$, and moreover, 
\begin{align*}
\left\Vert \frac{u}{d^s} \right\Vert_{C^{\alpha}\left( \overline{ \{ u > 0 \} } \cap B_{\rho} \right)} \le C \Vert u \Vert_{L^{1}_{2s}(\R^n)},
\end{align*}
for some $C, \rho > 0$, depending only on $n,s,K$.
\end{theorem}

\begin{remark}
The dependence of the constants $\eps,\delta_0,C,\rho > 0$ can be improved in such a way that they depend on $K$ only through $\lambda,\Lambda,\Vert K \Vert_{C^{1-2s+\beta}(\mathbb{S}^{n-1})}$. To do so, in \eqref{eq:lin-convergence-intro} one needs to consider sequences $(u_k)$ solving \eqref{eq:first-variation} with respect to kernels $K_k$ satisfying \eqref{eq:Kcomp} with $\lambda,\Lambda$ for every $k \in \N$. With this modification, all the proofs go through without any substantial changes.
\end{remark}

\begin{remark}
The assumption $K \in C^{1-2s+\beta}(\mathbb{S}^{n-1})$ for some $\beta > \max\{0,2s-1\}$ can most likely be relaxed, at least in case $s > 1/2$. For $s \le 1/2$ it is only required in the proof of \autoref{lemma:linearized-problem}(iv) to guarantee interior Lipschitz regularity of solutions to the one-phase problem. Moreover, in case $s > 1/2$, we assume H\"older regularity in order to guarantee that $f$ in \autoref{lemma:linearized-problem}(iv) is continuous, which is a technical assumption in order for the notion of viscosity solution to make sense.
\end{remark}

The previous result was known so far only for viscosity solutions to \eqref{eq:first-variation} for the fractional Laplacian due to \cite{DeRo12,DeSa12,DSS14}. However, as was mentioned before, in these papers, the authors exclusively work with the equivalent thin one-phase problem, making their approach entirely local. Our proof, however, is of completely nonlocal nature, and thus entirely new, and independent of the proofs in \cite{DeRo12,DeSa12,DSS14}. Let us also point out that apart from \cite{DSV20}, where a purely nonlocal improvement of flatness scheme has been developed in the context of nonlocal phase transitions, there seem to be no results in this direction in the literature, so far. In fact, \autoref{thm:free-boundary-regularity} seems to be the first nonlocal improvement of flatness result for a nonlocal free boundary problem.

Finally, let us draw the reader's attention to the fact that in the local case when $L = -\Delta$, \autoref{thm:free-boundary-regularity} yields a characterization of regularity properties of (Reifenberg flat) domains in terms of regularity of the Poisson kernel (or harmonic measure) (see \cite{AlCa81}, \cite{Jer90}), since in that case the free boundary condition in \eqref{eq:first-variation} is given by $\partial_{\nu} u = 1$. A similar connection in the nonlocal case has not been explored, yet, and we believe this to be an interesting topic for further research.

\subsection{Acknowledgments}

The authors were supported by the European Research Council under the Grant Agreements No. 801867 (EllipticPDE) and No. 101123223 (SSNSD), and by AEI project PID2021-125021NA-I00 (Spain).
Moreover, X.R was supported by the grant RED2022-134784-T funded by AEI/10.13039/501100011033, by AGAUR Grant 2021 SGR 00087 (Catalunya), and by the Spanish State Research Agency through the Mar\'ia de Maeztu Program for Centers and Units of Excellence in R{\&}D (CEX2020-001084-M).

\subsection{Organization of the paper}

The paper is organized as follows. In Section \ref{sec:prelim} we introduce some notation and recall the basic properties of minimizers of \eqref{eq:onephase-intro}, which were established in \cite{RoWe24b}. Section \ref{sec:viscosity} contains a derivation of the first variation of \eqref{eq:onephase-intro} and the definition of viscosity solutions to \eqref{eq:first-variation}. Moreover, we prove that minimizers are viscosity solutions and establish some basic properties. The proof of the flatness implies $C^{1,\alpha}$ result (see \autoref{thm:free-boundary-regularity} for viscosity solutions is contained in Section \ref{sec:flatness-implies-C1alpha}. Finally, in Section \ref{sec:free-boundary-reg} we establish our main results for minimizers of \eqref{eq:onephase-intro}, namely \autoref{thm:main-C1alpha} and \autoref{cor:open-dense}, \autoref{cor:two-dim}.

\section{Preliminaries}
\label{sec:prelim}

In this section we collect several definitions and auxiliary lemmas that will become important throughout the course of this article. In particular, we recall some basic properties of minimizers of \eqref{eq:onephase-intro} which were established in \cite{RoWe24b}, such as optimal $C^s$ regularity, and non-degeneracy (see Subsection \ref{subsec:aux}), as well as some important properties about blow-ups (see Subsection \ref{sec:blow-ups}).

\subsection{Function spaces and solution concepts}

Given an open, bounded domain $\Omega \subset \R^n$, let us introduce the following function spaces, which are naturally associated with the energy $\cI$ from \eqref{eq:onephase-intro}:
\begin{align*}
V^s(\Omega | \Omega') &:= \left \{ u \hspace{-0.1cm}\mid_{\Omega} \hspace{0.1cm} \in L^2(\Omega) : [u]^2_{V^s(\Omega | \Omega')} := \int_{\Omega} \int_{\Omega'} \frac{(u(x) - u(y))^2}{|x-y|^{n+2s}} \d y \d x < \infty \right \}, ~~ \Omega \Subset \Omega'\\
H^s(\Omega) &:=  \left\{ u \in L^2(\Omega) : [u]^2_{H^s(\Omega)}  := \int_{\Omega} \int_{\Omega} \frac{(u(x) - u(y))^2}{|x-y|^{n+2s}} \d y \d x < \infty \right\},\\
L^1_{2s}(\R^n) &:= \left\{ u : \R^n \to \R : \Vert u \Vert_{L^1_{2s}(\R^n)} := \int_{\R^n} |u(y)| (1 + |y|)^{-n-2s} \d y < \infty \right\}.
\end{align*}
These spaces are equipped with the following norms:
\begin{align*}
\Vert u \Vert_{V^s(\Omega | \Omega')} := \Vert u \Vert_{L^2(\Omega)} + [u]_{V^s(\Omega | \Omega')}, \qquad \Vert u \Vert_{H^s(\Omega)} := \Vert u \Vert_{L^2(\Omega)} + [u]_{H^s(\Omega)}.
\end{align*}

Moreover, the following quantity captures the long-range interactions caused by the nonlocality:
\begin{align*}
\tail(u;R,x_0) := R^{2s} \int_{\R^n \setminus B_R(x_0)} |u(y)||y - x_0|^{-n-2s} \d y, ~~ x_0 \in \R^n, ~~ R > 0.
\end{align*}
When $x_0 = 0$, we will often write $\tail(u;R,0) = \tail(u;R)$.\\
Given a kernel $K : \R^n \to [0,\infty]$ satisfying \eqref{eq:Kcomp} and a set $\cD \subset \R^n \times \R^n$, we introduce the notation
\begin{align*}
\cE_{\cD}(u,v) = \iint_{\cD} (u(x) - u(y))(v(x) - v(y)) K(x-y) \d y \d x.
\end{align*}
Moreover, if $\cD = D \times D$ for some $D \subset \R^n$, we write $\cE_{D} := \cE_{\cD}$, and if $\cD = \R^n \times \R^n$, we write $\cE := \cE_{\cD}$.\\
Moreover, given $K$ satisfying \eqref{eq:Kcomp}, $\Omega \subset \R^n$, we denote
\begin{align}
\label{eq:OP}
\cI_{\Omega}(u) := \cE_{(\Omega^c \times \Omega^c)^c}(u,u) + |\{ u > 0 \} \cap \Omega|,
\end{align}
whenever this expression is finite.

We recall the definition of minimizers of $\cI_{\Omega}$, which is general enough to allow for functions that grow like $t \mapsto t^s$ at infinity. Note that such minimizers arise as blow-up limits and are crucial for the study of the free boundary, but do not belong to $V^s(\Omega | \R^n)$.

\begin{definition}[minimizers]
\label{def:minimizer}
Let $K$ satisfy \eqref{eq:Kcomp}. Let $\Omega \Subset \Omega' \subset \R^n$ be an open, bounded domain. We say that $u \in V^s(\Omega | \Omega') \cap L^1_{2s}(\R^n)$ with $u \ge 0$ in $\R^n$ is a (local) minimizer of $\cI_{\Omega}$ (in $\Omega$) if for any $v \in V^s(\Omega | \Omega') \cap L^1_{2s}(\R^n)$ with $u = v$ in $\R^n \setminus \Omega$, it holds
\begin{align*}
\iint\limits_{(\Omega^c \times \Omega^c)^c} \Big[(u(x) - u(y))^2 - (v(x) - v(y))^2 \Big] K(x-y) \d y \d x + \Big[ | \{u > 0 \} \cap \Omega | - | \{v > 0 \} \cap \Omega | \Big] \le 0.
\end{align*}
\end{definition}

Note that, given a jumping kernel $K$, the energy $\cE$ gives rise to an integro-differential operator $L$ given by \eqref{eq:L} via the relation
\begin{align*}
\cE_{(\Omega^c \times \Omega^c)^c}(u,\phi) = (L u , \phi) ~~ \forall \phi \in H^s(\R^n) ~~ \text{ with } u \equiv 0 ~~ \text{ in } \R^n \setminus \Omega.
\end{align*}

Since \eqref{eq:OP} is a variational problem, minimizers will naturally be weak (sub)solutions (see \cite{RoWe24b}, or \autoref{lemma:aux}). Therefore, in \cite{RoWe24b} we have applied energy methods to verify their regularity properties. We recall these results in the following subsection. Note that once these basic regularity properties are established, minimizers of $\cI_{\Omega}$ can be interpreted to satisfy \eqref{eq:first-variation} in the viscosity sense. For more details on this conclusion, we refer the reader to Section \ref{sec:viscosity}.

\subsection{Properties of minimizers}
\label{subsec:aux}

We recall the following properties of minimizers of $\cI_{\Omega}$, which were established in \cite{RoWe24b}. We start with the following elementary result.

\begin{lemma}[see Lemma 4.1 in \cite{RoWe24b}]
\label{lemma:aux}
Assume \eqref{eq:Kcomp}. Let $u$ be a minimizer of $\cI_{\Omega}$. Then, the following properties hold true:
\begin{itemize}
\item[(i)] $L u \le 0$ in $\Omega$ in the weak sense.
\item[(ii)] $u \ge 0$ in $\Omega$.
\item[(iii)] $u \in L^{\infty}_{loc}(\Omega)$.
\item[(iv)]$L u = 0$ in $\Omega \cap \{ u > 0 \}$ in the weak sense.
\end{itemize}
\end{lemma}

We recall the optimal regularity, non-degeneracy, as well as the density estimate of the free boundary. These results will be helpful when proving that minimizers of \eqref{eq:onephase-intro} are viscosity solutions to \eqref{eq:first-variation} (see \autoref{lemma:min-visc}) and to give characterizations of points where the free boundary of minimizers is flat (see \autoref{thm:close-to-halfspace}).

\begin{lemma}[optimal regularity]
\label{thm:or}
Assume \eqref{eq:Kcomp}. Let $u$ be a minimizer of $\cI_{\Omega}$ with $B_2 \subset \Omega$.
Then, $u \in C^s_{loc}(B_2)$, and
\begin{align*}
\Vert u \Vert_{C^s(B_R)} \le C R^{-s} \left( 1 + \dashint_{B_{2R}} u \ d x \right) ~~ \forall R \in (0,1].
\end{align*}
Moreover, if $0 \in \partial \{ u > 0 \}$, then
\begin{align*}
\Vert u \Vert_{L^{\infty}(B_{R})} \le C R^s ~~ \forall R \in (0,1],  \qquad \text{ and } ~~  \Vert u \Vert_{C^s(B_{1})} \le C.
\end{align*}
The constant $C > 0$ depends only on $n$, $s$, $\lambda$, $\Lambda$.
\end{lemma}

\begin{proof}
This result follows directly from (rescaled versions of) \cite[Theorem 1.5, Theorem 4.5, and Lemma 4.7]{RoWe24b}.
\end{proof}

\begin{lemma}[non-degeneracy]
\label{thm:non-degeneracy}
Assume \eqref{eq:Kcomp}. Let $u$ be a minimizer of $\cI_{\Omega}$ with $B_2 \subset \Omega$. Then, it holds for any $x \in B_1$
\begin{align*}
u(x) \ge c \dist( x , \partial\{ u > 0\})^s.
\end{align*}
Moreover, if $0 \in \overline{ \{ u > 0 \} }$, then
\begin{align*}
\Vert u \Vert_{L^{\infty}(B_R)} \ge c R^s ~~ \forall R \in (0,1].
\end{align*}
The constant $c > 0$ depends only on $n$, $s$, $\lambda$, $\Lambda$.
\end{lemma}

\begin{proof}
This result follows directly from \cite[Theorem 4.8,  Lemma 4.9]{RoWe24b}.
\end{proof}

\begin{lemma}[density estimates for the free boundary]
\label{thm:density-est}
Assume \eqref{eq:Kcomp}. Let $u$ be a minimizer of $\cI_{\Omega}$ with $B_2 \subset \Omega$. Then, if $0 \in \partial \{ u > 0 \}$ it holds
\begin{align*}
0 < c_1 \le \frac{|\{ u > 0 \} \cap B_R|}{|B_R|} \le 1 - c_2 < 1 \qquad \forall R \in (0,1],
\end{align*}
where $c_1 > 0$ and $c_2 > 0$ depend only on $n$, $s$, $\lambda$, $\Lambda$.
\end{lemma}

\begin{proof}
This result follows directly from \cite[Theorem 4.11]{RoWe24b}.
\end{proof}

The following energy estimate will be used in the classification of blow-ups in 2D (see \autoref{thm:two-dimensional-classification}).

\begin{lemma}[energy estimate]
\label{lemma:energy-bound}
Assume \eqref{eq:Kcomp}. Let $u$ be a minimizer of $\cI_{\Omega}$ with $B_2 \subset \Omega$. Then, if $0 \in \partial \{ u > 0 \}$ it holds for any $R \in (0,1]$
\begin{align*}
\cE_{B_R \times B_R}(u,u) \le C R^n \qquad \tail(u;R) \le C R^{s}, \qquad \Vert u \Vert_{L^p}^p \le C R^{n+sp} ~~ \forall p \in (0,\infty)
\end{align*}
for some $C > 0$, depending only on $n$, $s$, $\lambda$, $\Lambda$.
\end{lemma}

\begin{proof}
This result follows directly by combination of \cite[Lemma 4.3]{RoWe24b} and \autoref{thm:or}.
\end{proof}

\subsection{Properties of blow-ups}
\label{sec:blow-ups}

In this section we recall several basic properties of blow-up sequences from \cite{RoWe24b}. In particular, we recall that blow-ups are global minimizers of $\cI_{\Omega}$ (see \autoref{cor:blowups}). Moreover, we recall a compactness result for minimizers from \cite{RoWe24b}.

\begin{definition}[blow-ups]
Given a minimizer $u$ of \eqref{eq:OP} in $\Omega$, and $x_0 \in \partial \{ u > 0 \} \cap \Omega$ we define
\begin{align*}
u_{r,x_0}(x) = \frac{u(x_0 + r x)}{r^s} ~~ \forall x \in \R^n, \qquad \forall r > 0.
\end{align*}
The sequence of functions $(u_{r,x_0})_r$ is called blow-up sequence for $u$ at $x_0$. If there exists a sequence $r_k \searrow 0$ such that $u_{r_k,x_0} \to u_{x_0}$, as $k \to \infty$ for a function $u_{x_0} : \R^n \to \R$ locally uniformly, we say that $u_{x_0}$ is a blow-up (limit) of $u$ at $x_0$.\\
If $x_0 = 0$, or if no confusion about the free boundary point $x_0$ can arise, we sometimes write $u_{r} := u_{r,x_0}$.
\end{definition}

In \cite{RoWe24b} we have established the following properties of blow-ups.

\begin{lemma}
\label{cor:blowups}
Assume \eqref{eq:Kcomp}. Let $\Omega \subset \R^n$. Let $u$ be a minimizer of $\cI_{\Omega}$ with $B_2 \subset \Omega$ and $x_0 \in \partial \{ u > 0 \} \cap B_1$. Then, there exists a subsequence $(r_k)_k$ with $r_k \searrow 0$ such that $u_{r_k,x_0} \to u_{x_0}$, as $k \to \infty$, locally uniformly. Moreover, for any such $(r_k)_k$, it holds:
\begin{itemize}
\item[(i)] $u_{x_0}$ is a non-trivial minimizer of $\cI$ in $\R^n$, i.e., $u_{x_0}$ is a minimizer of $\cI_{B_R}$ for any $R > 0$.
\item[(ii)] Up to a subsequence, $u_{r_k,x_0} \to u_{x_0}$ in $H^s(B_R)$, in $L^1_{2s}(\R^n)$, and pointwise a.e. in $B_R$ for any $R > 0$.
\item[(iii)] Up to a subsequence, $\1_{\{ u_{r_k,x_0} > 0 \}} \to \1_{u_{x_0} > 0 }$ strongly in $L^1(B_R)$, and pointwise a.e. in $B_R$ for any $R > 0$.
\item[(iv)] Up to a subsequence, $\overline{\{ u_{r_k,x_0} > 0 \}} \to \overline{\{ u_{x_0} > 0 \}}$ locally in $B_R$ for any $R > 0$ in the Hausdorff-sense.
\end{itemize}
\end{lemma}

The following lemma is a compactness result which immediately follows from the proofs of \cite[Lemma 4.13, Lemma 4.14, Corollary 4.16]{RoWe24b}.

\begin{lemma}
\label{lemma:scaling-blow-up}
Let $\Omega \Subset \Omega' \subset \R^n$. Let $x_0 \in \R^n$ be such that $B_2(x_0) \subset \Omega$, and $R > 0$. Let $(K_k)_k$ be a sequence of kernels $K_k$ satisfying \eqref{eq:Kcomp}. Let $(u^{(k)})_k \subset V^s(\Omega| \Omega') \cap L^1_{2s}(\R^n)$ be minimizers of $\cI_{\Omega}$ with respect to $K_k$ such that $u^{(k)}(x_0) = 0$. Then, there exists a subsequence $(r_k)_k$ with $r_k \searrow 0$, such that $u^{(k)}_{r_k,x_0} \to u_{x_0}$ locally uniformly, in $L^1_{2s}(\R^n)$, and weakly in $H^s(B_R)$ to some $u_{x_0} \in H^s(B_R) \cap L^1_{2s}(\R^n)$ for any $R > 0$. Moreover, $\overline{\{ u^{(k)} > 0 \}} \to \overline{\{ u_{\infty} > 0 \}}$ locally in $B_R$ for any $R > 0$ in the Hausdorff-sense. Moreover, there is a kernel $K_{\infty}$ satisfying \eqref{eq:Kcomp}, such that weakly in the sense of measures $\min\{1,|h|^2\} K_k(h) \d h \to \min\{1,|h|^2\} K_{\infty}(h) \d h$, and $u_{\infty} \in H^s(B_R) \cap L^1_{2s}(\R^n)$ is a minimizer of $\cI_{B_R}$ with respect to $K_{\infty}$ for any $R > 0$.
\end{lemma}

\section{Viscosity solutions to the one-phase problem}
\label{sec:viscosity}

In order to prove fine properties of the free boundary for minimizers to the nonlocal one-phase problem we need to study the behavior of minimizers $u$ to $\mathcal{I}_{B_1}$ at the free boundary. When $L = (-\Delta)^s$ and the free boundary $\partial \{ u > 0 \}$ is $C^{1,\alpha}$ in $B_1$, then in \cite{CRS10,FeRo22}, it was shown by analyzing the first variation of the energy functional $\cI$ that 
\begin{align*}
\frac{u}{d^s} = \Gamma(1+s)^{-1} ~~ \text{ on } \partial \{ u > 0 \} \cap B_1,
\end{align*}
where $d(x) := \dist(x,\partial \{ u > 0 \})$ for $x \in \{ u > 0 \}$, and for $x_0 \in \partial \{ u > 0 \}$, we denote $\frac{u}{d^s}(x_0) = \lim_{\Omega \ni x \to x_0} \frac{u}{d^s}(x)$. Having such information at the free boundary turns out to be crucial in order to study its regularity properties.\\
The goal of this section is threefold: First, we generalize the aforementioned result to general nonlocal operators $L$ (see \autoref{prop:free-bound-cond}). It turns out that due to the anisotropy of these operators, the constant of the free boundary condition at each point will depend on the normal vector of the free boundary at that point as in \eqref{eq:first-variation}. This result leads to a viscosity formulation of the free boundary condition, which remains valid also at non-smooth free boundary points (see \autoref{def:viscosity}).\\
Second, we also prove that minimizers of $\mathcal{I}_{B_1}$ are viscosity solutions to \eqref{eq:first-variation} (see \autoref{lemma:min-visc}), and third, we prove that viscosity solutions are $C^s$ (see \autoref{lemma:visc-Cs}).

\subsection{The free boundary condition}

The following is the main result of this subsection.

\begin{proposition}
\label{prop:free-bound-cond}
Assume \eqref{eq:Kcomp}. Let $u$ be a minimizer of $\mathcal{I}_{B_1}$, and assume that $\partial \{ u > 0 \} \cap B_1 \in C^{1,\alpha}$ for some $\alpha \in (0,1)$. Then, $u$ satisfies
\begin{align}
\label{eq:free-bound-cond}
\frac{u}{d^s}(x) = A(\nu_x) ~~ \forall x \in \partial \{ u > 0 \} \cap B_1,
\end{align}
where $\nu_x \in \mathbb{S}^{n-1}$ denotes the normal vector of $\{ u > 0 \}$ at $x$, and $A : \mathbb{S}^{n-1} \to (0,\infty)$ is given by
\begin{align}
\label{eq:A-def}
A(\nu) = c_{n,s} \left( \int_{\mathbb{S}^{n-1}} K(\theta) |\theta \cdot \nu|^{2s} \d \theta \right)^{-\frac{1}{2}}, \qquad \nu \in \mathbb{S}^{n-1},
\end{align}
where $c_{n,s} > 0$ is a constant. Moreover, we have
\begin{align}
\label{eq:A-properties}
c_1 \le A(\nu) \le c_2 ~~ \forall \nu \in \mathbb{S}^{n-1}, \qquad \Vert A \Vert_{C^{1+2s-\eps}(\mathbb{S}^{n-1})} < \infty.
\end{align}
for any $\eps > 0$, where $c_1,c_2 > 0$ depend only on $n,s,\lambda,\Lambda$.
\end{proposition}

\begin{remark}
Note that when $K \in C^{\beta}(\mathbb{S}^{n-1})$, then $\Vert A \Vert_{C^{1+2s-\eps+\beta}} \le C$ for any $\eps > 0$, where $ C> 0$ depends on $n,s,\lambda,\Lambda,\eps$, and $\Vert K \Vert_{C^{\beta}(\mathbb{S}^{n-1})}$.
\end{remark}

The following lemma provides a first-order expansion of an $L$-harmonic function in a $C^{1,\alpha}$ domain, and of the operator $L$ applied to this function outside the domain. Already in this result, the anisotropy of the operator leads to a direction-dependent constant in the expansion:

\begin{lemma}
\label{lemma:expansion-smooth-domain}
Assume \eqref{eq:Kcomp}. Let $\Omega \subset \R^n$ be such that $\partial \Omega \in C^{1,\alpha}$ for some $\alpha \in (0,1)$ with $0 \in \partial \Omega$, and $u$ such that
\begin{align*}
\begin{cases}
L u = 0 ~~ \text{ in } B_1 \cap \Omega,\\
u = 0 ~~ \text{ in } B_1 \setminus \Omega.
\end{cases}
\end{align*}
Then, there exists $U_0 \in \R$ such that
\begin{align*}
u(x) &= U_0 d^s(x) + O(|x|^{s+\alpha}) ~~ \text{ for } x \in \Omega,\\
Lu(x) &= B(\nu) U_0 d^{-s}(x) + O(|x|^{\alpha}) d^{-s}(x) ~~ \text{ for } x \in B_1 \setminus \Omega,
\end{align*}
where $\nu$ is the normal vector of $\partial \Omega$ at $0$, and $B : \mathbb{S}^{n-1} \to (0,\infty)$ is given by $B(\nu) = c_{n,s} A(\nu)^{-2}$ for some $c_{n,s} > 0$.
\end{lemma}

\begin{proof}
First, we observe that for any $e \in \mathbb{S}^{n-1}$, and any (one-dimensional) smooth function satisfying $v(x) = v(x \cdot e)$, we can compute
\begin{align*}
Lv(x) = \left[c_{n,s} \int_{\mathbb{S}^{n-1}} K(\theta) |\theta \cdot e|^{2s} \d \theta \right] (-\Delta)^s_{\R}v(x \cdot e) =: B(e) (-\Delta)^s_{\R}v(x \cdot e).
\end{align*}
Following the arguments of \cite[Lemma 2.6]{FeRo22}, this implies that
\begin{align}
\label{eq:L-exp-1d}
L(x \cdot e)_+^s = B(e) (x \cdot e)_-^{-s}.
\end{align}
We will now use the identity \eqref{eq:L-exp-1d} in the proof of the lemma. First of all, note that by the regularity results of \cite{FeRo23}, we have $u/d^s \in C^{\alpha}(\overline{\Omega} \cap B_{1/2})$, and therefore, setting $U_z = (u/d^s)(z)$ for any $z \in \partial \Omega \cap B_{1/2}$, we obtain
\begin{align*}
u(z + x) &= (u/d^s)(z+x) d^s(z+x) = (U_z + O(|x|^{\alpha})) d^s(z+x) \\
&= U_z d^s(z+x) + O(|x|^{\alpha}) d^s(z+x) =  U_z d^s(z+x) + O(|x|^{s+\alpha}),
\end{align*}
as desired. In particular, we obtain
\begin{align*}
u(z + x) = U_z (x \cdot \nu_z)_+^s + O(|x|^{s+\alpha}),
\end{align*}
where $\nu_z \in \mathbb{S}^{n-1}$ is the normal vector to $\partial \Omega$ at $z$.
Thus, using \eqref{eq:L-exp-1d} for $x = -t\nu_z \in \R^n \setminus \Omega$ with $t > 0$, we obtain
\begin{align*}
L u(z + x) = U_z L(x \cdot \nu_z)_+^s + O(t^{-s+\alpha}) = U_z B(\nu_z) t^{-s} + O(t^{-s+\alpha}).
\end{align*}
Note that, since $\partial \Omega \in C^{1,\alpha}$, for any $x \in B_{1/2} \setminus \Omega$, we can find $z \in \partial \Omega \cap B_{1/2}$ such that  $d(x) = |x-z|$. Therefore, we can rewrite
\begin{align*}
Lu(x) = U_z B(\nu_z) d^{-s}(x) + O(d^{-s+\alpha}(x)) =  U_0 B(\nu) d^{-s}(x) + O(|x|^{\alpha}) d^{-s}(x),
\end{align*}
where we also used $U_z B(\nu_z) = U_0 B(\nu) + O(|z|^{\alpha}) = U_0 B(\nu) + O(|x|^{\alpha})$, which follows from $(z \mapsto U_z) \in C^{\alpha}$, $(z \mapsto \nu_z) \in C^{\alpha}$, and $(z \mapsto B_{\nu_z}) \in C^{\alpha}$. The latter regularity result relies on the fact that $(e \mapsto B(e)) \in C^{1+2s-\eps}(\mathbb{S}^{n-1})$ for any $\eps > 0$. Indeed, we have 
\begin{align*}
\left( e \mapsto B(e) := \int_{\mathbb{S}^{n-1}} K(\theta)|\theta \cdot e|^{2s} \d \theta \right) \in C^{2s+1-\eps}(\mathbb{S}^{n-1}),
\end{align*}
since $K \in L^{\infty}(\mathbb{S}^{n-1})$ and $e \mapsto |\theta \cdot e|^{2s} \in C^{2s}(\mathbb{S}^{n-1})$ and thus $\left[\nu \mapsto |\theta \cdot \nu|^{2s} \right]_{C^{2s+1 -\eps}(e)} \in L^1(\mathbb{S}^{n-1})$ uniformly in $e \in \mathbb{S}^{n-1}$.
\end{proof}

We are now in a position to prove \autoref{prop:free-bound-cond}.

\begin{proof}[Proof of \autoref{prop:free-bound-cond}]
The proof closely follows the one in \cite{FeRo22}, using \autoref{lemma:expansion-smooth-domain}. Let us give a short sketch of how their argument looks like in our setting. As in \cite{FeRo22}, we set $\Omega = \{ u > 0 \}$ and consider competitors of the form $u_{\eps}(x) = u(x + \eps \Psi(x))$ for smooth domain variations $\Psi$ supported in $B_1$. Let $f \in C_c^{\infty}(\partial \Omega)$ be a nonnegative function, supported in $B_1 \cap \partial \Omega$. We introduce
\begin{align*}
\Omega_{\eps} = \{ x \in \Omega : d(x) \ge \eps f(\pi_{\Omega}(z)) \}, 
\end{align*}
where $z = \pi_{\Omega}(z)$ is the unique $z \in \partial \Omega$ such that $d(x) = |x-z|$, and let $v_{\eps}$ be the competitor, defined as the solution to
\begin{align*}
\begin{cases}
L v_{\eps} &= 0 ~~ \text{ in } \Omega_{\eps} \cap B_1,\\
v_{\eps} &= u ~~ \text{ in } \R^n \setminus B_1,\\
v_{\eps} &= 0 ~~ \text{ in } B_1 \setminus \Omega_{\eps}.
\end{cases}
\end{align*}
Moreover, we set $\Theta_{\eps} = (\Omega \setminus \Omega_{\eps}) \cap B_1$, and parametrize the points in $\Theta_{\eps}$ as $z + t \nu_z$, where $z \in \partial \Omega$, $t > 0$, and $\nu_z \in \mathbb{S}^{n-1}$ denotes the inward normal vector of $\partial \Omega$ at $z$. Then, we can expand
\begin{align}
\label{eq:u-exp}
u(z + t \nu_z) = \frac{u}{d^s}(z) t^s + o(t^s), \qquad \text{ where } \qquad \frac{u}{d^s}(z) = \lim_{\tau \to 0} \frac{u(z + t \nu_z)}{\tau^s},
\end{align}
and for $x_0 = z + \eps f(z) \nu_z \in \partial \Omega_{\eps}$:
\begin{align*}
v_{\eps}(x) = \frac{v_{\eps}}{d^s_{\eps}}(x_0) d^s_{\eps}(x) + o(|x-x_0|^s) ~~ \text{ in } \Omega_{\eps}, \qquad \text{ where } \qquad d_{\eps} = \dist(\cdot, \Omega_{\eps}).
\end{align*}
An application of \autoref{lemma:expansion-smooth-domain} to $v_{\eps}$, using that $d_{\eps}(x) = (\eps f(z) - t)(1 + o(\eps))$ for $x = z + t \nu_z \in \Theta_{\eps}$, where $0 < t < \eps f(z)$, yields
\begin{align}
\label{eq:v-eps-exp}
\begin{split}
L v_{\eps}(x) &= \frac{v_{\eps}}{d^s_{\eps}}(x_0) B(\nu_{x_0}) d^{-s}_{\eps}(x) + o(d^{-s}_{\eps}(x)) \\
&= \frac{v_{\eps}}{d^s_{\eps}}(x_0) B(\nu_{x_0}) \left[(\eps f(z) - t)(1 + o(\eps))\right]^{-s} + o(t^{-s}) \\
&= \frac{u}{d^s}(z) B(\nu_{z}) (\eps f(z) - t)^{-s}(1 + o(1)) + o(t^{-s}),
\end{split}
\end{align}
where we used $(v_{\eps}/d_{\eps}^s)(x_0) \to (u/d^s)(z)$, and $\nu_{x_0} \to \nu_z$, as $\eps \to 0$, i.e., that $(v_{\eps}/d_{\eps}^s)(x_0)B(\nu_{x_0}) = (u/d^s)(z)B(\nu_{z}) + o(1)$.\\
Now, by following the same arguments as in \cite{FeRo22} and plugging in \eqref{eq:u-exp} and \eqref{eq:v-eps-exp}, we obtain
\begin{align*}
- \int_{\Theta_{\eps}} u Lv_{\eps} &= \int_{\partial \Omega} \int_0^{\eps f(z)} t^s c_{n,s} \left( B(\nu_z) \left( \frac{u}{d^s}(z) \right)^2 (\eps f(z) - t)^{-s} [1 + o(1)] + o(t^{-s}) \right) \d t \d z\\
&= \eps c_{n,s} \int_{\partial \Omega} f(z) B(\nu_z) \left( \frac{u}{d^s}(z) \right)^2 \d z.
\end{align*}
This implies that for any nonnegative $f \in C_c^{\infty}(B_1 \cap \partial \Omega)$ it holds
\begin{align*}
0 = \lim_{\eps \to 0} \frac{\cI_{B_1}(v_{\eps}) - \cI_{B_1}(u)}{\eps} &= \lim_{\eps \to 0} \left( - \eps^{-1} \int_{\Theta_{\eps}} u Lv_{\eps} - \int_{\partial \Omega} f(z) \d z \right) \\
&= \int_{\partial \Omega} f(z) \left[ c_{n,s} B(\nu_z) \left( \frac{u}{d^s}(z) \right)^2 - 1 \right] \d z,
\end{align*}
which implies \eqref{eq:free-bound-cond}. Finally, \eqref{eq:A-properties} follows immediately from \eqref{eq:Kcomp} and the regularity of $B$ (see the proof of \autoref{lemma:expansion-smooth-domain}).
\end{proof}

\subsection{Viscosity solutions}

Having at hand \autoref{prop:free-bound-cond} (and also \autoref{lemma:aux}), we are now in a position to give a natural notion of viscosity solution to the one-phase free boundary problem $\mathcal{I}_{B_1}$.

Let us first introduce the notion of viscosity solutions to equations of the form $L u = f$.

\begin{definition}[viscosity solutions]
\label{def:viscosity-solution}
Let $\Omega \subset \R^n$ be an open domain. Let $f \in C(\Omega)$.
We say that $u \in C(\Omega) \cap L^1_{2s}(\R^n)$ is a viscosity subsolution to $Lu \le f$ in $\Omega$ if for any $x \in \Omega$ and any neighborhood $N_x \subset \Omega$ of $x$ it holds 
\begin{align}
\label{eq:viscosity-sol}
L \phi(x) \le f(x) ~~ \forall \phi \in C^2(N_x) \cap L^1_{2s}(\R^n) ~~ \text{ s.t. } u(x) = \phi(x), ~~ \phi \ge u.
\end{align}
We say that $u$ is a viscosity supersolution to $Lu \ge f$ in $\Omega$ if \eqref{eq:viscosity-sol} holds true for $-u$ and $-f$ instead of $u$ and $f$. Moreover, $u$ is a viscosity solution to $Lu = f$ in $\Omega$, if it is a viscosity subsolution and a viscosity supersolution.
\end{definition}

\begin{definition}
\label{def:viscosity}
We say that $u \in C(B_1) \cap L^1_{2s}(\R^n)$ with $u \ge 0$ in $\R^n$ is a viscosity solution to the nonlocal one-phase problem \eqref{eq:first-variation} (for $K$) in $B_1$, if
\begin{itemize}
\item[(i)] $u$ is a viscosity solution to $L u = 0$ in $\{ u > 0 \} \cap B_1$, and a viscosity subsolution to $Lu \le 0$ in $B_1$ in the sense of \autoref{def:viscosity-solution}.
\item[(ii)] For any $x_0 \in \partial \{ u > 0 \} \cap B_1$ and any  function $\phi^{1/s} \in C^{\infty}(B_1)$ that satisfies $0 \le \phi := (\phi^{1/s})_+^s \in L^1_{2s}(\R^n)$, $\phi \le (\ge) u$ in $\R^n$, and $\phi(x_0) = u(x_0)$, it holds
\begin{align*}
|\nabla \phi^{1/s}(x_0)|^s \le (\ge) A\left(\frac{\nabla \phi^{1/s}(x_0)}{|\nabla \phi^{1/s}(x_0)|}\right).
\end{align*}
\end{itemize}
\end{definition}

\begin{lemma}
\label{lemma:min-visc}
Assume \eqref{eq:Kcomp}. Let $u$ be a minimizer of $\mathcal{I}_{B_1}$ in $B_1$. Then, $u$ is a viscosity solution to the nonlocal one-phase problem \eqref{eq:first-variation} in $B_1$ in the sense of \autoref{def:viscosity}.
\end{lemma}

The following lemma is the main technical ingredient in the proof of \autoref{lemma:min-visc}:

\begin{lemma}
\label{lemma:one-sided-expansion}
Assume \eqref{eq:Kcomp}. Let $u \ge 0$ be such that $u \in C^s(B_1)$ and $u(0) = 0$, and $u \ge C (x_n)_+^s$ (or $u \le C (x_n)_+^s$) for some $C > 0$. Moreover, assume that $L u \le 0$ in $B_1$, and $L u = 0$ in $\{ u > 0 \} \cap B_{1}$. Then, for $x \in \{ x_n \ge 0 \}$ near $0$ it holds:
\begin{align*}
u(x) = \alpha (x_n)_+^s + o(|x|^s)
\end{align*}
for some $\alpha \ge 0$.
\end{lemma}

Note that in case $u \ge C (x_n)_+^s$, we clearly have $\alpha > 0$.

\begin{proof}
First, we prove the result under the assumption that $u \ge C (x_n)_+^s$. We define
\begin{align*}
\alpha(R) = \sup \{ \alpha > 0 : u \ge \alpha(x_n)_+^s  \text{ in } B_R \}.
\end{align*}
Note that by definition and by assumption, $\alpha(R)$ is decreasing in $R$ and bounded away from zero and from infinity. Thus, there exists $\alpha := \lim_{R \to 0} \alpha(R) = \sup_R \alpha(R) \ge C$, and observe that for any $x \in \{ x_n \ge 0 \}$ near $0$:
\begin{align}
\label{eq:min-visc-help-1}
u(x) \ge \alpha(|x|)(x_n)_+^s \ge \alpha (x_n)_+^s + [\alpha(|x|) - \alpha]|x|^s = \alpha (x_n)_+^s + o(|x|^s).
\end{align}
Next, we claim that for every $\beta > 0$ and every $\delta \in (0,1)$, there exists a radius $r > 0$, such that 
\begin{align}
\label{eq:min-visc-claim}
u(x) \le (\alpha + \delta)(x_n)_+^s ~~ \text{ in } B_r \cap \{ x_n \ge \beta |x'| \}.
\end{align}
Before we prove \eqref{eq:min-visc-claim}, let us assume that \eqref{eq:min-visc-claim} holds true and show how it allows us to conclude the proof. In fact, setting $\beta = \delta = 1/k$ for some $k \in \N$, we deduce from \eqref{eq:min-visc-claim} that for some $r_k > 0$
\begin{align*}
u(x) \le (\alpha + k^{-1})(x_n)_+^s \le \alpha (x_n)_+^s + k^{-1}|x|^s ~~ \text{ in } B_{r_k} \cap \{ k x_n \ge |x'| \}.
\end{align*}
Next, in case $x \in B_{r_k} \cap \{ 0 < k x_n \le |x'| \}$, we find $y \in B_{r_k} \cap \{ kx_n = |x'| \}$ with $|x-y| \le |x|/k$ such that by $u \in C^s(B_1)$, and application of the previous estimate to $y$, we get
\begin{align*}
u(x) \le u(y) + ck^{-s}|x|^s &\le \alpha(y_n)_+^s + k^{-1}|y|^s + ck^{-s}|x|^s \\
&\le c(k^{-s} + k^{-1} + k^{-s})|x|^s ~~ \text{ in } B_{r_k} \cap \{0 < kx_n \le |x'| \}.
\end{align*} 
This yields for any $x \in \{ x_n \ge 0 \}$ close to $0$:
\begin{align*}
u(x) \le \alpha(x_n)_+^s + c k^{-s}|x|^s = \alpha (x_n)_+^s + o(|x|^s),
\end{align*}
and implies the desired result upon combination with \eqref{eq:min-visc-help-1}.

Thus, it remains to prove the claim \eqref{eq:min-visc-claim}. Let us assume by contradiction that there exist $\beta > 0$, $\delta \in (0,1)$, and a sequence $x_k \to 0$ in $\{ x_n \ge 0 \}$ with $(x_k)_n \ge \beta |x'_k|$ such that
\begin{align*}
u(x_k) \ge (\alpha + \delta)((x_k)_n)_+^s.
\end{align*}
Note that by definition of $\alpha$, for every $\tau \in (0,1)$, there exists $r(\tau) > 0$ (depending also on $\delta$) such that
\begin{align*}
u(x) \ge (\alpha - \tau \delta) (x_n)_+^s ~~ \text{ in } B_{r(\tau)} \cap \{ x_n \ge 0 \}.
\end{align*}
Let us now define $v(x) = u(x) - (\alpha - \tau \delta) (x_n)_+^s$, where we will choose $\tau$ later in an appropriate way.
Clearly, by construction, we have $v \ge 0$ in $B_{r(\tau)}$, and moreover, since $u \in C^s(B_1)$, it holds for any $x \in B_{\kappa |x_k|}(x_k)$, and $\kappa$ small enough, depending on $\delta$:
\begin{align}
\label{eq:v-bd-below}
v(x) \ge v(x_k) - c (\kappa |x_k|)^s \ge \delta (1 - \tau)((x_k)_n)_+^s - c (\kappa |x_k|)^s \ge c \delta |x_k|^s - c (\kappa |x_k|)^s \ge c |x_k|^s,
\end{align}
where we have used in the second to last step that $(x_k)_n \ge \beta |x_k'|$, and $c > 0$ depends on $\delta, \beta$, but not on $\tau$ if $\tau < 1/2$ is chosen small enough.\\
Let us now observe that for any $x \in B_{|x_k|} \cap \{ x_n > 0 \}$, for $k$ large enough compared to $r(\tau)$, since $L v \ge 0$ in $B_{|x_k|} \cap \{ x_n > 0 \}$ (recall that $Lu = 0$ in $B_1$ since $u \ge C (x_n)_+^s$), it holds
\begin{align*}
L(v \1_{B_{r(\tau)}})(x) & \ge - L(v \1_{\R^n \setminus B_{r(\tau)}})(x) \ge c  \int_{\R^n \setminus B_{r(\tau)}} v(y) |y|^{-n-2s} \d y \\
& \ge - c(\alpha-\tau\delta) \int_{\R^n \setminus B_{r(\tau)}} (y_n)_+^s |y|^{-n-2s} \d y \\
& \ge - c(\alpha-\tau\delta) r(\tau)^{-s} =: -C_0,
\end{align*} 
where we also used that $u \ge 0$, and $C_0 > 0$ depends on $\alpha,\tau,\delta$.
Let us define $h$ to be the solution to
\begin{align*}
\begin{cases}
L h = -C_0 ~~ \text{ in } B_{|x_k|} \cap \{ x_n > 0 \},\\
h = v \1_{B_{r(\tau)}} ~~ \text{ in } (\R^n \setminus B_{|x_k|}) \cup \{ x_n \le 0\}.
\end{cases}
\end{align*}
Note that by the comparison principle, we have
\begin{align*}
v \ge h ~~ \text{ in } B_{|x_k|} \cap \{ x_n > 0 \}.
\end{align*}
Moreover, note that we can write $h = h_1 + h_2$, where $h_1$ and $h_2$ solve
\begin{align*}
\begin{cases}
L h_1 &= -C_0 ~~ \text{ in } B_{|x_k|} \cap \{ x_n > 0 \},\\
h_1 &= 0 ~~ \text{ in } (\R^n \setminus B_{|x_k|}) \cup \{ x_n \le 0\},
\end{cases}
\qquad
\begin{cases}
L h_2 &= 0 ~~ \text{ in } B_{|x_k|} \cap \{ x_n > 0 \},\\
h_2 &= v \1_{B_{r(\tau)}} ~~ \text{ in } (\R^n \setminus B_{|x_k|}) \cup \{ x_n \le 0\}.
\end{cases}
\end{align*}
We claim that there exist constants $c_1,c_2 > 0$ such that
\begin{align}
\label{eq:h1-est}
h_1(x) &\ge - c_1 C_0 (x_n)_+^s |x_k|^s ~~ \forall x \in B_{|x_k|/2} \cap \{ x_n > 0 \},\\
\label{eq:h2-est}
h_2(x) &\ge c_2 (x_n)_+^s \qquad \qquad ~~ \forall x \in B_{|x_k|/2} \cap \{ x_n > 0 \}.
\end{align}
To see \eqref{eq:h1-est}, let us observe that $\tilde{h}_1(x) := h_1(|x_k|x)|x_k|^{-2s}$ solves $L \tilde{h}_1 = - C_0$ in $B_1 \cap \{ x_n > 0 \}$ with $\tilde{h}_1 \equiv 0$ in $(\R^n \setminus B_1) \cup \{ x_n \le 0 \}$. Thus, by the barrier argument in \cite[Proof of Proposition 2.6.4]{FeRo23},  we have
\begin{align*}
-\tilde{h}_1(x) \le c_1 C_0 (x_n)_+^s ~~ \forall x \in B_{1/2} \cap \{ x_n > 0 \}.
\end{align*}
Thus, \eqref{eq:h1-est} follows by recalling the relation between $\tilde{h}_1$ and $h_1$.\\
To see \eqref{eq:h2-est}, we observe that $\tilde{h}_2(x) := h_2(|x_k|x)|x_k|^{-s}$ solves $L \tilde{h}_2 = 0$ in $B_1 \cap \{ x_n > 0 \}$ with $\tilde{h}_2(x) = |x_k|^{-s} v(|x_k|x)\1_{B_{r(\tau)}}(|x_k|x)$ in $(\R^n \setminus B_1) \cup \{ x_n \le 0 \}$. In particular, since $v\1_{B_{r(\tau)}} \ge 0$ and by \eqref{eq:v-bd-below}, we deduce 
\begin{align*}
\tilde{h}_2 \ge c \1_{B_{\kappa}(x_k/|x_k|)} ~~ \text{ in } (\R^n \setminus B_1) \cup \{ x_n \le 0 \}.
\end{align*}
 Therefore, by the Hopf lemma (see \cite[Proposition 2.6.6]{FeRo23}), we have
\begin{align*}
\tilde{h}_2(x) \ge c_2 (x_n)_+^s ~~ \forall x \in B_{1/2} \cap \{ x_n > 0 \}
\end{align*}
for a constant $c_2 > 0$ that is independent of $|x_k|$. Thus, \eqref{eq:h2-est} follows by recalling the relation between $\tilde{h}_2$ and $h_2$.

Thus, there exists a number $k_0 \in \N$ such that for any $k \ge k_0$:
\begin{align*}
v \ge h = h_1 + h_2 \ge (c_2 - c_1C_0 |x_k|^s )(x_n)_+^s \ge c_0 (x_n)_+^s ~~ \text{ in } B_{|x_k|/2} \cap \{ x_n > 0 \},
\end{align*}
where $c_0 := c_2/2 > 0$ is independent of $k$ and $\tau$. By the definition of $v$, this implies
\begin{align*}
u \ge (\alpha + (c_0-\tau\delta))(x_n)_+^s ~~ \text{ in } B_{|x_k|/2}.
\end{align*}
Thus, choosing first $\tau < 1$ so small, depending on $\delta$, such that $c_0 - \tau\delta > 0$, and then $k \in \N$ so large that the previous argument goes through, we obtain a contradiction with the definition of $\alpha$. This proves \eqref{eq:min-visc-claim}, and we conclude the proof.

In case $u \le C (x_n)_+^s$, the proof has to be modified slightly, but follows the same line of arguments. First, one defines 
\begin{align*}
\alpha(R) = \inf \{ \alpha > 0 : u \le \alpha (x_n)_+^s \text{ in } B_R \},
\end{align*}
and observes that $\alpha(R)$ is increasing in $R$ and $\alpha = \lim_{R \to 0} \alpha(R) = \inf_R \alpha(R) \in [0,C]$ exists. As before, it is easy to show that
\begin{align*}
u(x) \le \alpha(x_n)_+^s + o(|x|^s).
\end{align*}
For the lower estimate, instead of \eqref{eq:min-visc-claim}, we claim that for every $\beta > 0$ and every $\delta \in (0,1)$, there exists a radius $r > 0$, such that
\begin{align}
\label{eq:min-visc-claim-2}
u(x) \ge (\alpha - \delta)(x_n)_+^s ~~ \text{ in } B_r \cap \{ x_n \ge \beta |x'| \}.
\end{align}
From here, the desired result follows by the exact same arguments as before, after changing some of the signs. To prove \eqref{eq:min-visc-claim-2}, we argue again by contradiction, assuming that there exist $\beta > 0$ and $\delta \in (0,1)$, and $x_k \to 0$ in $\{ x_n \ge 0\}$ with $(x_k)_n \ge \beta |x_k'|$ such that
\begin{align*}
u(x_k) \le (\alpha - \delta)((x_k)_n)_+^s.
\end{align*}
This time, we define $v(x) = (\alpha + \tau \delta)(x_n)_+^s - u(x)$, and observe that $L v \ge 0$ in $B_{r(\tau)} \cap \{ x_n \ge 0 \}$ (since $Lu \le 0$ in $B_1$), and satisfies $v \ge 0$ in $B_{r(\tau)}$, and \eqref{eq:v-bd-below}, as before. Moreover, we have the following estimate for $x \in B_{|x_k|} \cap \{ x_n \ge 0 \}$:
\begin{align*}
L(v \1_{B_{r(\tau)}})(x) &\ge -L(v \1_{\R^n \setminus B_{r(v)}}) \ge c \int_{\R^n \setminus B_{r(\tau)}} v(y) |y|^{-n-2s} \d y \\
&\ge c(\alpha + \tau \delta) r(\tau)^{-s} - c\int_{\R^n \setminus r(\tau)} u(y) |y|^{-n-2s} \d y \\
&\ge c(\alpha+ \tau \delta - c) r(\tau)^{-s} =: -C_0,
\end{align*}
where we used that $u \le C(x_n)_+^s$ in the last step, and $C_0 > 0$ is a constant. From here, the proof follows as in the first case, defining $h = h_1 + h_2$.
\end{proof}

We are now in a position to give the proof of \autoref{lemma:min-visc}:

\begin{proof}[Proof of \autoref{lemma:min-visc}]
We have already shown that $L u = 0$ in $\{ u > 0\} \cap B_1$ and $L u \le 0$ in $B_1$ in the weak sense (see \autoref{lemma:aux}). This implies (i) by \cite[Lemma 2.2.32, 3.4.13]{FeRo23} (see also \cite[Lemma 2.7]{RoWe23}).\\
Let us explain how to prove (ii). Let $0 \in \partial \{ u > 0 \} \cap B_1$ and $\phi^{1/s}$ be as in (ii), with $\phi \le u$. We will not explain the proof in case $\phi \ge u$, since it goes by the same arguments. Note that $\phi$ is always nonnegative. First, we consider the blow-up sequences $u_r := u(rx)/r^s$, and $\phi_r := \phi(rx)/r^s = [\phi^{1/s}(rx)/r]_+^{s}$, and observe that $u_r \to u_0$ by \autoref{cor:blowups}, where $u_0$ is a global minimizer of $\mathcal{I}$. Moreover, since $\phi^{1/s}$ is smooth, $\phi_r \to \phi_0$, where $\phi_0(x) = (\nabla \phi^{1/s}(0) \cdot x)_+^s$. Let us assume without loss of generality that $\nabla \phi^{1/s}(0) / |\nabla \phi^{1/s}(0)| = e_n$.\\
Next, we apply \autoref{lemma:one-sided-expansion} to $u_0$, which yields the existence of $\alpha \ge 0$ such that for $x \in \{ x_n \ge 0 \}$ near $0$ it holds:
\begin{align}
\label{eq:u_0-asymp}
u_0(x) = \alpha(x_n)_+^s + o(|x|^s).
\end{align} 
Clearly $\alpha \not= 0$ since $\phi_0 \le u_0$.
Note that $u_0 \in C^s(\R^n)$ as a global minimizer (see \autoref{thm:or}), and that $|\nabla \phi^{1/s}(0)|^s(x_n)_+^s = \phi_0(x) \le u_0(x)$, which is why \autoref{lemma:one-sided-expansion} is applicable to $u_0$.\\
In particular, this implies
\begin{align}
\label{eq:vector-order}
|\nabla \phi^{1/s}(0)|^s \le \alpha.
\end{align}
Next, we blow up $u_0$ again, i.e., we take $v(x) = \lim_{r \to 0} u_0(rx)/r^s$. By \eqref{eq:u_0-asymp}, we have for $x \in \{ x_n \ge 0\}$:
\begin{align}
\label{eq:v-asymp-1}
v(x) = \alpha (x_n)_+^s.
\end{align} 
Since $v$ is again a minimizer by \autoref{cor:blowups}, it holds $v \in C^s(\R^n)$ by \autoref{thm:or}. Since $v = 0$ in $\{ x_n = 0\}$, there exists $C > 0$ such that $v \le C (x_n)_-^s$ in $\{ x_n \le 0 \}$. Thus, we can apply \autoref{lemma:one-sided-expansion} to $v \1_{\{ x_n \le 0\}}$ in $\{ x_n \le 0\}$ (after a rotation, replacing $(x_n)_+$ by $(x_n)_-$ in  \autoref{lemma:one-sided-expansion}) and deduce that for $x \in \{ x_n \le 0 \}$ near $0$ it holds:
\begin{align}
\label{eq:v-asymp-2}
v(x) = \beta (x_n)_-^s + o(|x|^s). 
\end{align}
If $\beta \not= 0$, then \eqref{eq:v-asymp-1} and \eqref{eq:v-asymp-2} imply $|\{ v = 0 \} \cap B_r|/|B_r| \to 0$ as $r \searrow 0$, which contradicts the measure density estim ates for minimizers (see \autoref{thm:density-est}). Thus, we conclude $\beta = 0$.\\
Therefore, blowing up $v$ again, i.e., defining $w(x) := \lim_{r \to 0} v(rx)/r^s$, we deduce from \eqref{eq:v-asymp-1} and \eqref{eq:v-asymp-2}, using $\beta = 0$, 
\begin{align*}
w(x) = \alpha (x_n)_+^s ~~ \text{ in } \R^n.
\end{align*}
Since $\partial \{ w > 0 \} = \{ x_n = 0 \} \in C^{1,\alpha}$, we can apply \autoref{prop:free-bound-cond} and obtain $\alpha = A(e_n)$. Due to \eqref{eq:vector-order}, this proves the desired result.
\end{proof}

\subsection{Optimal regularity for viscosity solutions}

We end this section by proving that viscosity solutions are $C^s$ regular.

\begin{lemma}
\label{lemma:visc-Cs}
Assume \eqref{eq:Kcomp}. Let $u \in C(B_1) \cap L^1_{2s}(\R^n)$ be a viscosity solution to the nonlocal one-phase problem in $B_1$ and $0 \in \partial \{ u > 0 \}$. Then, $u \in C^s_{loc}(B_1)$, and it holds
\begin{align*}
\Vert u \Vert_{C^s(B_{1/2})} \le C \left(1 + \dashint_{B_1} u \d x \right)
\end{align*}
for some constant $C > 0$, depending only on $n,s,\lambda,\Lambda$. Moreover, if  $0 \in \partial \{ u > 0 \}$, then it holds
\begin{align*}
\Vert u \Vert_{C^s(B_{1/2})} \le C.
\end{align*}
\end{lemma}

The proof of \autoref{lemma:visc-Cs} requires the following lemma, which is closely related to \autoref{lemma:one-sided-expansion}.

\begin{lemma}
\label{lemma:interior-ball-blowup-viscosity}
Assume \eqref{eq:Kcomp}. Let $u \ge 0$ be such that $u \in C^s(B_1)$ and $L u \le 0$ in $B_1$, and $L u = 0$ in $\{ u > 0 \} \cap B_2$. Moreover, assume that there exists a ball $B \subset \{ u > 0 \}$ with $\overline{B} \cap \partial\{ u > 0 \} = \{ 0 \}$. Then, there exists $\alpha \ge 0$ such that for any $x \in B \cap \{d_B(x) \ge |x|/2 \}$ (non-tangential region inside $B$) near $0$ it holds
\begin{align*}
u(x) = \alpha (x \cdot \nu)_+^s + o(|x|^s),
\end{align*}
where $\nu \in \mathbb{S}^{n-1}$ denotes the normal vector of $\partial B$ at zero, inward to $\{ u > 0 \}$.
\end{lemma}

\begin{proof}
Without loss of generality, we assume that $\nu = e_n$. The proof follows closely the arguments in the proof of \autoref{lemma:one-sided-expansion}. We start by defining
\begin{align*}
\alpha(R) = \sup \left\{ \alpha > 0 : u \ge \alpha (x_n)_+^s ~~ \text{ in } B \cap B_R \cap \{ d_B(x) \ge |x|/2 \} \right\}.
\end{align*}
As in the proof of \autoref{lemma:one-sided-expansion}, one can show that there exists $\alpha := \lim_{R \to 0} \alpha(R) = \sup_{R} \alpha(R) \ge 0$ and that for any $x \in B \cap \{ d_B(x) \ge |x|/2 \}$ near $0$ it holds
\begin{align*}
u(x) \ge \alpha (x_n)_+^s + o(|x|^s).
\end{align*}
Moreover, as in the proof of \autoref{lemma:one-sided-expansion}, the claim \eqref{eq:claim-interior-ball} follows once we show that for any $\beta > 0$ and $\delta \in (0,1)$, there exists a radius $r > 0$ such that 
\begin{align}
\label{eq:claim-interior-ball-2}
u(x) \le (\alpha + \delta)(x_n)_+^s ~~ \text{ in } B_r \cap \{ x_n \ge \beta |x'| \}.
\end{align}
By contradiction, we assume that there exist $\beta > 0$, $\delta \in (0,1)$, and $x_k \to 0$ in $B$ with $(x_k)_n \ge \beta |x_k'|$ such that
\begin{align*}
u(x_k) \ge (\alpha + \delta) ((x_k)_n)_+^s.
\end{align*} 
Note that if $k$ is large enough, then $x_k \in B$. We define $v(x) = u(x) - (\alpha - \tau \delta)(x_n)_+^s$ for $\tau > 0$ to be chosen later. Proceeding as in \autoref{lemma:one-sided-expansion}, but replacing $\{ x_n > 0 \}$ by $B$, we can show that there is $k_0 \in \N$ such that for any $k \ge k_0$:
\begin{align}
\label{eq:claim-interior-ball-2-help}
v \ge h \ge c_0 d_B^s(x) ~~ \text{ in } B \cap B_{|x_k|/2},
\end{align}
where $c_0 > 0$ is independent of $k$ and $\tau$. Indeed, as in the proof of \autoref{lemma:one-sided-expansion} we have $L v \ge 0$ in $B_{|x_k|} \cap B$, where we use that $L u = 0$ in $B$ since $B \subset \{ u > 0 \}$. Moreover, the barrier argument from \cite[Proof of Proposition 2.6.4]{FeRo23} and the Hopf lemma (see \cite[Proposition 2.6.6]{FeRo23}) remain true in this setting since $B$ is a smooth domain and $0 \in \partial B$. In particular, \eqref{eq:claim-interior-ball-2-help} implies
\begin{align*}
v \ge c_0|x|^s/2^s  \ge c_1 (x_n)_+^s ~~ \text{ in } B \cap B_{|x_k|/2} \cap \{ d_B(x) \ge |x|/2 \},
\end{align*}
for some $c_1 > 0$ (since $|x| \ge x_n \ge 0$). Thus, by the definition of $v$, we get
\begin{align*}
v \ge (\alpha + (c_1 - \tau \delta))(x_n)_+^s ~~ \text{ in } B \cap B_{|x_k|/2} \cap \{ d_B(x) \ge |x|/2 \}.
\end{align*}
This yields a contradiction with the definition of $\alpha$ upon choosing $\tau < 1$ small, and $k \in \N$ large enough. This establishes \eqref{eq:claim-interior-ball-2}, and therefore \eqref{eq:claim-interior-ball}. The proof is complete.
\end{proof}

\begin{proof}[Proof of \autoref{lemma:visc-Cs}]
We claim that for any $x \in B_{1/2}$ it holds
\begin{align}
\label{eq:visc-Cs-claim-1}
|u(x)| \le c \dist(x, \partial \{ u > 0 \})^s.
\end{align}
From here, the claims follow immediately by using the interior regularity theory (see \cite{FeRo23}) in the same way as in the proofs of \cite[Theorem 4.5, first part of Theorem 1.5]{RoWe24b}.
To see \eqref{eq:visc-Cs-claim-1}, let us assume without loss of generality that $x := e_n/2$ and $\dist(e_n/2 , \partial \{ u > 0 \}) = 1/2 = |e_n/2|$, and $0 \in \partial \{ u > 0 \}$. We claim that 
\begin{align}
\label{eq:visc-Cs-claim-2}
|u(e_n/2)| \le C
\end{align}
for some constant $C > 0$, depending only on $n,s,\lambda,\Lambda$. Indeed, if \eqref{eq:visc-Cs-claim-2} holds true, then by scaling, shifting, and rotating we immediately deduce \eqref{eq:visc-Cs-claim-1}. To prove \eqref{eq:visc-Cs-claim-2}, we define $w$ to be the solution to
\begin{align*}
\begin{cases}
L w &= 0 ~~ \text{ in } B_{1/2}(e_n/2) \cap B_{3/4},\\
w &= 0 ~~ \text{ in } \R^n \setminus B_{1/2}(e_n),\\
w &= u ~~ \text{ in } B_{1/2}(e_n/2) \setminus B_{3/4}.
\end{cases}
\end{align*}
By the comparison principle, we have
\begin{align*}
u \ge w ~~ \text{ in } \R^n.
\end{align*}
Moreover, by the Hopf lemma (see \cite[Proposition 2.6.6]{FeRo23}), we have
\begin{align*}
w \ge c w(e_n/2)(x_n)_+^s ~~ \text{ in } B_{1/2}(e_n/2) \cap B_{3/4} \cap \{ d_{B_{1/2}(e_n/2) \cap B_{3/4}} \ge (x_n)_+/2 \},
\end{align*}
where we used that the constant in \cite[Proposition 2.6.6]{FeRo23} depends on $\inf_{\{d \ge \delta \}} w$, which we can estimate from below by $w(e_n/2)$ due to the Harnack inequality. Hence, using that $B_{1/2}(e_n/2)$ is an interior tangent ball for $\{ u > 0 \}$, we can apply \autoref{lemma:interior-ball-blowup-viscosity}, which implies that upon taking the limit $x \to 0$ in $\{ d_{B_{1/2}(e_n/2) \cap B_{3/4}} \ge (x_n)_+/2 \}$, we have
\begin{align*}
u(x)/(x_n)_+^s \le \alpha
\end{align*}
for some $\alpha > 0$, depending only on $n,s,\lambda,\Lambda$. Hence, we have shown that altogether 
\begin{align*}
w(e_n/2) \le \frac{\alpha}{c}.
\end{align*}
It remains to estimate $w(e_n/2)$ by $u(e_n/2)$. To do so, an application of Harnack's inequality at the boundary (see \cite[Theorem 3.4]{KiLe23}) yields 
\begin{align*}
w(e_n/2) \ge c \inf_{B_{1/3}(e_n/2) \setminus B_{3/4}} w.
\end{align*}
Since $u = w$ in $B_{1/3}(e_n/2) \setminus B_{3/4}$ by construction, and $\{ u > 0 \}$ in $B_{1/2}(e_n/2)$ by assumption, we can apply the interior Harnack inequality for $u$ to deduce that
\begin{align*}
w(e_n/2) \ge c \inf_{B_{1/3}(e_n/2) \setminus B_{3/4}} u \ge c u(e_n/2).
\end{align*}
Altogether, we have proved \eqref{eq:visc-Cs-claim-2}, and the proof is complete.
\end{proof}

\section{Flatness implies $C^{1,\alpha}$ for viscosity solutions}
\label{sec:flatness-implies-C1alpha}

The goal of this section is to prove \autoref{thm:free-boundary-regularity}, namely that the free boundary of a viscosity solution to the nonlocal one-phase problem in the sense of \autoref{def:viscosity} is $C^{1,\alpha}$ near flat free boundary points. To prove this result, we first develop an improvement of flatness scheme (see \autoref{thm:improvement-of-flatness}), which yields the regularity of the free boundary near flat points after application of an iterative scheme.

\subsection{Improvement of flatness}
\label{subsec:improvement-of-flatness}

In this section, we show the following improvement of flatness result for viscosity solutions:

\begin{theorem}
\label{thm:improvement-of-flatness}
Let $K \in C^{1-2s+\beta}(\mathbb{S}^{n-1})$ for some $\beta > \max\{0,2s-1\}$ and assume \eqref{eq:Kcomp}. Let $u$ be a viscosity solution to the nonlocal one-phase problem for $K$ in $B_2$ with $0 \in \partial \{ u > 0\}$. Then, there are $\eps_0, \delta_0, \rho_0, C > 0$, depending only on $n, s, \lambda, \Lambda$, and $\Vert A \Vert_{C^{1+\beta}(\mathbb{S}^{n-1})}$, such that for any $\eps \in (0,\eps_0)$ it holds: If
\begin{align}
\label{eq:eps-flat-iof}
A(e_n)(x \cdot e_n -\eps)_+^s \le u(x) \le A(e_n)(x \cdot e_n + \eps)_+^s ~~ \forall x \in B_1,
\end{align}
and
\begin{align}
\label{eq:tail-smallness-iof}
T_{\eps} := \tail([u-A(e_n)(x \cdot e_n -\eps)_+^s]_- ; 1) + \tail([A(e_n)(x \cdot e_n + \eps)_+^s - u]_- ; 1) \le \eps \delta_0,
\end{align}
then we have for some $\nu \in \mathbb{S}^{n-1}$ with $|\nu - e_n| < C \eps$:
\begin{align*}
A(\nu)\left(x \cdot \nu - \frac{\eps}{2}\right)_+^s \le u_{\rho_0}(x) \le A(\nu)\left(x \cdot \nu + \frac{\eps}{2}\right)_+^s ~~ \forall x \in B_{1},
\end{align*}
and 
\begin{align*}
T_{\rho_0,\frac{\eps}{2}} := \tail\left(\left[u_{\rho_0}-A(\nu)\left(x \cdot \nu - \frac{\eps}{2}\right)_+^s\right]_- ; 1\right) + \tail\left(\left[A(\nu)\left(x \cdot \nu + \frac{\eps}{2}\right)_+^s - u_{\rho_0}\right]_- ; 1\right) \le \frac{\eps}{2} \delta_0.
\end{align*}
\end{theorem}

\autoref{thm:improvement-of-flatness} is the central ingredient in the proof of our main result on the regularity of the free boundary (see \autoref{thm:main-C1alpha}). It establishes an iteration scheme, from which the regularity of the free boundary near all points at which \eqref{eq:eps-flat-iof} and \eqref{eq:tail-smallness-iof} hold true, follows by standard arguments (see Subsection \ref{subsec:iterate-improvement-of-flatness}). The proof of \autoref{thm:improvement-of-flatness} goes by a compactness argument. The convergence of the compactness sequence will follow from a partial boundary Harnack inequality (see Subsection \ref{subsubsec:partial-BH}), and the contradiction will follow from the regularity of the so called ``linearized problem'' (see Subsection \ref{subsubsec:reg-linearized} and Subsection \ref{subsubsec:conclusion}), which occurs as the PDE satisfied by the limit of the compactness sequence (see Subsection \ref{subsubsec:linearized-problem}). This regularity was established in our previous work \cite{RoWe24}.

\subsubsection{Partial Boundary Harnack}
\label{subsubsec:partial-BH}

The first step in the proof of \autoref{thm:improvement-of-flatness} is to establish the following (rescaled) partial boundary Harnack inequality:

\begin{lemma}
\label{lemma:partial-BH-scaled}
Assume \eqref{eq:Kcomp}. Let $u$ be a viscosity solution to the nonlocal one-phase problem for $K$ in $B_r$ for some $r \in (0,1]$ with $0 \in \partial \{ u > 0 \}$. Then, there are $\eps_0, c > 0$, and $\theta, \delta_0 \in (0,1)$, depending only on $n,s,\lambda,\Lambda$, such that if $a_0 \le b_0$ are such that
\begin{align}
\label{eq:initial-closeness}
|b_0 - a_0| \le r \eps_0 \qquad \text{ and } \qquad A(e_n)(x_n + a_0)_+^s \le u(x) \le A(e_n)(x_n + b_0)_+^s ~~ \forall x \in B_{r},
\end{align}
and
\begin{align*}
T_r := \tail([u-A(e_n)(x_n + a_0)_+^s]_- ; r) + \tail([A(e_n)(x_n + b_0)_+^s - u]_- ; r) \le r^{s-1} |a_0 - b_0| \delta_0,
\end{align*}
then there are $a_0 \le a_1 < b_1 \le b_0$ with
\begin{align*}
|b_1 - a_1| = (1-\theta)|a_0 - b_0| \qquad \text{ and } \qquad A(e_n)(x_n + a_1)_+^s \le u(x) \le A(e_n)(x_n + b_1)_+^s ~~ \forall x \in B_{r/20}
\end{align*}
and, for any $K \ge 20$ we have
\begin{align*}
&\tail\left(\left[u-A(e_n)(x_n + a_1)_+^s\right]_- ; \frac{r}{K} \right) + \tail\left(\left[A(e_n)(x_n + b_1)_+^s - u\right]_- ; \frac{r}{K}\right) \\
&\qquad \le c \theta |a_0 - b_0| \left( \frac{r}{K} \right)^{s-1} + K^{-2s} T_r.
\end{align*}
\end{lemma}

\begin{remark}
Note that by making $K$ larger and $\theta$ smaller in \autoref{lemma:partial-BH-scaled} (which only makes the result weaker), namely by taking $K^{-1-s} \le 1-\theta$, and $\theta \le c\delta_0 (1-\theta)$, it is easy to verify that the tail estimate in the conclusion of \autoref{lemma:partial-BH-scaled} becomes
\begin{align}
\label{eq:partial-BH-K-theta}
\begin{split}
\tail & \left(\left[u-A(e_n)(x_n + a_1)_+^s\right]_- ; \frac{r}{K} \right) \\
&+ \tail\left(\left[A(e_n)(x_n + b_1)_+^s - u\right]_- ; \frac{r}{K}\right) \le \left(\frac{r}{K}\right)^{s-1}(1-\theta)|a_0 - b_0| \delta_0.
\end{split}
\end{align}
This bound will allow us to apply \autoref{lemma:partial-BH-scaled} in an iterative scheme (see \autoref{lemma:Holder-continuity-partial-BH}).
\end{remark}

The following partial boundary Harnack inequality at scale one implies \autoref{lemma:partial-BH-scaled}.

\begin{lemma}
\label{lemma:partial-BH-boundary}
Assume \eqref{eq:Kcomp}. Let $u$ be a viscosity solution to the nonlocal one-phase problem for $K$ in $B_1$ Then, there are $\eps_0  > 0$, $\theta, \delta_0 \in (0,1)$, depending only on $n,s,\lambda,\Lambda$, such that if 
\begin{align*}
A(e_n)(x_n + \sigma)_+^s \le u(x) \le A(e_n)(x_n + \sigma + \eps)_+^s ~~ \forall x \in B_{1}
\end{align*}
for some $\eps \in (0,\eps_0)$ and $\sigma \in \R$ with $|\sigma| < 1/10$, and moreover,
\begin{align*}
\tail([u-A(e_n)(x_n + \sigma)_+^s]_- ; 1) + \tail([A(e_n)(x_n + \sigma + \eps)_+^s - u]_- ; 1) \le \delta_0 \eps,
\end{align*}
then at least one of the following holds true:
\begin{itemize}
\item[(i)] $A(e_n)(x_n + \sigma + \theta \eps)_+^s \le u(x) \le A(e_n)(x_n + \sigma + \eps)_+^s ~~ \forall x \in B_{1/20}$,
\item[(ii)] $A(e_n)(x_n + \sigma)_+^s \le u(x) \le A(e_n)(x_n + \sigma + (1-\theta)\eps)_+^s ~~ \forall x \in B_{1/20}$.
\end{itemize}
\end{lemma}

Before we prove \autoref{lemma:partial-BH-boundary}, let us explain how it implies \autoref{lemma:partial-BH-scaled}.

\begin{proof}[Proof of \autoref{lemma:partial-BH-scaled}]
The first claim in case $r = 1$ follows directly from \autoref{lemma:partial-BH-boundary}. Note that since $0 \in \partial \{ u > 0 \}$, it must be $a_0 \le 0 \le b_0$. Moreover, by choosing $\eps_0 > 0$ small enough, we can assume that $|a_0| < 1/10$.\\
In order to obtain the first claim with general $r \in (0,1)$, note that if $u$ is a viscosity solution to the nonlocal one-phase problem for $K$ in $B_r$ with 
\begin{align*}
A(e_n)(x_n + a_0)_+^s \le u(x) \le A(e_n)(x_n + b_0)_+^s ~~ \forall x \in B_r, \qquad \text{ and } \qquad T_r \le r^{s-1}|a_0 - b_0| \delta_0,
\end{align*}
then $\tilde{u}(x) = u(rx)/r^s$ is a viscosity solution to the nonlocal one-phase problem for $K$ in $B_1$ satisfying
\begin{align*}
A(e_n)(x_n + a_0/r)_+^s &\le \tilde{u}(x) \le A(e_n)(x_n + b_0/r)_+^s ~~ \forall x \in B_1,\\
\tail([\tilde{u} - A(e_n)(x_n + a_0/r)_+^s]_- ; 1 ) &+ \tail([ A(e_n)(x_n + b_0/r)_+^s - \tilde{u}]_- ; 1 ) = r^{-s} T_r \le r^{-1} |a_0 - b_0| \delta_0.
\end{align*}
Thus, we can apply \autoref{lemma:partial-BH-boundary} on scale one if $|b_0 - a_0| \le r \eps_0$, which is exactly what we assumed.

Finally, let us explain how to deduce the second claim, namely the tail estimates on scale $r/M$. For this, let us assume that we are in case (i) of \autoref{lemma:partial-BH-boundary}, i.e., $b_1 = b_0$ and $a_1 = a_0 + \theta |a_0 - b_0|$ (the reasoning in case (ii) goes analogously). Then, the estimate for $\tail([A(e_n)(x_n + b_1)_+^s - u]_- ; r/M)$ follows directly from the assumption. We focus on the estimate for $\tail([u-A(e_n)(x_n + a_1)_+^s]_- ; r/M)$.\\
Note that by the assumption \eqref{eq:initial-closeness} we have
\begin{align*}
[u-A(e_n)(x_n + a_1)_+^s]_- \le A(e_n)(x_n + a_1)_+^s - A(e_n)(x_n + a_0)_+^s ~~\text{ in } B_r.
\end{align*}
Thus, we can estimate
\begin{align*}
\tail([u-A(e_n)(x_n + a_1)_+^s]_- ; r/M) &\le A(e_n) \left( \frac{r}{M} \right)^{2s} \int_{\R^n \setminus B_{r/M}} [(x_n + a_1)_+^s - (x_n + a_0)_+^s] |x|^{-n-2s} \d x \\
&\quad + M^{-2s} \tail([u-A(e_n)(x_n + a_0)_+^s]_- ; r)\\
&= I_1 + I_2.
\end{align*}
While for $I_2$, we get by assumption $I_2 \le M^{-2s} T_r$, for $I_1$ we compute
\begin{align}
\label{eq:L1-exponent-gain}
\begin{split}
I_1 &\le A(e_n)\left( \frac{r}{M} \right)^{2s} \sum_{k = 1}^{\infty} \int_{B_{2^{k+1}(r/M)} \setminus B_{2^k (r/M)}} [(x_n + a_1)_+^s - (x_n + a_0)_+^s] |x|^{-n-2s} \d x\\
&\le c\left( \frac{r}{M} \right)^{-n} \sum_{k = 1}^{\infty} 2^{-k(n+2s)} \int_{B_{2^{k+1}(r/M)} \setminus B_{2^k (r/M)}} [(x_n + a_1)_+^s - (x_n + a_0)_+^s] \d x \\
&\le c \left( \frac{r}{M} \right)^{-1} \sum_{k = 1}^{\infty} 2^{-k(1+2s)} \int_{-2^{k+1}(r/M)}^{2^{k+1}(r/M)} [(x + a_1)_+^s - (x + a_0)_+^s] \d x \\
&\le c \left( \frac{r}{M} \right)^{-1} \sum_{k = 1}^{\infty} 2^{-k(1+2s)} \left(\int_{-a_1}^{2^{k+1}(r/M)} (x + a_1)^s \d x - \int_{-a_0}^{2^{k+1}(r/M)} (x + a_0)^s \d x \right) \\
&\le c \left( \frac{r}{M} \right)^{-1} \sum_{k = 1}^{\infty} 2^{-k(1+2s)} \left(\int_{2^{k+1}(r/M) + a_0}^{2^{k+1}(r/M) + a_1} x^s \d x \right)\\
&\le c \left( \frac{r}{M} \right)^{-1} \sum_{k = 1}^{\infty} 2^{-k(1+2s)} \left((2^{k+1}(r/M) + a_1)^{s+1} - (2^{k+1}(r/M) + a_0)^{s+1} \right)\\
&\le c \left( \frac{r}{M} \right)^{-1} \sum_{k = 1}^{\infty} 2^{-k(1+2s)} |a_1 - a_0| (2^{k+1}(r/M))^{s} \\
&\le c |a_1 - a_0| \left( \frac{r}{M} \right)^{s-1} \sum_{k = 1}^{\infty} 2^{-k(1+s)} = c \theta |a_0 - b_0| \left( \frac{r}{M} \right)^{s-1}.
\end{split}
\end{align}
This yields the second claim and concludes the proof.
\end{proof}

It remains to prove \autoref{lemma:partial-BH-boundary}. The proof is quite involved, and requires the following auxiliary lemma. The following lemma states that perturbations of the function $(x_n)_+^s$, which is harmonic with respect to $L$, are still close to being harmonic with respect to $L$.

\begin{lemma}
\label{lemma:dist-comp}
Assume \eqref{eq:Kcomp}. Let $|\sigma| \le 1/10$, $\delta \in (0,1)$. Let $D \subset \R^n$ be a bounded $C^{1,1}$ domain and $h \in C^{\infty}(\overline{D})$ be such that $c d_{D} \le h \le c^{-1}  d_{D}$ for some $c > 0$. Then, there exists $C > 0$, depending only on $n,s,\lambda,\Lambda,h,D,c$ such that
\begin{align*}
L((x_n + \sigma + \delta h)_+^s) \le C \delta ~~ \text{ in } \overline{D}.
\end{align*} 
\end{lemma}

\begin{proof}
The proof is a modification of \cite[Proposition B.2.1]{FeRo23}. Note that since $d(x) := (x_n + \sigma + \delta h(x))_+ \ge 0$, the claim is trivially satisfied in $D \cap \{ d = 0 \}$. Let us fix $x_0 \in D \cap \{ d > 0 \}$ and define the linearizations of $h$ and $d$ around $x_0$ as follows
\begin{align*}
\tilde{h}(x) = h(x_0) + \nabla h(x_0) \cdot (x - x_0), \qquad l(x) = (x_n + \sigma + \delta\tilde{h}(x))_+.
\end{align*}
Note that $L(l^s)(x_0) = 0$ since $[x \mapsto x_n + \sigma + \delta \tilde{h}(x)]$ is affine linear. Moreover, we claim that
\begin{align}
\label{eq:dist-comp-claim-1}
|d(x_0 + y) - l(x_0 + y)| \le C \delta |y|^2 ~~ y \in \R^n.
\end{align}
To see this, note that since $|a_+ - b_+| \le |a-b|$, we have
\begin{align*}
|d(x_0 + y) - l(x_0 + y)| \le \delta |h(x_0 + y) - \tilde{h}(x_0 + y)|.
\end{align*}
Then, the proof goes on by the exact same arguments as in \cite[Lemma B.2.2]{FeRo23}, since $h$ is a regularized distance with respect to the $C^{1,1}$ domain $D$.\\
As a consequence of \eqref{eq:dist-comp-claim-1} we obtain
\begin{align}
\label{eq:dist-comp-claim-2}
|d^s(x_0 + y) - l^s(x_0 + y)| \le C \delta |y|^2 (d^{s-1}(x_0 + y) + l^{s-1}(x_0 +y)).
\end{align}
Let us now denote $\Omega := \{ d > 0 \}$, observe that $\partial \Omega \in C^{0,1}$ with a Lipschitz radius that is independent of $\delta$, and set $\rho := \min\{ d_{\Omega}(x_0), \rho_0 \}$ for some $\rho_0 > 0$ to be determined later. We observe that there exists $c_1 > 0$, independent of $\delta, \sigma$,  such that
\begin{align*}
c_1 d_{\Omega} \le d \le c_1^{-1} d_{\Omega}.
\end{align*}
Moreover, note that since $d \in C^{0,1}(\overline{\Omega})$, there is $\kappa \in (0,1)$, depending only on the regularity constants of $h$ such that 
\begin{align*}
d(x_0)/2 \le d \le 2 d(x_0) ~~ \text{ in } B_{\kappa \rho}(x_0).
\end{align*}
We will apply \eqref{eq:dist-comp-claim-2} in case $y \in (x_0 - \overline{D}) \setminus B_{\kappa \rho}$. In case $y \in B_{\kappa \rho}$, we observe that we have $\Vert d^{s-1} \Vert_{L^{\infty}(B_{\kappa \rho}(x_0))} \le 2 c_1^{s-1}d_{\Omega}^{s-1}(x_0) \le c_2 \rho^{s-1}$, where $c_2 > 0$ is independent of $\delta, \sigma$. Moreover, by \eqref{eq:dist-comp-claim-1} we have that $l \ge d - c_3 \rho^2 \ge c_4 \rho$ in $B_{\kappa\rho}(x_0)$ for $c_3, c_4 > 0$ independent of $\delta,\sigma$, once $\rho \le \rho_0$ is small enough. Thus, we deduce
\begin{align}
\label{eq:dist-comp-claim-3}
|d^s(x_0 + y) - l^s(x_0 + y)| \le C \delta |y|^2 \rho^{s-1} ~~ \forall y \in B_{\kappa \rho}.
\end{align}
Having at hand \eqref{eq:dist-comp-claim-2}, and \eqref{eq:dist-comp-claim-3}, we estimate
\begin{align*}
|L(d^s)(x_0)| &= |L(d^s - l^s)(x_0)| \le C \int_{\R^n} |d^s(x_0 + y) - l^s(x_0  +y)||y|^{-n-2s} \d y \\
&\le C \delta \left( \rho^{s-1} \int_{B_{\kappa \rho}} \hspace{-0.2cm} |y|^{-n-2s+2} \d y + \int_{ [x_0 - D] \setminus B_{\kappa \rho}} \hspace{-0.4cm} \left( d_{\Omega}^{s-1}(x_0 + y) + l^{s-1}(x_0 + y) \right) |y|^{-n-2s+2} \d y \right) \\
&\le C \delta \left( \rho^{1-s} + 1 \right) \le C \delta,
\end{align*}
where we used that $d^s = l^s$ in $\R^n \setminus \overline{D}$, and applied \cite[Lemma B.2.4]{FeRo23} to estimate the second integral in the last step. Note that $C > 0$ depends only on $n,s,\lambda,\Lambda$, and on $h$ through $\diam(D)$, $\Vert h \Vert_{C^{1,1}(\overline{D})}$, and the $C^{1,1}$ radius of $D$.
\end{proof}

\begin{proof}[Proof of \autoref{lemma:partial-BH-boundary}]
First, note that it suffices to prove the lemma under the following slightly stronger assumption
\begin{align}
\label{eq:partial-BH-boundary-ass}
A(e_n)(x_n + \sigma)_+^s < u(x) < A(e_n)(x_n + \sigma + \eps)_+^s ~~ \forall x \in B_{1} \cap \{ u > 0 \}.
\end{align}
Indeed, if we can deduce the desired result from \eqref{eq:partial-BH-boundary-ass}, then by applying it with $\sigma$ and $\sigma + \eps$ replaced by $\sigma-\kappa$ and $\sigma+\eps+\kappa$ for some $\kappa > 0$, we get the result under the original assumption upon taking the limit $\kappa \to 0$.\\
We take $z = e_n/4 \in \R^n$. Let us first consider the case $u(z) \ge A(e_n)(z_n + \sigma + \eps/2)_+^s$. Then, by the mean value theorem, there exists $\lambda \in (0,1)$ such that
\begin{align*}
A(e_n)(z_n + \sigma + \eps/2)_+^s &= A(e_n)(z_n + \sigma)_+^s + s A(e_n)(z_n + \sigma + \lambda \eps/2)_+^{s-1} \eps/2 \\
&\ge A(e_n)(z_n + \sigma)_+^s + c \eps A(e_n)(z_n + \sigma)_+^{s-1},
\end{align*}
where we applied in the the second step the Harnack inequality to the function $x \mapsto s A(e_n)(x_n + \sigma + \lambda \eps/2)_+^{s-1}$, which is harmonic with respect to $L$ in $\{ x_n > - \sigma - \lambda\eps/2 \}$. Next, we observe that by assumption $L(u-A(e_n)(x_n + \sigma)_+^s) = 0$ in $B_1 \cap \{ x_n > 1/10 \}$ and also $u- A(e_n)(x_n + \sigma)_+^s \ge 0$ in $B_1$. Thus, again by the nonlocal Harnack inequality (see \cite[Theorem 6.9]{Coz17}), and using that by the previous computation it holds
\begin{align*}
u(z) - A(e_n)(z_n + \sigma)_+^s \ge c \eps A(e_n)(z_n + \sigma)_+^{s-1},
\end{align*}
we deduce, using also the assumption on the tails:
\begin{align*}
u &\ge A(e_n)(x_n + \sigma)_+^s + c\eps A(e_n)(z_n + \sigma)_+^{s-1} \\
&\quad - c \tail([u - A(e_n)(x_n + \sigma)_+^s];1) \\
&\ge A(e_n)(x_n + \sigma)_+^s + c\eps \Big( A(e_n)(z_n + \sigma)_+^{s-1} - \delta_0 \Big) \qquad \text{ in } B_{9/10} \cap \{ x_n \ge 1/9\}.
\end{align*}
Next, note that again by Taylor's formula and Harnack applied to $x \mapsto s A(e_n)(x_n + \sigma + \mu \eps)_+^{s-1}$, we deduce that for any $c_0 \in (0,1)$ and $x \in B_{9/10} \cap \{ x_n \ge 1/9\}$ and some $\lambda = \lambda(x) \in (0,1)$
\begin{align*}
A(e_n)(x_n + \sigma + c_0 \eps)_+^s &\le A(e_n)(x_n + \sigma)_+^s + (c_0 \eps) s A(e_n)(x_n + \sigma + \lambda c_0 \eps)_+^{s-1} \\
&\le A(e_n)(x_n + \sigma)_+^s + C(c_0 \eps) s A(e_n)(z_n + \sigma)_+^{s-1}.
\end{align*}
Next, by the definition of $z = e_n/4$, upon choosing $\delta_0, c_0 > 0$ small enough, depending only on $s,c,C$, we can estimate
\begin{align*}
C(c_0 \eps) s A(e_n)(z_n + \sigma)_+^{s-1} < c\eps \Big( A(e_n)(z_n + \sigma)_+^{s-1} - \delta_0 \Big).
\end{align*}
Thus, by combination of the previous three estimates, we deduce
\begin{align}
\label{eq:partial-BH-boundary-help-1}
u > A(e_n) \left(x_n + \sigma + c_0 \eps \right)_+^s ~~ \text{ in } B_{9/10} \cap \{ x_n \ge 1/9 \}.
\end{align}
Next, let us define $\omega = \frac{\nabla A^{1/s}(e_n)}{|\nabla A^{1/s}(e_n)|} \in \mathbb{S}^{n-1}$.
We construct $h \in C^{\infty}(\supp(h))$ with $0 \le h \le 1$, such that $\supp(h)$ is a bounded $C^{1,1}$ domain, and $a_1 d_{\supp(h)} \le h \le a_2 d_{\supp(h)}$ for some $0 < a_1 \le a_2 < \infty$ to be a function with the following properties:
\begin{align*}
\begin{cases}
\partial_{\omega} h &> 0 ~~ \text{ in } \{ u = 0 \} \cap \{ h > 0 \},\\
h &\ge c_1 ~~ \text{ in } B_{1/20},\\
h &= 0 ~~ \text{ in } \R^n \setminus B_{3/4}.
\end{cases}
\end{align*}
Here, the constants $a_1,a_2,c_1 > 0$ can be chosen freely, but depending only on $n,s,\lambda,\Lambda$. Note that this construction is always possible, since by \autoref{prop:free-bound-cond} we have $\omega_n \ge \delta$ for some $\delta > 0$, depending only on $n,s,\lambda,\Lambda$. In fact, if $\omega_n = e_n$, we can just make $h$ radial with respect to $z$ and choose $\supp(h) = B_{1/2}(z)$ to be a ball, using that $\partial_{n} h > 0$ since $\{u = 0\} \subset \{ x_n \le 1/9 \}$. In the general case, we choose $\supp(h)$ to be an appropriate ellipsoid and $h$ to be a regularized distance for $\supp(h)$.\\
Next, we define
\begin{align*}
G := B_{9/10} \cap \{ x_n \le 1/9 \}.
\end{align*}
We also introduce a bump function $\psi \in C_c^{\infty}(B_{1/16}(z))$ with $\psi \equiv 1$ in $B_{1/32}(z)$, $0 \le \psi \le 1$, and $|\nabla \psi| \le 64$.
Moreover, for $t \in [0,1]$, we consider the function
\begin{align*}
\phi_t(x) = A(e_n)\left( x_n + \sigma + \frac{c_0 \eps}{1 + C} [t h(x) + C \psi(x)] \right)_+^s,
\end{align*}
We claim that upon choosing $C > 0$ large enough, we have that
\begin{align}
\label{eq:phi-t-subh}
L \phi_t \le - \overline{c} \eps ~~ \text{ in } G
\end{align}
for some $\overline{c} > 0$, depending only on $n,s,\lambda,\Lambda$. To prove the claim, we compute 
\begin{align*}
L\phi_t(x) = L \Big(\phi_t - A(e_n)\Big( x_n + \sigma + \frac{c_0 \eps}{1 + C} t h \Big)_+^s \Big)(x) + L \Big( A(e_n) \Big( x_n + \sigma + \frac{c_0 \eps}{1 + C} t h \Big)_+^s \Big) (x) = J_1 + J_2.
\end{align*}
Clearly, by \autoref{lemma:dist-comp} we deduce that $J_2 \le c_2 \frac{c_0 \eps t}{1 + C} \le \frac{c_2 c_0 \eps}{1+C}$ for some $c_2 > 0$. Moreover, for $J_1$, we have by definition of $\psi$ and since $x \in \{ \psi \equiv 0\} $
\begin{align*}
J_1 &= - \int_{\supp(\psi)}  \Big[\phi_t - A(e_n) \Big( x_n + \sigma + \frac{c_0 \eps}{1 + C} t h \Big)_+^s \Big] K(x-y) \d y \\
&\le - c_3 \int_{\{\psi \equiv 1\}} \frac{c_0 \eps C}{1 + C} K(x-y) \d y \le - c_4 \eps
\end{align*}
for some $c_3, c_4 > 0$, where we used that $\left( x_n + \sigma + \frac{c_0 \eps}{1 + C} t h \right)_+^s$ is smooth in $\supp(\psi)$ and applied Taylor's formula. Thus, by choosing $C > 1$ large enough, we deduce
\begin{align*}
J_1 + J_2 \le - c_4 \eps + \frac{c_2 c_0 \eps}{1+C} \le - \frac{c_4 \eps}{2}.
\end{align*}
This concludes the proof of the claim \eqref{eq:phi-t-subh} with $\overline{c} = c_4/2$.

Moreover, note that by \eqref{eq:partial-BH-boundary-help-1} it holds for any $t \in [0,1]$:
\begin{align}
\label{eq:partial-BH-boundary-help-2}
\phi_t(x) \le A(e_n)(x_n + \sigma + c_0 \eps)_+^s < u(x) ~~ \forall x \in B_{9/10} \cap \{ x_n \ge 1/9 \}.
\end{align}
Let us define
\begin{align*}
t^{\ast} = \max \left\{ t \in [0,1] : \phi_t \le u ~~ \text{ in } \bar{G} \right\}.
\end{align*}
First, we observe that $t^{\ast} \in [0 , 1]$ exists, since $\phi_0 = (x_n + \sigma)_+^s \le u$ in $\overline{G} \subset B_1$ by assumption.\\
We claim that 
\begin{align}
\label{eq:t-claim}
t^{\ast} = 1.
\end{align}
Let us prove \eqref{eq:t-claim} by contradiction, i.e., assume that $t^{\ast} < 1$. In that case, there would be a point $x_0 \in \overline{G} \cap \overline{ \{ u > 0 \} }$ such that $u(x_0) = \phi_{t^{\ast}}(x_0)$ and $\phi_{t^{\ast}} \le u$ in $\overline{G}$.\\
First, clearly we must have $t^{\ast} > 0$, since otherwise we have a contradiction with the strict bound in \eqref{eq:partial-BH-boundary-ass}. Thus, it must also hold $h(x_0) > 0$, since otherwise $\phi_{t^{\ast}}(x_0) = \phi_{0}(x_0)$.
Consequently, by $\supp(h) \subset \overline{B_{3/4}}$, and by \eqref{eq:partial-BH-boundary-help-2} we have that $x_0 \not \in \partial G$. Moreover, $x_0 \not \in \overline{G} \cap \partial \{ \phi_{t^{\ast}} > 0 \}$, since in that case, $h(x_0) > 0$, which would imply that $\partial_{\omega} h(x_0) > 0$, and therefore
\begin{align*}
\nabla \phi_{t^{\ast}}^{1/s}(x_0) = A^{1/s}(e_n) \nabla \Big((x_0)_n + \sigma + \frac{c_0 \eps t^{\ast}}{1+C} h(x_0) \Big) = A^{1/s}(e_n) \Big(e_n + \frac{c_0 \eps t^{\ast}}{1+C} \nabla h(x_0) \Big).
\end{align*} 
Since by assumption $\omega \cdot \nabla h(x_0) > 0$ and $t^{\ast} > 0$, and since by definition $\omega = \frac{\nabla A^{1/s}(e_n)}{|\nabla A^{1/s}(e_n)|}$ is the normal vector to the graph $\mathbb{S}^{n-1} \ni \nu \mapsto \nu A^{1/s}(\nu)$, we have $\nabla \phi_{t^{\ast}}^{1/s}(x_0) \not \in \{ \nu \kappa : \kappa \le A^{1/s}(\nu)\}$, and therefore:
\begin{align*}
|\nabla \phi_{t^{\ast}}^{1/s}(x_0)| > A(\nabla \phi_{t^{\ast}}^{1/s}(x_0)/|\nabla \phi_{t^{\ast}}^{1/s}(x_0)|)^{1/s}.
\end{align*}
However, since $\phi_{t^{\ast}} \le u$ and $0 = \phi_{t^{\ast}}(x_0) = u(x_0)$, $x_0 \in \partial \{ u > 0 \}$, so $\phi_{t^{\ast}}$ is a test-function for the free boundary condition for $u$ at $x_0$, and therefore, by \autoref{def:viscosity} it must be $|\nabla \phi_{t^{\ast}}^{1/s}(x_0)| \le A(\nabla \phi_{t^{\ast}}^{1/s}(x_0)/|\nabla \phi_{t^{\ast}}^{1/s}(x_0)|)^{1/s}$, a contradiction.\\
Consequently, it only remains to rule out $x_0 \in \overline{G} \cap \{ \phi_t > 0 \}$. To do so, recall that by assumption and due to \eqref{eq:partial-BH-boundary-help-1} we have 
\begin{align}
\label{eq:partial-BH-boundary-help-3}
\phi_{t^{\ast}} \le u ~~ \text{ in } \overline{G} \cup (B_{9/10} \cap \{ x_n \ge 1/9\}) = B_{9/10}.
\end{align}
Moreover, note that by \eqref{eq:phi-t-subh} we have for $x \in \overline{G} \cap \{ u > 0 \} \cap \supp(h) \subset B_{3/4}$:
\begin{align*}
L([\phi_{t^{\ast}} - u]\1_{B_{9/10}})(x) &= L\phi_{t^{\ast}}(x) - Lu(x) - L([\phi_{t^{\ast}} - u] \1_{\R^n \setminus B_{9/10}})(x) \\
&\le - \overline{c} \eps + c \int_{\R^n \setminus B_{9/10}} [\phi_{t^{\ast}} - u](y) |y|^{-n-2s} \d y \\
&\le - \overline{c} \eps + c \tail([A(e_n)(x_n + \sigma + c_0\eps)_+^s - u]_+ ; 9/10),
\end{align*}
where we also used that $L u = 0$ in $\{ u > 0 \}$. Let us estimate further the last summand. We obtain by assumption
\begin{align*}
c \tail([A(e_n)(x_n + \sigma + c_0\eps)_+^s - u]_+ ; 9/10) &\le c \int_{\R^n \setminus B_{9/10}} [(x_n + \sigma + c_0 \eps)_+^s - (x_n + \sigma)_+^s] |x|^{-n-2s} \d x \\
&\quad + c \tail([A(e_n)(x_n + \sigma)_+^s - u]_+ ; 1) \\
&\le c \int_{\R^n \setminus B_{9/10}} c_0\eps_0 |x|^{-n-2s} \d x\\
&\quad + c \tail([ u - A(e_n)(x_n + \sigma)_+^s]_- ; 1) \\
&\le c c_0 \eps + c \delta_0 \eps,
\end{align*}
where we have used a similar computation as in \eqref{eq:L1-exponent-gain} to compute the first integral. Thus, upon choosing $c_0, \delta_0 > 0$ smaller if necessary, we deduce
\begin{align}
\label{eq:partial-BH-boundary-help-4}
L([\phi_{t^{\ast}} - u]\1_{B_{9/10}})(x) \le  - \overline{c} \eps + c c_0 \eps + c \delta_0 \eps \le 0 ~~ \text{ in } \overline{G} \cap \{ u > 0 \} \cap \supp(h).
\end{align}
Thus, if we had $x_0 \in \overline{G} \cap \{ \phi_{t^{\ast}} > 0 \}$, then first, we must have $x_0 \in \{ u > 0 \} \cap \supp(h)$. However, by \eqref{eq:partial-BH-boundary-help-3} and \eqref{eq:partial-BH-boundary-help-4}, we can apply the strong maximum principle (see \cite[Theorem 2.4.15]{FeRo23}) to the function $[\phi_{t^{\ast}} - u]\1_{B_{9/10}}$ and deduce that it must be $\phi_{t^{\ast}} \equiv u $ in the connected component of $\overline{G} \cap \{ u > 0 \} \cap \supp(h)$ around $x_0$. Since we have already shown that $\phi_{t^{\ast}} < u$ on $\partial G \cup (\partial \{ u > 0\} \cap \overline{G} ) \cup (\{ h = 0 \} \cap \overline{G})$, this is a contradiction. Hence, we have shown $t^{\ast} = 1$, as claimed in \eqref{eq:t-claim}.

Using \eqref{eq:t-claim}, we are now in a position to conclude the proof.
In fact, it follows since $h \ge c_1$ in $B_{1/20}$ by construction:
\begin{align*}
u \ge \phi_1 = A(e_n) \Big(x_n + \sigma + \frac{c_0 \eps}{1+C} h \Big)_+^s \ge A(e_n) \Big(x_n + \sigma + \frac{c_0 c_1 \eps}{1+C} \Big)_+^s ~~ \text{ in } B_{1/20}.
\end{align*}
This implies the desired result (i) with $\theta = \frac{c_0 c_1}{1+C}$.

Finally, in case $u(z) \le A(e_n)(z_n + \sigma + \eps/2)_+^s$, we deduce by analogous arguments as in the beginning of the proof:
\begin{align*}
u \le A(e_n)(x_n + \sigma + (1 - c_0) \eps)_+^s ~~ \text{ in } B_{9/10} \cap \{ x_n \ge 1/9 \}.
\end{align*}
Then, the rest of the proof continues by similar arguments, defining 
\begin{align*}
\phi_{t^{\ast}}(x) = A(e_n)\Big(x_n + \sigma + \eps - \frac{c_0 \eps}{1 + C} [t h(x) + C \psi(x)]  \Big)_+^s.
\end{align*}
\end{proof}

\subsubsection{H\"older regularity and identification of the linearized problem}
\label{subsubsec:linearized-problem}

An important technical tool in our proof is the so-called domain variation, which we will define in the following:

\begin{definition}[domain variation] $~$
\vspace{-0.1cm}
\begin{itemize}
\item[(i)] We say that a function $u : \R^n \to [0,\infty)$ is $\eps$-flat in $B_{\rho}(x_0)$ in the $\nu$-direction for some $\eps, \rho > 0$, $x_0 \in \R^n$, and $\nu \in \mathbb{S}^{n-1}$ if
\begin{align*}
A(\nu)(x \cdot \nu - \eps)_+^s \le u(x) \le A(\nu)(x \cdot \nu + \eps)_+^s ~~ \text{ in } B_{\rho}(x_0).
\end{align*}
\item[(ii)] If $u$ is $\eps$-flat in $B_{\rho}(x_0)$ in the $\nu$-direction, then the set $\tilde{u}_{\eps}(x)$ of all $w \in \R$ such that
\begin{align}
\label{eq:dom-variation}
A(\nu)(x \cdot \nu)_+^s = u(x - \eps \nu w)
\end{align}
is non-empty for any $x \in B_{\rho - \eps}(x_0) \cap \{ x \cdot \nu > 0\}$.\\
In this case, we call $\tilde{u}_{\eps}$ the domain variation of $u$ in $B_{\rho-\eps}(x_0)$ in the $\nu$-direction.\\
If $\{ u = 0 \} \cap \overline{B_{\rho}(x_0)}$ is closed, we extend $\tilde{u}_{\eps}(x)$ to $x \in B_{\rho - \eps}(x_0) \cap \{ x \cdot \nu = 0\}$, taking 
\begin{align*}
\tilde{u}_{\eps}(x) = \max \{ w \in [-1,1] : u(x-\eps \nu w) = 0 \}.
\end{align*}
\end{itemize}
\end{definition}

Note that domain variations also play an important role in the proof of improvement of flatness results for the thin one-phase problem (see \cite{DeRo12}, \cite{DeSa12}, \cite{DSS14}).

\begin{remark}
Note that any $w \in \R$ satisfying \eqref{eq:dom-variation} satisfies $w \in [-1,1]$. Moreover, if $u$ is strictly monotone in the $\nu$-direction in $B_{\rho}(x_0) \cap \{ u > 0 \}$, then there exists a unique $w = \tilde{u}_{\eps}(x)$ with the property \eqref{eq:dom-variation}.
\end{remark}

We will need the following elementary lemma on domain variations:

\begin{lemma}
\label{lemma:comp-dom-var}
Let $u,v \in C(B_{\rho})$ for some $\rho > 0$ and assume that $u$ is $\eps$-flat in $B_{\rho}$ in the $e_n$-direction for some $\eps \in (0,\rho)$.
Moreover, assume that $v$ is strictly increasing in the $e_n$-direction. Then, the following hold true:
\begin{itemize}
\item[(i)] If $\tilde{v}_{\eps}$ is defined in $B_{\rho - \eps} \cap \{ x_n > 0 \}$, then
\begin{align*}
v \le u ~~ \text{ in } B_{\rho} \quad \Rightarrow \quad \tilde{v}_{\eps} \le \tilde{u}_{\eps} \text{ in } B_{\rho - \eps} \cap \{ x_n > 0 \}.
\end{align*} 
\item[(ii)] If $\tilde{v}_{\eps}$ is defined in $B_{r} \cap \{ x_n > 0 \}$ for some $r \in (0,\rho)$, then
\begin{align*}
\tilde{v}_{\eps} \le \tilde{u}_{\eps} ~~ \text{ in } B_{r} \quad \Rightarrow \quad v \le u \text{ in } B_{r- \eps}.
\end{align*} 
\end{itemize} 
\end{lemma}

\begin{proof}
The proof is elementary and goes along the lines of \cite[Lemma 3.1]{DeRo12}.
\end{proof}

With the aid of the partial boundary Harnack inequality (see \autoref{lemma:partial-BH-scaled}), we can prove the following estimate on the oscillation of $\tilde{u}_{\eps}$:

\begin{lemma}
\label{lemma:Holder-continuity-partial-BH}
Assume \eqref{eq:Kcomp}. Let $u$ be a viscosity solution to the nonlocal one-phase problem for $K$ in $B_2$. Then, there are $\eps_0, \delta_0 \in (0,1)$, and $\theta \in (0,1)$, and $M \ge 20$, such that if $0 < \eps \le \eps_0/2$, and $m_0  \in \N$ are such that the following hold true
\begin{align}
\label{eq:a0b0-control-iteration}
2\eps (1-\theta)^{m_0} M^{m_0} \le \eps_0, \qquad A(e_n)(x_n - \eps )_+^s \le u(x) \le A(e_n)(x_n + \eps)_+^s ~~ \text{ in } B_1,
\end{align}
and
\begin{align}
\label{eq:tail-smallness-iteration}
T := \tail([u-A(e_n)(x_n -\eps)_+^s]_- ; 1) + \tail([A(e_n)(x_n + \eps)_+^s - u]_- ; 1) \le (2\eps) \delta_0,
\end{align}
then the set
\begin{align*}
\Gamma_{\eps} := \left\{ (x,\tilde{u}_{\eps}(x)) : x \in B_{1-\eps} \cap \{ x_n \ge 0 \} \right\}  \cap (B_{3/4} \times [-1,1] )
\end{align*}
is above the graph of a function $x \mapsto a_{\eps}(x)$ and below  the graph of a function $x \mapsto b_{\eps}(x)$ with
\begin{align*}
b_{\eps} - a_{\eps} \le 2 (1-\theta)^{m} ~~ \text{ in } B_{\frac{1}{2} M^{-m}}, ~~ m \le m_0.
\end{align*}
The constants $\eps_0,\delta_0,\theta,M$ depend only on $n,s,\lambda,\Lambda$.\\
Moreover, the functions $a_{\eps}, b_{\eps}$ have a modulus of continuity bounded by $t \mapsto \alpha t^{\beta}$ for some $\alpha,\beta > 0$, depending only on $\theta$. Furthermore, for any $m \in \N$ with $m \le m_0$ it holds
\begin{align*}
\tail([u-A(e_n)(x_n - \eps)_+^s]_- ; M^{-m}) + \tail([A(e_n)(x_n + \eps)_+^s - u]_- ; M^{-m}) \le 2\eps \delta_0 M^{m(1-s)}(1-\theta)^m.
\end{align*}
\end{lemma}

\begin{proof}
By assumption, we can apply the partial boundary Harnack inequality (see \autoref{lemma:partial-BH-scaled}) several times. In each step, we obtain:
\begin{align*}
A(e_n)(x_n + a_m)_+^s \le u(x) \le A(e_n)(x_n + b_m)_+^s ~~ \text{ in } B_{M^{-m}},
\end{align*}
where $-\eps \le a_1 \le \dots \le a_m \le 0 \le b_m \le \dots \le b_1 \le \eps$, and $b_m - a_m \le 2 \eps (1-\theta)^m$. Moreover, note that, according to \eqref{eq:partial-BH-K-theta}, we can choose $M^{-1-s} \le 1-\theta$, and $\theta \le c\delta_0 (1-\theta)$, which yields
\begin{align*}
\tail([u-A(e_n)(x_n + a_m)_+^s]_- ; M^{-m}) &+ \tail([A(e_n)(x_n + b_m)_+^s - u]_- ; M^{-m}) \\
&\le M^{m(1-s)}|a_m - b_m| \delta_0.
\end{align*}
This estimate verifies the last claim. Moreover, note that we can apply the partial boundary Harnack inequality (see \autoref{lemma:partial-BH-scaled}) again as long as $b_m - a_m \le M^{-m} \eps_0$, which is in particular the case as long as $2\eps (1-\theta)^{m} M^{m} \le \eps_0$, i.e., by construction as long as $m \le m_0$.
As a consequence, the domain variation $\tilde{u}_{\eps}$ satisfies
\begin{align*}
\frac{a_m}{\eps} \le \tilde{u}_{\eps} \le \frac{b_m}{\eps} ~~ \text{ in } B_{\frac{1}{2} M^{-m}} \cap \{ x_n \ge 0\}.
\end{align*}
In particular, this means
\begin{align*}
\Gamma_{\eps} \cap \left( B_{\frac{1}{2} M^{-m} } \times [-1,1]) \right) \subset B_{\frac{1}{2} M^{-m} } \times \left[ \frac{a_m}{\eps}, \frac{b_m}{\eps}\right].
\end{align*}
This concludes the proof.
\end{proof}

The previous lemma is sufficient to yield compactness for domain variations, associated to a sequence of $\eps_k$-flat viscosity solutions in $B_2$ with $\eps_k \to 0$. This allows us to establish the following lemma, which is the crucial ingredient in the proof of \autoref{thm:improvement-of-flatness}.

\begin{lemma}
\label{lemma:linearized-problem}
Assume \eqref{eq:Kcomp}. Let $(u_k)_k$ be a sequence of viscosity solution to the nonlocal one-phase problem for $K$ in $B_2$ with $0 \in \partial \{ u_k > 0\}$. Let $(\eps_k)_k$ be such that $\eps_k \searrow 0$, and
\begin{align}
\label{eq:eps-k-flat}
A(e_n)(x_n - \eps_k)_+^s \le u_k(x) \le A(e_n)(x_n + \eps_k)_+^s ~~ \text{ in } B_1,
\end{align}
and
\begin{align}
\label{eq:eps-k-tail-smallness}
T_k := \tail([u_k-(x_n -\eps_k)_+^s]_- ; 1) + \tail([(x_n + \eps_k)_+^s - u_k]_- ; 1) \le \eps_k \delta_0.
\end{align}

Let us denote $\tilde{u}_k = \tilde{u}_{\eps_k}$ and $\Gamma_k = \Gamma_{\eps_k}$.  Then, the following hold true:
\begin{itemize}
\item[(i)] For any $\delta > 0$ it holds $\tilde{u}_k \to \tilde{u}$ uniformly in $B_{3/4} \cap \{ x_n \ge \delta \}$, where $\tilde{u}$ is H\"older continuous.
\item[(ii)] The sequence of graphs 
\begin{align*}
\Gamma_k \to \Gamma := \left\{ (x,\tilde{u}(x)) : x \in B_{3/4} \cap \{ x_n \ge 0 \} \right\}
\end{align*}
uniformly in the Hausdorff sense in $B_{3/4} \cap \{ x_n \ge 0 \}$.
\item[(iii)] There exist $C >0$ and $k_0 \in \N$, depending only on $n,s,\lambda,\Lambda$, such that for any $k \in \N$  with $k \ge k_0$ it holds
\begin{align}
\label{eq:convergence-tail-est}
\tail(|u_k - A(e_n)(x_n)_+^s|; 3/4) \le C \eps_k.
\end{align}
\item[(iv)] Let $K \in C^{1-2s+\beta}(\mathbb{S}^{n-1})$ for some $\beta > \max\{0,2s-1\}$. Then there is $f \in C^{1-2s+\beta}(B_{1/2})$ such that the function $\tilde{u}$ solves in the viscosity sense
\begin{align*}
\begin{cases}
L((x_n)_+^{s-1} \tilde{u} \1_{B_{3/4}}) &= f ~~ \text{ in } B_{1/2} \cap \{ x_n > 0 \},\\
(A^{1/s}(e_n)e_n - \nabla A^{1/s}(e_n))\nabla \tilde{u} &= 0 ~~ \text{ on } B_{1/2} \cap \{ x_n = 0 \}.
\end{cases}
\end{align*}
Moreover, $\Vert f \Vert_{L^{\infty}(B_{1/2})} \le C$ for some $C > 0$, depending only on $n,s,\lambda,\Lambda$.
\end{itemize}
\end{lemma} 

For the notion of viscosity solutions to nonlocal equations, we refer to \autoref{def:viscosity}. The interpretation of the boundary condition in (iv) in the viscosity sense is as follows:

\begin{definition}
We say that $\partial_{\nu} \tilde{u} = 0$ for some $\nu \in \mathbb{S}^{n-1}$ on $B_{1/2} \cap \{ x_n = 0\}$ in the viscosity sense, if for any $x_0 \in B_{1/2} \cap \{x_n = 0 \}$ and any smooth function $\tilde{\phi} : \R^n \to \R$ that satisfies $\tilde{\phi} \le \tilde{u}$ (resp. $\tilde{\phi} \ge \tilde{u}$) in $B_{1/2} \cap \{ x_n \ge 0 \}$, and $\tilde{\phi}(x_0) = \tilde{u}(x_0)$, it holds $\partial_{\nu} \tilde{\phi}(x_0) \le 0$ (resp. $\partial_{\nu} \tilde{\phi}(x_0) \ge 0$).
\end{definition}

\begin{proof}[Proof of \autoref{lemma:linearized-problem}]

The proof of (i) is standard (see \cite{DSS14}). By \autoref{lemma:Holder-continuity-partial-BH}, we have that 
\begin{align*}
|\tilde{u}_k(x) - \tilde{u}_k(y)| \le c |x-y|^{\alpha} ~~ \forall x,y \in B_{3/4} \cap  \{ x_n \ge \delta \} ~~ \text{ s.t. } |x-y| \ge \eps_k\eps_0^{-1}
\end{align*}
for some $\alpha \in (0,1)$ and any $\delta > 0$.
Here, $\tilde{u}_k$ denotes any function constructed by taking arbitrary elements of the corresponding sets $\tilde{u}_k(x)$ for any $x \in B_{3/4} \cap \{ x_n \ge 0\}$. Moreover, $\Vert \tilde{u}_k \Vert_{L^{\infty}(B_{3/4} \cap \{ x_n \ge \delta \} )} \le 1$. Therefore by the Arzel\`a-Ascoli theorem, there exists a subsequence converging uniformly in $B_{3/4} \cap \{ x_n \ge \delta \}$ to a $C^{\alpha}$ function $\tilde{u} : B_{3/4} \cap \{ x_n \ge \delta \} \to [-1,1]$. Since $\Vert \tilde{u} \Vert_{C^{\alpha}(B_{3/4} \cap \{ x_n \ge \delta \} )}$ is independent of $\delta > 0$, we can extend $\tilde{u}$ to a H\"older continuous function on $B_{3/4} \cap \{x_n \ge 0\}$.

Also the proof of (ii) is standard (see \cite[Lemma 7.14]{Vel23}). First, given $\tilde{x} = (x,\tilde{u}(x)) \in \Gamma$, and any $\delta > 0$, find $y \in B_{3/4} \cap \{ x_n \ge \delta/2 \}$ such that $|x-y| \le \delta$, and set $\tilde{y} = (y,\tilde{u}(y))$. Then, by (i), we have
\begin{align*}
|\tilde{x} - \tilde{y}| \le |x-y| + |\tilde{u}(x) - \tilde{u}(y)| \le \delta + c \delta^{\alpha}.
\end{align*}
Thus, for $k$ so large that $\eps_k \le \delta$, we have
\begin{align}
\label{eq:HD-help-1}
\dist(\tilde{x},\Gamma_k) \le |\tilde{x} - \tilde{y}| + \dist(\tilde{y},\Gamma_k) \le (\delta + c \delta^{\alpha}) + \Vert \tilde{u} - \tilde{u}_k \Vert_{L^{\infty}(B_{3/4} \cap \{ x_n \ge \delta/2 \} )}.
\end{align}
Next, given any $\tilde{x}_k = (x_k,\tilde{u}_k(x_k)) \in \Gamma_k$, and $k$ so large that $\eps_k / \eps_0 \le \delta/2$, we take $y_k \in B_{3/4} \cap \{ x_n \ge \delta \}$ such that $\delta/2 \le |x_k - y_k| \le 2 \delta$, and set $\tilde{y}_k = (y_k,\tilde{u}(y_k))$. Note that
\begin{align*}
|\tilde{x}_k-\tilde{y}_k| \le |x_k - y_k| + |\tilde{u}_k(x_k) - \tilde{u}_k(y_k)| \le 2\delta + c \delta^{\alpha},
\end{align*}
where we used \autoref{lemma:Holder-continuity-partial-BH}. Thus,
\begin{align}
\label{eq:HD-help-2}
\dist(\tilde{x}_k,\Gamma) \le |\tilde{x}_k - \tilde{y}_k | + \dist(\tilde{y}_k , \Gamma) \le (2\delta + c \delta^{\alpha}) + \Vert \tilde{u} - \tilde{u}_k \Vert_{L^{\infty}(B_{1/2} \cap \{ x_n \ge \delta/2 \} )}.
\end{align}
Finally using that $\tilde{u}_k \to \tilde{u}$ uniformly in $B_{1/2} \cap \{ x_n \ge \delta/2\}$ by (i), and that $\delta > 0$ can be taken arbitrarily small in \eqref{eq:HD-help-1}, \eqref{eq:HD-help-2}, we deduce the desired result.

Let us now prove (iii). If we choose $\theta \in (0,1)$ so small that $\theta \le c \delta_0 (1-\theta)$ and also $M = \max((1-\theta)^{-(1+s)},20)$, then by \autoref{lemma:Holder-continuity-partial-BH} (see also \eqref{eq:partial-BH-K-theta}), we obtain for any $m \le m_0$, where $m_0 \in \N$ satisfies $2 \eps_k (1-\theta)^{m_0}M^{m_0} \le \eps_0$:
\begin{align*}
\tail & ([u_k - A(e_n)(x_n - \eps_k)_+^s]_- ; M^{-m}) \\
&+ \tail([A(e_n)(x_n + \eps_k)_+^s - u_k]_- ; M^{-m}) \le 2 \eps_k M^{m(1-s)}(1-\theta)^{m}  \delta_0,
\end{align*}
which implies
\begin{align*}
\tail & ([u_k - A(e_n)(x_n - \eps_k)_+^s]_- ; 3/4) \\
&+ \tail([A(e_n)(x_n + \eps_k)_+^s - u_k]_- ; 3/4) \le 2 \eps_k M^{m(1+s)}(1-\theta)^{m}  \delta_0.
\end{align*}
Note that by construction, we have that $(1-\theta)M \ge (1-\theta)^{-s} > 1$. Thus, $m_0 \asymp \log_{(1-\theta)M}(\eps_0/(2\eps_k)) \to \infty$, as $k \to \infty$, so in particular, we can take $m = 1$ for any large enough $k$ in the previous estimate. This way, the previous estimate  becomes uniform in $k$.
Moreover, by the same computation as in \eqref{eq:L1-exponent-gain}, we deduce
\begin{align*}
\tail([u_k - A(e_n)(x_n)_+^s]_- ; 3/4) &\le \tail([u_k - A(e_n)(x_n - \eps_k)_+^s]_- ; 3/4) \\
&\quad + \tail([A(e_n)(x_n)_+^s - (x_n - \eps_k)_+^s]_- ; 3/4) \le c \eps_k
\end{align*}
for some constant $c > 0$. A similar reasoning applies to $\tail([u_k - A(e_n)(x_n)_+^s]_+ ; 3/4)$, so that we deduce \eqref{eq:convergence-tail-est}, as desired.

Now, we prove (iv). Let $\delta > 0$ and consider $k \in \N$ so large that $\delta > 4 \eps_k$. We define 
\begin{align*}
f_k(x) = -L\left( \left( \frac{u_k(x) - A(e_n)(x_n)_+^s}{\eps_k} \right) \1_{\R^n \setminus B_{3/4}} \right),
\end{align*}
and observe that 
\begin{align*}
L \left( \frac{u_k(x) - A(e_n)(x_n)_+^s}{\eps_k} \1_{B_{3/4}} \right) = f_k(x) ~~ \text{ in } B_{1/2} \cap \{ x_n > \delta \}.
\end{align*}
We claim that there exists $f \in C^{1-2s+\beta}(B_{1/2})$ such that, as $k \to \infty$
\begin{align}
\label{eq:lin-problem-iii-help-1}
\frac{u_k(x) - A(e_n)(x_n)_+^s}{\eps_k} &\to A(e_n)s(x_n)_+^{s-1} \tilde{u} ~~ \qquad \text{ uniformly in } B_{3/4} \cap \{ x_n > \delta \},\\
\label{eq:lin-problem-iii-help-2}
f_k &\to f ~~~ \qquad\qquad\qquad\qquad \text{ uniformly in } B_{1/2},\\
\label{eq:lin-problem-iii-help-3}
\frac{u_k(x) - A(e_n)(x_n)_+^s}{\eps_k} \1_{B_{3/4}} &\to A(e_n)s(x_n)_+^{s-1} \tilde{u}\1_{B_{3/4}} ~~ \text{ in } L^1_{2s}(\R^n).
\end{align}
Clearly, combining the previous four statements, we obtain 
\begin{align*}
L(A(e_n)s(x_n)_+^{s-1} \tilde{u} \1_{B_{3/4}}) &= f ~~ \text{ in } B_{1/2} \cap \{ x_n > \delta \}
\end{align*}
by the stability of viscosity solutions (see \cite[Proposition 3.2.12]{FeRo23}). Then, since $\delta > 0$ was arbitrary, we conclude the proof of the first part of (iii).

Let us prove \eqref{eq:lin-problem-iii-help-1}. By assumption \eqref{eq:eps-k-flat} we have that $u_k(x) \to A(e_n)(x_n)_+^s =: u(x)$ uniformly in $B_1$. 
Moreover, we can rewrite for any $x \in B_{3/4} \cap \{ x_n > \delta \}$, (note that the expression is zero if $\tilde{u}_k(x) = 0$)
\begin{align*}
\frac{u_k(x) - A(e_n)(x_n)_+^s}{\eps_k} = \frac{u_k(x) - u_k(x - \eps_k \tilde{u}_k(x) e_n)}{\tilde{u}_k(x) \eps_k} \tilde{u}_k(x).
\end{align*}
Also, note that 
\begin{align*}
\frac{u(x) - u(x - \eps_k \tilde{u}_k(x) e_n)}{\tilde{u}_k(x) \eps_k} \tilde{u}_k(x) \to \partial_n u(x) \tilde{u}(x) = A(e_n)s(x_n)_+^{s-1} \tilde{u}(x) ~~ \text{ uniformly in } B_{3/4} \cap \{ x_n > \delta \},
\end{align*}
since $\tilde{u}_k(x) \eps_k \in [-\eps_k , \eps_k]$ and therefore $\tilde{u}_k(x) \eps_k \to 0$ uniformly, and also using (i).\\
Consequently, it suffices to show that 
\begin{align*}
I_k(x) := \frac{u_k(x) - u_k(x - \eps_k \tilde{u}_k(x) e_n)}{\tilde{u}_k(x) \eps_k} - \frac{u(x) - u(x - \eps_k \tilde{u}_k(x) e_n)}{\tilde{u}_k(x) \eps_k} \to 0 ~~ \text{uniformly in } B_{3/4} \cap \{ x_n > \delta \}.
\end{align*}
To see this, note that $L(u_k - u) = 0$ in $B_1 \cap \{ x_n > \delta/4 \}$, and therefore by interior regularity estimates, using that $K \in C^{1-2s+\beta}(\mathbb{S}^{n-1})$, (see \cite[Theorem 2.4.1]{FeRo23}) for $x \in B_{3/4} \cap \{ x_n > \delta \}$:
\begin{align*}
|I_k(x)| \le [u_k - u]_{C^{0,1}(B_{7/8} \cap \{ x_n > \delta/2 \} )} \le c(\delta) \Vert u_k - u \Vert_{L^{\infty}(B_1)} + c(\delta) \tail(|u_k - u| ; 1).
\end{align*}

Now, note that by assumption \eqref{eq:eps-k-flat}, we have that $\Vert u_k - u \Vert_{L^{\infty}(B_1)} \to 0$, as $k \to \infty$. Moreover, from the assumption \eqref{eq:eps-k-tail-smallness}, and after performing a similar computation as in \eqref{eq:L1-exponent-gain}, we deduce
\begin{align*}
\tail(|u_k - u| ; 1) &\le \tail((u_k - u)_+ ; 1) + \tail((u_k - u)_- ; 1) \\
&\le  \tail([u_k - A(e_n)(x_n + \eps_k)_+^s]_+ ; 1) + A(e_n) \tail( [(x_n + \eps_k)_+^s - (x_n)_+^s]_+; 1) \\
&\quad + \tail([u_k - A(e_n)(x_n - \eps_k)_+^s]_- ; 1) + A(e_n)\tail( [(x_n - \eps_k)_+^s - (x_n)_+^s]_-; 1) \\
&\le 2c\eps_k \delta_0 + c A(e_n)\eps_k \to 0.
\end{align*}
This concludes the proof of \eqref{eq:lin-problem-iii-help-1}. 

Let us now turn to the proof of \eqref{eq:lin-problem-iii-help-2}. Note that for any $x \in B_{1/2}$:
\begin{align*}
|f_k(x)| &\le \int_{\R^n \setminus B_{3/4}} \left| \frac{u_k(y) - A(e_n)(y_n)_+^s}{\eps_k} \right| K(x-y) \d y \\
&\le c\int_{\R^n \setminus B_{3/4}} \left| \frac{u_k(y) - A(e_n)(y_n)_+^s}{\eps_k} \right| |y|^{-n-2s} \d y\\
&\le c \tail(|u_k - A(e_n)(x_n)_+^s|/\eps_k ; 3/4) \le C,
\end{align*}
where we used \eqref{eq:convergence-tail-est} in the last step. Moreover, since $K \in C^{1-2s+\beta}(\mathbb{S}^{n-1})$, we deduce for any $x,z \in B_{1/2}$:
\begin{align*}
\frac{|f_k(x) - f_k(z)|}{|x-z|^{1-2s+ \beta}} &\le \int_{\R^n \setminus B_{3/4}} \left| \frac{u_k(y) - A(e_n)(y_n)_+^s}{\eps_k} \right| \frac{|K(x-y) - K(z-y)|}{|x-z|^{1-2s+\beta}} \d y\\
&\le \sum_{k = 1}^{\infty} \int_{B_{\frac{3}{4} 2^{k+1}} \setminus B_{\frac{3}{4} 2^k } } \left| \frac{u_k(y) - A(e_n)(y_n)_+^s}{\eps_k} \right| [K]_{C^{1-2s+\beta}(\R^n \setminus B_{\frac{1}{4} 2^{k} })} \d y\\
&\le c \int_{\R^n \setminus B_{3/4}} \left| \frac{u_k(y) - A(e_n)(y_n)_+^s}{\eps_k} \right|  |y|^{-n-2s-(1-2s+\beta)} \d y \\
&\le c \tail(|u_k - A(e_n)(x_n)_+^s|/\eps_k ; 3/4) \le C,
\end{align*}
where we used (iii). Thus, there exists $C > 0$, independent of $k$ such that $\Vert f_k \Vert_{C^{1-2s+\beta}(B_{1/2})} \le C$, and \eqref{eq:lin-problem-iii-help-2} follows by an application of the Arzel\`a-Ascoli theorem. By uniform convergence, we also get $\Vert f \Vert_{C^{1-2s+\beta}(B_{1/2})} \le C$. However, note that the bound on $\Vert f \Vert_{L^{\infty}(B_1)}$ is independent of $\Vert K \Vert_{C^{1-2s+\beta}(\mathbb{S}^{n-1})}$.

Finally, we prove \eqref{eq:lin-problem-iii-help-3}. First, note that since the sequence is compactly supported, it suffices to show 
\begin{align}
\label{eq:lin-problem-iii-help-3-help}
\frac{u_k(x) - A(e_n)(x_n)_+^s}{\eps_k} \to A(e_n)s(x_n)_+^{s-1} \tilde{u} ~~ \text{ in } L^1(B_{3/4}).
\end{align}
Note that by assumption \eqref{eq:eps-k-flat}, for any $\delta > 0$ it holds for $k$ large enough:
\begin{align*}
\frac{u_k(x) - A(e_n)(x_n)_+^s}{\eps_k} = 0 ~~ \text{ in } B_{3/4} \cap \{ x_n \le - \delta \}.
\end{align*}
Therefore, using also \eqref{eq:lin-problem-iii-help-1}, we know that for any $\delta \in (0,1)$
\begin{align}
\label{eq:uniform-conv-delta-strip}
\frac{u_k(x) - A(e_n)(x_n)_+^s}{\eps_k} \to A(e_n) s(x_n)_+^{s-1} \tilde{u} ~~ \text{ uniformly in } B_{3/4} \cap \{ |x_n| > \delta \}.
\end{align}

Moreover, it is easy to see from the $\eps_k$-flatness of $u_k$ in $B_1$ (see \eqref{eq:eps-k-flat}) that there is $C > 0$, independent of $\delta$ and $k$, such that for any $\delta \in (0,\frac{3}{4})$:
\begin{align}
\label{eq:L1-bound-k}
\begin{split}
\left\Vert \frac{u_k(x) - A(e_n)(x_n)_+^s}{\eps_k} \right\Vert_{L^1(B_{3/4} \cap \{ |x_n| \le \delta \} )} &\le A(e_n)\left\Vert \frac{(x_n - \eps_k)_+^s - (x_n)_+^s}{\eps_k} \right\Vert_{L^1(B_{3/4} \cap \{ |x_n| \le \delta \})} \\
&\quad + A(e_n)\left\Vert \frac{(x_n + \eps_k)_+^s - (x_n)_+^s}{\eps_k} \right\Vert_{L^1(B_{3/4}  \cap \{ |x_n| \le \delta \} )} \\
& \le C \delta^{s}.
\end{split}
\end{align}
Indeed, this follows from a similar computation as in \eqref{eq:L1-exponent-gain}:
\begin{align*}
\left\Vert \frac{(x_n - \eps_k)_+^s - (x_n)_+^s}{\eps_k} \right\Vert_{L^1(B_{3/4} \cap \{ |x_n| \le \delta \})} \hspace{-0.5cm} &\le c \eps_k^{-1} \int_{-\delta}^{\delta} (x_+^s - (x-\eps_k)_+^s) \d x \\
&\le c \eps_k^{-1}\left(\int_0^{\delta} x^s \d x - \int_{\eps_k}^{\delta} (x-\eps_k)^s \d x \right) \\
&= c \eps_k^{-1}\left( \int_{\delta - \eps_k}^{\delta} x^s \d x \right) = c \eps_k^{-1} \left(\delta^{1+s} - (\delta - \eps_k)^{1+s} \right) \le c \delta^{s},
\end{align*}
and an analogous argument yields the estimate for the other term.\\
Moreover, since $\Vert \tilde{u}_k \Vert_{L^{\infty}(B_{3/4})} \le 1$ for any $k$, and $\tilde{u}_k \to \tilde{u}$ uniformly in $B_{3/4} \cap \{x_n \ge \delta\}$ for any $\delta > 0$, we also have $\Vert \tilde{u} \Vert_{L^{\infty}(B_{3/4} \cap \{x_n > 0\} )} \le 1$ and thus
\begin{align}
\label{eq:L1-bound-limit}
\Vert A(e_n) s (x_n)_+^{s-1} \tilde{u} \Vert_{L^1(B_{3/4})} \le c \Vert (x_n)_+^{s-1} \Vert_{L^1(B_{3/4} \cap \{x_n > 0\} )} \le C.
\end{align}
Altogether, this implies \eqref{eq:lin-problem-iii-help-3-help}. Indeed, given any $\eta \in (0,1)$, due to \eqref{eq:L1-bound-k} and \eqref{eq:L1-bound-limit} we can find $\delta > 0$ such that 
\begin{align*}
\left\Vert \frac{u_k(x) - A(e_n)(x_n)_+^s}{\eps_k} \right\Vert_{L^1(B_{3/4} \cap \{ |x_n| \le \delta\} )} + \Vert A(e_n) s (x_n)_+^{s-1} \tilde{u} \Vert_{L^1(B_{3/4} \cap \{|x_n| \le \delta \} )} \le \eta/2,
\end{align*}
and by \eqref{eq:uniform-conv-delta-strip}, we can choose $k$ large enough, such that
\begin{align*}
\left\Vert \frac{u_k(x) - A(e_n)(x_n)_+^s}{\eps_k} - A(e_n) s(x_n)_+^{s-1} \tilde{u} \right\Vert_{L^{1}(B_{3/4} \cap \{|x_n| \ge \delta \} )} < \eta/2.
\end{align*}

Finally, we verify the Neumann condition (iv). To this end, let $\tilde{\phi}$ be a smooth function in $B_1$ such that $\tilde{\phi} \le \tilde{u}$ in $B_{3/4} \cap \{x_n \ge 0\}$ and $\tilde{\phi}(x_0) = \tilde{u}(x_0)$ for some $x_0 \in B_{1/2} \cap \{ x_n = 0\}$. The proof in case $\tilde{\phi} \ge \tilde{u}$ goes in the same way, and we will skip it.\\
For simplicity of notation, we will write from now on $x_0 = 0$. Due to the Hausdorff-convergence (ii) in $\overline{B_{3/4}} \cap \{ x_n \ge 0 \}$, and \cite[Lemma 3.2.10]{FeRo23} (note that the proof therein carries over to our setup since $\Omega = \overline{B_{3/4}} \cap \{x_n \ge 0 \}$ is closed) we have that there exist $\tilde{x}_k \in B_{3/4} \cap \{ x_n \ge 0 \}$ and $c_k \in \R$ with $c_k \to 0$ and $\tilde{\phi}_k = \tilde{\phi} + c_k$ such that $\tilde{\phi}_k \le \tilde{u}_k$ in $ \overline{B_{3/4}} \cap \{ x_n \ge 0 \}$ and $\tilde{\phi}_k(\tilde{x}_k) = \tilde{u}_k(\tilde{x}_k)$.\\
Note that $x - \eps_k \tilde{\phi}_k(x)$ is invertible in $B_{1/2}$, when $\eps_k$ is small enough. Therefore, there exists a function $\phi_k^{1/s} : B_{1/2} \to \R$ such that
\begin{align}
\label{eq:smooth-expansion}
\phi_k^{1/s}(x - \eps_k \tilde{\phi}_k(x)e_n) = A(e_n)^{1/s} x_n ~~ \forall x \in B_{1/2}.
\end{align}

Since, for $\eps_k$ small enough, $1 - \eps_k \partial_n\tilde{\phi}_k(x) > 0$ for $x \in B_{1/2}$, the function $\phi_k^{1/s}$ is smooth in $B_{1/2}$, by the implicit function theorem. Moreover, note that in particular, for $x \in B_{1/2} \cap \{ x_n \ge 0\}$, we have by construction (where we define $\phi_k = (\phi_k^{1/s})_+^s$)
\begin{align*}
\phi_k(x - \eps_k \tilde{\phi}_k(x)e_n) = A(e_n)(x_n)_+^s.
\end{align*}
The function $\phi_k$ is a priori only defined in $B_{1/2}$. Let us set $\phi_k = u_k$ in $(\R^n \setminus B_{1/2})$.
Note that the $\phi_k$ are also smooth in $\{ \phi_k > 0 \}$, globally $C^s$, and strictly increasing in the $e_n$-direction for $\eps_k$ small enough, and thus, by \autoref{lemma:comp-dom-var}, we have  $\phi_k \le u_k$ in $\R^n$, and $\phi_k(x_k) = u_k(x_k)$, where $x_k = \tilde{x}_k - \eps_k \tilde{\phi}_k(\tilde{x}_k) e_n$.\\
The next goal of the proof is to argue that $x_k \in \partial \{ u_k > 0\}$. To rule out that $x_k \in \overline{ \{ u_k = 0 \} } \setminus \partial \{ u_k > 0 \}$, observe that 
if $u_k(x_k) = 0$, this implies
\begin{align*}
0 = u_k(x_k) = \phi_k(x_k) = \phi_k(\tilde{x}_k - \eps_k \tilde{\phi}_k(\tilde{x}_k) e_n) = A(e_n)((\tilde{x}_k)_n)_+^s,
\end{align*}
and therefore $(\tilde{x}_k)_n = 0$. Now, if $x_k \in \overline{ \{ u_k = 0 \} } \setminus \partial \{ u_k > 0 \}$, this would imply that for any point $\tilde{z}$ in a small neighborhood of $\tilde{x}_k$ with $\tilde{z}_n > 0$ it holds
\begin{align*}
A(e_n)(\tilde{z}_n)_+^s = \phi_k(\tilde{z} - \eps_k \tilde{\phi}_k(\tilde{z})e_n ) \le u_k(\tilde{z} - \eps_k \tilde{\phi}_k(\tilde{z})e_n)  = 0,
\end{align*}
since $\tilde{z} - \eps_k \tilde{\phi}_k(\tilde{z})e_n$ is close to $x_k$ by the smoothness of $\tilde{\phi}_k$, a contradiction.

Next, let us assume that $x_k \in \{ u_k > 0 \}$. In this case, our goal will be to prove that $L \phi_k(x_k) < 0$, since this gives a contradiction to $Lu(x_k) = 0$, by the notion of viscosity solution in \autoref{def:viscosity}, (resp. \autoref{def:viscosity-solution}) it implies that $x_k \in \partial \{ x_k > 0\}$, as desired.\\
To see that $L \phi_k(x_k) < 0$, let us first observe that $\phi_k(x_k) = u(x_k) > 0$, since $\{ u_k > 0 \}$ is an open set, so $\phi_k$ is smooth around $x_k$. Next, by Taylor's formula for $x \in B_{3/4} \cap \{ x_n > - 2\eps_k \} \cap \{ \phi_k > 0 \}$, using
\begin{align*}
A(e_n)(x_n + 2\eps_k)_+^s &= \phi_k(x + \eps_k e_n (2- \tilde{\phi}_k(x + 2\eps_k e_n)) ) \\
&= \phi_k(x) + \eps_k (2 - \tilde{\phi}_k(x + 2\eps_k e_n)) \partial_n\phi_k(x) + \eps_k^2 \Psi(x),
\end{align*}
where $\Psi$ is a placeholder for a smooth function with bounded derivatives, which might change in the following lines, and that therefore 
\begin{align*}
A(e_n)s(x_n + 2\eps_k)_+^{s-1} = \partial_n \phi_k(x) + \eps_k \Psi(x),
\end{align*}
we finally get the following:
\begin{align}
\label{eq:phi-tilde-phi-formula-2}
A(e_n)(x_n + 2\eps_k)_+^s = \phi_k(x) + \eps_k (2 - \tilde{\phi}_k(x + 2\eps_k e_n)) A(e_n) s(x_n + 2\eps_k)_+^{s-1} + \eps_k^2\Psi(x).
\end{align}
Thus, using \eqref{eq:phi-tilde-phi-formula-2}, we obtain for any $x \in B_{1/2} \cap \{ x_n > -\eps_k\} \cap \{ \phi_k > 0 \} $ (and so especially for $x_k$):
\begin{align}
\label{eq:L-tilde-phi-estimate-1}
\begin{split}
L \phi_k(x) &= L (\phi_k \1_{B_{3/4} \cap \{ x_n > -2\eps_k \} })(x) + L (\phi_k \1_{(\R^n \setminus B_{3/4}) \cup \{ x_n < -2\eps_k \} })(x)\\
&= L(A(e_n)(x_n + 2\eps_k)_+^s)(x) + L( [\phi_k - A(e_n)(x_n + 2 \eps_k)_+^s ] \1_{(\R^n \setminus B_{3/4}) \cup \{ x_n < - 2\eps_k \} })(x) \\
& \quad + \eps_k L(A(e_n)s(x_n + 2\eps_k)_+^{s-1} \tilde{\phi}_k(x + 2\eps_k e_n ) \1_{B_{3/4} \cap \{ x_n > - 2\eps_k \} }  )(x) \\
& \quad - 2\eps_k L(A(e_n)s(x_n + 2\eps_k)_+^{s-1} )(x) + 2\eps_k L(A(e_n)s(x_n + 2\eps_k)_+^{s-1} \1_{(\R^n \setminus B_{3/4}) \cup \{ x_n < - 2\eps_k \}  })(x) \\
&\quad - \eps_k^2 L(\Psi \1_{B_{3/4} \cap \{ x_n > - 2\eps_k \}})(x) \\
&= L( [\phi_k - A(e_n)(x_n + 2 \eps_k)_+^s ] \1_{(\R^n \setminus B_{3/4}) \cup \{ x_n < - 2\eps_k \} })(x)  + \eps_k L(A(e_n)s(x_n)_+^{s-1} \tilde{\phi}_k \1_{B_{3/4}})(x + 2\eps_k) \\
&\quad + 2\eps_k L(A(e_n)s(x_n + 2\eps_k)_+^{s-1} \1_{\R^n \setminus B_{3/4} })(x) - \eps_k^2  L(\Psi \1_{B_{3/4} \cap \{ x_n > - 2\eps_k \}})(x) \\
&=: I_1 + I_2 + I_3 + I_4,
\end{split}
\end{align}
where we used that $L((x_n)_+^s) = L((x_n)_+^{s-1}) = 0$ in $\{ x_n > 0 \}$.
Let us now discuss how to estimate the terms $I_1,I_2,I_3,I_4$.

For $I_1$, we observe that since $\phi_k = u_k$ in $\R^n \setminus B_{1/2}$ and by construction in $B_{1/2}$, we have $\phi_k \ge 0$,
\begin{align*}
I_1 &= L( [\phi_k - A(e_n)(x_n + 2 \eps_k)_+^s ] \1_{(\R^n \setminus B_{3/4}) \cup \{ x_n < - 2\eps_k \} })(x) \\
& \le - c\int_{(\R^n \setminus B_{3/4}) \cup \{ x_n < -2\eps_k \} } \phi_k(y) |x-y|^{-n-2s} \d y + c\int_{ \R^n \setminus B_{3/4} } (y_n + 2 \eps_k)_+^s |x-y|^{-n-2s} \d y \le c.
\end{align*}
For $I_3$, we have the trivial estimate $I_3 \le 0$, and for $I_4$, we use that $|L \Psi \1_{B_{3/4}}| \le C$ to get, using that $x_n > - \eps_k$ implies $\dist(x, \{x_n > - 2 \eps_k\}) \ge \eps_k$, and that $\Psi$ is bounded:
\begin{align*}
I_4 \le C\eps_k^2 + \eps_k^2L(\Psi \1_{B_{3/4} \setminus \{ x_n > - 2\eps_k \} }) \le C\eps_k^2 + C \eps_k^{2-2s}.
\end{align*}
Since the estimation if $I_2$ is significantly more involved, we postpone it, and summarize first, that altogether, we have shown that for any $x \in B_{1/2} \cap \{ x_n > -\eps_k\} \cap \{ \phi_k > 0 \}$
\begin{align}
\label{eq:L-phik-xk-estimate-help-1}
L \phi_k(x) \le C (1 +\eps_k^{2-2s}) + \eps_k L(A(e_n)s(x_n)_+^{s-1} \tilde{\phi}_k \1_{B_{3/4}})(x + 2\eps_k).
\end{align}
To estimate further the second summand, let us define 
\begin{align*}
g(x) = -\delta x_n + \delta^{-1}(x_n)_+^{1 + \eta} \psi(x),
\end{align*}
where $\eta \in (s,2s)$, and $\psi \in C^{\infty}(\R^n)$ satisfies $0 \le \psi \le 1$, $\psi \equiv 1$ in $B_{\delta^{2/\eta}/2}$, and $\supp(\psi) = \overline{ B_{\delta^{2/\eta}} }$.\\
Then, we replace $\tilde{\phi}$ in the proof above by $\tilde{\phi}^{(\delta)} = \tilde{\phi} + g$ and observe that all the previous arguments, in particular the estimate \eqref{eq:L-phik-xk-estimate-help-1}, remain valid. In fact, $\tilde{\phi}^{(\delta)}$ is still smooth, and since $g \le 0$ in $\{ x_n \ge 0 \}$ and $g(0) = 0$, we have that $\tilde{\phi}^{(\delta)}$ still touches $\tilde{u}$ from below at $0$. Moreover, $\nabla \tilde{\phi}^{(\delta)} = \nabla \tilde{\phi} - \delta e_n$ in $B_{\delta^{2/\eta}/2}$, and therefore it suffices to prove that $(A^{1/s}(e_n)e_n - \nabla A^{1/s}(e_n)) \cdot \nabla \tilde{\phi}^{(\delta)}(0) \le 0$ for every $\delta > 0$ small enough.\\
The advantage of the modification of $\tilde{\phi}$ is that we have gained some control over the value of $L \phi_k(x_k)$. 
To see this, let us estimate $L ((x_n)_+^{s-1} \tilde{\phi}_k^{(\delta)} \1_{B_{3/4}})(\cdot + 2\eps_k)$, i.e., the second summand in \eqref{eq:L-phik-xk-estimate-help-1}. First, observe that for $x \in B_{1/2} \cap \{ x_n > - \eps_k \}$: 
\begin{align}
\label{eq:L-tilde-phi-estimate-2}
\begin{split}
L ((x_n)_+^{s-1} & \tilde{\phi}_k^{(\delta)} \1_{B_{3/4}} )(x + 2\eps_k) \\
&= L ((x_n)_+^{s-1} \tilde{\phi} \1_{B_{3/4}})(x + 2\eps_k) + c_k L((x_n)_+^{s-1} \1_{B_{3/4}})(x + 2\eps_k) \\
&\quad - \delta L((x_n)_+^{s} \1_{B_{3/4}})(x + 2\eps_k) + \delta^{-1} L ((x_n)_+^{s+\eta}\psi \1_{B_{3/4}})(x + 2\eps_k) \\
&\le C + \delta^{-1} L ((x_n)_+^{s+\eta}\psi \1_{B_{3/4}})(x + 2\eps_k),
\end{split}
\end{align}
for some constant $C > 0$. Here, we used again $L((x_n)_+^s)(\cdot + 2\eps_k) = L((x_n)_+^{s-1})(\cdot + 2\eps_k) = 0$, as well as 
\begin{align}
\label{eq:L-tilde-phi-estimate-3}
|L ((x_n)_+^{s-1} \tilde{\phi} )(\cdot + 2\eps_k)| \le C,
\end{align}
and that therefore
\begin{align*}
& |L((x_n)_+^{s-1} \1_{B_{3/4}})(\cdot + 2\eps_k)| + |L((x_n)_+^{s} \1_{B_{3/4}})(\cdot + 2\eps_k)| + |L ((x_n)_+^{s-1} \tilde{\phi} \1_{B_{3/4}})(\cdot + 2\eps_k)| \\
& \le |L((x_n)_+^s \1_{\R^n \setminus B_{3/4}})(\cdot + 2\eps_k)| + |L((x_n)_+^{s-1} \1_{\R^n \setminus B_{3/4}})(\cdot + 2\eps_k)| \\
&\quad + |L ((x_n)_+^{s-1} \tilde{\phi} \1_{\R^n \setminus B_{3/4}})(\cdot + 2\eps_k)| + C \\
&\le C.
\end{align*}
To see \eqref{eq:L-tilde-phi-estimate-3}, we recall \cite[Lemma 9.5]{RoSe16}, which implies that $L_1((x_n)_+^s \tilde{\phi}) \in C^1(B_{1} \cap \{ x_n > 0 \})$, where $L_1$ denotes the operator $L$ but with $K$ replaced by $K\1_{B_1}$. Then, the product rule yields,
\begin{align*}
L_1(s(x_n)_+^{s-1} \tilde{\phi} ) = \partial_n L_1((x_n)_+^s \tilde{\phi} ) - L_1((x_n)_+^s \partial_n \tilde{\phi} ) \in L^{\infty}(B_{1} \cap \{ x_n > 0 \} ).
\end{align*}
Since a computation with polar coordinates reveals that $(L - L_1)((x_n)_+^{s-1} \tilde{\phi}) \in  L^{\infty}(B_{1} \cap \{ x_n > 0 \} )$, we deduce \eqref{eq:L-tilde-phi-estimate-3}, as desired.\\
To estimate the second term in \eqref{eq:L-tilde-phi-estimate-2}, we first compute at zero and for $\delta > 0$ small:
\begin{align*}
\delta^{-1} L ((x_n)_+^{s+\eta}\psi \1_{B_{3/4}})(0) &= -\delta^{-1} \int_{B_{\delta^{2/\eta}}} (y_n)_+^{s+\eta} \psi(y)\1_{B_{3/4}}(y) K(y) \d y \\
&\le - c\delta^{-1}\int_{B_{\delta^{2/\eta}/2}} (y_n)_+^{s+\eta} |y|^{-n-2s} \d y\\
&\le - c\delta^{-1}\int_{B_{\delta^{2/\eta}/2} \cap \{ y_n \ge |y|/2 \} } |y|^{-n-s+\eta} \d y \le - c\delta^{-1 + \frac{2}{\eta}(-s+\eta)} = -c \delta^{1 - \frac{2s}{\eta}}.
\end{align*}
Therefore, $\delta^{-1} L ((x_n)_+^{s+\eta}\psi \1_{B_{3/4}})(0) \to -\infty$, as $\delta \to 0$. Moreover, since $x \mapsto [(x_n)_+^{s+\eta}\psi(x)\1_{B_{3/4}}(x)] \in C^{s+\eta}(B_{3/4}) \cap L^{\infty}(\R^n)$, we know that $x \mapsto [L ((x_n)_+^{s+\eta} \psi \1_{B_{3/4}})(x)] \in C^{\eta - s - \eps}(B_{1/2})$ is H\"older continuous. Therefore, using also \eqref{eq:L-phik-xk-estimate-help-1} and \eqref{eq:L-tilde-phi-estimate-2}, we find that there exists $\delta_0 \in (0,1)$ depending only on $n,s,\lambda,\Lambda,$ and $\Vert K \Vert_{C^{1+\beta}(\mathbb{S}^{n-1})}$, such that for any $\delta \in (0,\delta_0)$, there exists a small neighborhood around $0$ inside $B_{1/2} \cap \{ x_n > - \eps_k \} \cap \{ \phi_k > 0 \}$, in which $L \phi_k(\cdot) < 0$. Thus, altogether, using that $x_k \in B_{1/2} \cap \{ x_n > - \eps_k \} \cap \{ \phi_k > 0 \}$ converges to $0$ as $k \to \infty$, we conclude that $L \phi_k(x_k) < 0$ for $k$ large enough, and thus $x_k \in \partial \{ u_k > 0 \}$ for $k$ large enough, as desired.

We have shown that $x_k \in \partial \{ u_k > 0\}$ for $k$ large enough. Let us now conclude the proof. To do so, note that by \eqref{eq:smooth-expansion}, it holds 
\begin{align}
\label{eq:comp-phi_k-derivative}
\nabla \phi_k^{1/s}(x_k) = A^{1/s}(e_n) (e_n + \eps_k \nabla \tilde{\phi}_k(\tilde{x}_k) + O(\eps_k^2)).
\end{align}
Indeed, if we set $g(x) := x - \eps_k \tilde{\phi}_k(x) e_n$, then \eqref{eq:smooth-expansion} reads
\begin{align*}
\phi_k^{1/s}(x) = A^{1/s}(e_n) (g^{-1}(x)_n)_+.
\end{align*}
Moreover, since $D g(x) = I - \eps_k D(\tilde{\phi}_k(x) e_n)$, and $\tilde{\phi}_k = \tilde{\phi} + c_k$ for some constant $c_k \in \R$, the implicit function theorem yields \eqref{eq:comp-phi_k-derivative} where the constant in the term $O(\eps_k^2)$ only depends on $\Vert \tilde{\phi} \Vert_{C^2}$.\\
Let us now extend $A^{1/s}$ to $\R^n \setminus \{ 0 \}$ by setting $A^{1/s}(r \theta) = A^{1/s}(\theta)$ for any $\theta \in \mathbb{S}^{n-1}$ and $r > 0$.
Thus, using that for some $\gamma \in (0,s)$ it holds $A^{1/s} \in C^{1,\gamma}(\R^n \setminus \{ 0 \})$ by \eqref{eq:A-properties}, we obtain
\begin{align*}
A^{1/s}(\nabla \phi_k^{1/s}(x_k)) &= A^{1/s}(A^{1/s}(e_n) e_n) + A^{1/s}(e_n) \eps_k \nabla \tilde{\phi}_k(\tilde{x}_k) \nabla A^{1/s}(A^{1/s}(e_n) e_n) + O(\eps_k^{1+\gamma}) \\
&= A^{1/s}(e_n) + \eps_k \nabla \tilde{\phi}_k(\tilde{x}_k) \nabla A^{1/s}(e_n) + O(\eps_k^{1+\gamma}),
\end{align*}
where we used that $A^{1/s}(A^{1/s}(e_n) e_n) = A^{1/s}(e_n)$ and $\nabla A^{1/s}(A^{1/s}(e_n) e_n) = A^{1/s}(e_n)^{-1} \nabla A^{1/s}(e_n)$.

Combining this observation with \eqref{eq:comp-phi_k-derivative} and the viscosity free boundary condition on $u_k$ (see \autoref{def:viscosity}(ii)), we obtain
\begin{align*}
A^{1/s}(e_n)^2 &+ 2\eps_k A^{1/s}(e_n)^2 \partial_n \tilde{\phi}_k(\tilde{x}_k) + O(\eps_k^{1+\gamma}) \\
&\le |A^{1/s}(e_n)(e_n + \eps_k \nabla \tilde{\phi}_k(\tilde{x}_k) + O(\eps_k^{1+\gamma}))|^2 \\
&= |\nabla \phi_k^{1/s}(x_k) |^2 \\
&\le A^{1/s}(\nabla \phi_k^{1/s}(x_k)/|\nabla \phi_k^{1/s}(x_k)|)^{2} = A^{1/s}(\nabla \phi_k^{1/s}(x_k))^{2} \\
&\le |A^{1/s}(e_n) + \eps_k\nabla \tilde{\phi}_k(\tilde{x}_k)\nabla A^{1/s}(e_n) + O(\eps_k^{1+\gamma})|^2\\
&\le A^{1/s}(e_n)^2 + 2 \eps_k A^{1/s}(e_n) \nabla \tilde{\phi}_k(\tilde{x}_k)\nabla A^{1/s}(e_n) + O(\eps_k^{1+\gamma}),
\end{align*}
which implies
\begin{align*}
A^{1/s}(e_n)\partial_n \tilde{\phi}_k(\tilde{x}_k) - \nabla \tilde{\phi}_k(\tilde{x}_k)\nabla A^{1/s}(e_n) \le O(\eps_k^{\gamma}),
\end{align*}
and yields the desired result upon taking the limit $k \to 0$, namely
\begin{equation*}
A^{1/s}(e_n)\partial_n \tilde{\phi}(x) - \nabla \tilde{\phi}(x)\nabla A^{1/s}(e_n) \le 0. \qedhere
\end{equation*}
\end{proof}

\subsubsection{Regularity of the linearized problem}
\label{subsubsec:reg-linearized}

The following theorem establishes the boundary regularity for solutions to the linearized problem. It is a direct consequence of \cite[Theorem 1.4, Theorem 6.2]{RoWe24}. For another one-phase problem whose linearized problem has an oblique boundary condition, we refer the interested reader to \cite{DeSa21}.

\begin{lemma}
\label{lemma:reg-linearized}
Assume \eqref{eq:Kcomp}. Let $\gamma \in (0,s)$. Let $f \in L^{\infty}(B_1)$, $\omega \in \mathbb{S}^{n-1}$ with $\omega_n \ge \delta$ for some $\delta > 0$, and $v$ be a viscosity solution to
\begin{align*}
\begin{cases}
L((x_n)_+^{s-1} v) &= f ~~ \text{ in } B_{1} \cap \{ x_n > 0 \},\\
\partial_{\omega} v &= 0 ~~ \text{ on } B_{1} \cap \{ x_n = 0 \}.
\end{cases}
\end{align*}
Then, $v \in C^{1,\gamma}(B_{1/2} \cap \{x_n > 0\})$ and for any $x_0 \in B_{1/2} \cap \{ x_n = 0 \}$ and $x \in B_{1/2} \cap \{ x_n > 0 \}$ it holds
\begin{align*}
|v(x) - v(x_0) - a(x_0) \cdot (x - x_0)| \le c \big( \Vert v \Vert_{L^{\infty}(\R^n)} + \Vert f \Vert_{L^{\infty}(B_1)} \big) |x-x_0|^{1+\gamma}
\end{align*}
for some $c > 0$, which only depends on $n,s,\lambda,\Lambda,\delta,\gamma$, and $a(x_0) \in \R^{n}$ with $a(x_0) \cdot \omega = 0$. Moreover, we have the following estimate: 
\begin{align*}
|a(x_0)| \le c \big( \Vert v \Vert_{L^{\infty}(\R^n)} + \Vert f \Vert_{L^{\infty}(B_1)} \big)
\end{align*}
for some $c > 0$, which only depends on $n,s,\lambda,\Lambda,\delta,\gamma$.
\end{lemma}

\begin{proof}
Note that in case $\omega = e_n$, the result immediately follows from \cite{RoWe24}. In fact, the expansion for $v$ is contained in \cite[Theorem 6.2 (i)]{RoWe24}. Note that we do not have to assume any regularity for $K$ due to the discussion in the proof of \cite[Theorem 1.4]{RoWe24}. The estimate for $a(x_0)$ follows by fixing some $x \in B_{1/2} \cap \{x_n > 0 \}$ such that $|x-x_0| = 1/4$ and $|a(x_0) \cdot (x - x_0)| \ge |a(x_0)|/8$ and applying the first estimate and the triangle inequality.\\
For general $\omega \in \mathbb{S}^{n-1}$ with $\omega_n \ge \delta$ for some $\delta > 0$, let us adapt an idea from \cite{DeSa21} and consider 
\begin{align*}
w(x) = v \left(x' + \omega' \frac{x_n}{\omega_n} , x_n \right).
\end{align*}
The function $w$ solves 
\begin{align*}
\begin{cases}
\tilde{L}((x_n)_+^{s-1} w) &= \tilde{f} ~~ \text{ in } B_{3/4} \cap \{ x_n > 0 \},\\
\partial_n w &= 0 ~~ \text{ on } B_{3/4} \cap \{ x_n = 0 \},
\end{cases}
\end{align*}
where $\tilde{f}(x) = f(x' + \omega' \frac{x_n}{\omega_n} , x_n)$ satisfies $\tilde{f} \in L^{\infty}(B_{3/4} \cap \{x_n > 0 \})$, and $\tilde{L}$ is an operator of the same form as $L$, but with kernel $\tilde{K}$ given as
\begin{align*}
\tilde{K}(h) &= \frac{|\omega_n|}{|\omega'|} \left(\sqrt{ \left|h' - \frac{\omega' h_n}{\omega_n}\right|^2 +h_n^2}\right)^{-n-2s} a\left(  \frac{ \left(\big( h' - \frac{\omega'h_n}{\omega_n} \big) , h_n \right) }{\left| \left( \big(h' - \frac{\omega'h_n}{\omega_n}\big) , h_n \right)\right|} \right) \\
&= \frac{|\omega_n|}{|\omega'|} |h|^{-n-2s} \left(\sqrt{1 + \left(\frac{|\omega'|h_n}{|h|\omega_n}\right)^2 - \frac{2(h' \cdot \omega')h_n}{|h|^2 \omega_n }}\right)^{-n-2s} a\left( \frac{ \left( \big( h' - \frac{\omega'h_n}{\omega_n}\big) , h_n \right) }{\left|\left( \big( h' - \frac{\omega'h_n}{\omega_n}\big) , h_n \right) \right|} \right).
\end{align*}
Note that $\tilde{K}$ is still homogeneous and symmetric, and moreover, by the assumption $\omega_n \ge \delta$, it also follows that $\tilde{\lambda} |h|^{-n-2s} \le \tilde{K}(h) \le \tilde{\Lambda} |h|^{-n-2s}$, where $0 < \tilde{\lambda} < \tilde{\Lambda}$ depend only on $n,s,\lambda,\Lambda,\delta$. Hence, it still satisfies \eqref{eq:Kcomp} with $\tilde{\lambda}, \tilde{\Lambda}$. Thus, the desired result follows after an application of the aforementioned results from \cite{RoWe24} to $w$, $\tilde{f}$, $\tilde{L}$, and translating the result back to $v$.
\end{proof}

\subsubsection{Conclusion of the proof of \autoref{thm:improvement-of-flatness}}
\label{subsubsec:conclusion}

In this section, we will conclude the proof of the improvement of flatness result (see \autoref{thm:improvement-of-flatness}). First, we establish the following elementary lemma:

\begin{lemma}
\label{lemma:conclusion-iof}
Let $u \in C(B_2)$ and $0 \in \partial \{ u > 0 \}$. Moreover, assume that
\begin{align*}
A(e_n)(x_n - \eps)_+^s \le u(x) \le A(e_n)(x_n + \eps)_+^s ~~ \text{ in } B_1
\end{align*}
and that $\tilde{u}_{\eps}$ satisfies
\begin{align}
\label{eq:expansion-ass}
a \cdot x - \sigma\frac{\rho}{4} \le \tilde{u}_{\eps}(x) \le a \cdot x + \sigma\frac{\rho}{4} ~~ \text{ in } B_{2\rho} \cap \{ x_n > 0 \}
\end{align}
for some $\rho > 0$, $\sigma \in (0,1)$ and $a \in \R^{n}$ satisfying $A^{1/s}(e_n) a_n = \nabla A^{1/s}(e_n) \cdot a$. Let $\gamma \in (0,s)$. Then, there exist $c > 0$ and $\eps_0 \in (0, \min \{ \rho,c \sigma^{1/\gamma} \} )$, depending only on an upper bound of $|a|$, and on $n,s,\lambda,\Lambda,\gamma$, and $\Vert A \Vert_{C^{1+\gamma}(\mathbb{S}^{n-1})}$, such that if $\eps \le \eps_0$, then
\begin{align*}
A(\nu)\Big(x \cdot \nu - \frac{\sigma\eps \rho}{2}\Big)_+^s \le u(x) \le A(\nu)\left(x \cdot \nu + \frac{\sigma\eps \rho}{2}\right)_+^s ~~ \text{ in } B_{\rho}
\end{align*}
for some $\nu \in \mathbb{S}^{n-1}$ with $|\nu - e_n| \le 4 \eps |a|$.
\end{lemma}

\begin{proof}
We only explain how to prove the lower bound. The proof of the upper bound goes in the same way. Let us define
\begin{align*}
\nu = \frac{e_n + \eps a}{|e_n + \eps a|}.
\end{align*}
Moreover, for $\eps > 0$ small enough, we have that $\nu_n > 0$, and therefore the function $x \mapsto w(x):= A(\nu)(x \cdot \nu - \sigma\eps\rho/2)_+^s$ is strictly increasing in the $e_n$-direction in $B_{2\rho} \cap \{ x \cdot \nu > \sigma\eps\rho/2 \}$.
Let us define
\begin{align*}
\tilde{w}_{\eps}(x) = \frac{x \cdot \nu  - [A(e_n)/A(\nu)]^{1/s} x_n}{\eps \nu_n} - \frac{\sigma\rho}{2 \nu_n},
\end{align*}
and observe that $\tilde{w}_{\eps}(x)$ is the $\eps$-domain variation of $w$ in the $e_n$-direction since
\begin{align*}
w(x - \eps \tilde{w}_{\eps}(x) e_n) = A(\nu)(x \cdot \nu - \eps \tilde{w}_{\eps}(x) \nu_n - \sigma\eps\rho/2)_+^s = A(e_n)(x_n)_+^s.
\end{align*}
Note that in the light of \autoref{lemma:comp-dom-var}, in order to show the desired lower bound, it suffices to show
\begin{align*}
\tilde{w}_{\eps} \le \tilde{u}_{\eps} ~~ \text{ in } B_{\rho + \eps} \cap \{ x_n > 0 \},
\end{align*}
and due to the assumption \eqref{eq:expansion-ass} it suffices to prove
\begin{align*}
\frac{x \cdot \nu - [A(e_n)/A(\nu)]^{1/s} x_n}{\eps \nu_n} - \frac{\sigma \rho}{2 \nu_n} = \tilde{w}_{\eps}(x) \le x \cdot a - \frac{\sigma\rho}{4} ~~ \forall x \in B_{\rho+\eps} \cap \{ x_n > 0 \},
\end{align*}
where we used the definition of $\tilde{w}_{\eps}$. Since $\nu_n \le 1$, it remains to show
\begin{align}
\label{eq:conc-help-1}
x \cdot \frac{\nu}{\nu_n} - \left[ \frac{A^{1/s}(e_n)}{A^{1/s}(\nu)} \right] x \cdot e_n \le x \cdot \eps a + \frac{\eps \sigma \rho}{4} ~~ \forall x \in B_{\rho+\eps} \cap \{ x_n > 0 \}.
\end{align}
Note that $\frac{\nu}{\nu_n} = \frac{e_n + \eps a}{1 + \eps a_n}$ and since $a \cdot A^{1/s}(e_n) e_n = a \cdot \nabla A^{1/s}(e_n)$, and $A^{1/2} \in C^{1+\gamma}(\mathbb{S}^{n-1})$ by \eqref{eq:A-properties}:
\begin{align*}
A^{1/s}(\nu) = A^{1/s}(e_n + \eps a) = A^{1/s}(e_n) + \eps a \cdot \nabla A^{1/s}(e_n) + O(\eps^{1+\gamma}) = A^{1/s}(e_n)(1 + \eps a_n + O(\eps^{1+\gamma})),
\end{align*}
and therefore
\begin{align*}
x \cdot \frac{\nu}{\nu_n} - \left[ \frac{A^{1/s}(e_n)}{A^{1/s}(\nu)} \right] x \cdot e_n  &= x \cdot \frac{e_n + \eps a}{1 + \eps a_n} - \frac{x \cdot e_n}{1 + \eps a_n + O(\eps^{1+\gamma})} \\
&= x \cdot \eps a + x \cdot \eps a \left(\frac{1}{1 + \eps a_n} - 1\right) + x \cdot e_n \left(\frac{1}{1 + \eps a_n} - \frac{1}{1 + \eps a_n + O(\eps^{1+\gamma})} \right).
\end{align*}
Now, we take $\eps > 0$ so small that $|1+ \eps a_n| \ge 1 - \eps |a| \ge 1/2$, i.e., the smallness depends only on an upper bound of $|a|$.   Then
\begin{align*}
\left| \frac{1}{1 + \eps a_n} - 1 \right| = \eps \left| \frac{a_n}{1 + \eps a_n} \right| \le 2|a|\eps, \qquad \left| \frac{1}{1 + \eps a_n} - \frac{1}{1 + \eps a_n + O(\eps^{1+\gamma})} \right| = \left| \frac{O(\eps^{1+\gamma})}{1 + O(\eps)} \right| \le c \eps^{1+\gamma},
\end{align*}
where $c > 0$ depends only on $n,s,\lambda,\Lambda,\gamma$, and $\Vert A \Vert_{C^{1+\gamma}(\mathbb{S}^{n-1})}$. Therefore, using also that $|x| \le \rho + \eps \le 2 \rho$, it follows
\begin{align*}
x \cdot \frac{\nu}{\nu_n} - \left[ \frac{A^{1/s}(e_n)}{A^{1/s}(\nu)} \right] x \cdot e_n &\le x \cdot \eps a + c(1 + |a|) \rho \eps^{1+\gamma}.
\end{align*}
This implies \eqref{eq:conc-help-1}, as desired, as long as $\eps^{\gamma} \le c \sigma$ for some $c > 0$ depending only on $n,s,\lambda,\Lambda,\gamma$, $\Vert A \Vert_{C^{1+\gamma}(\mathbb{S}^{n-1})}$, and an upper bound on $|a|$.

Let us finish the proof by showing that $|\nu - e_n| \le 4 \eps |a|$.
To see this, we estimate
\begin{align*}
|\nu - e_n|^2 = |\nu|^2 + |e_n|^2 - 2 \nu_n &= 2\left(1 - \frac{1 + \eps a_n}{\sqrt{(1 + \eps a_n)^2 + |\eps a \cdot e'|^2}} \right) \\
&= 2\left(1 - \frac{1}{\sqrt{1 + \left(\frac{|\eps a \cdot e'|}{|1 + \eps a_n|} \right)^2}} \right) \le 4 \left(\frac{|\eps a \cdot e'|}{|1 + \eps a_n|} \right)^2 \le 16 \eps^2 |a|^2,
\end{align*}
where we used again that $|1+ \eps a_n| \ge 1/2$, and we applied the algebraic inequality $1 - \frac{1}{\sqrt{1 + x}} \le 2x$, which holds true whenever $0 \le x \le 1/2$. In our case, this condition is verified once $\eps > 0$ is small enough, depending on an upper bound of $|a|$.
\end{proof}

We are now in a position to give the proof of the improvement of flatness:

\begin{proof}[Proof of \autoref{thm:improvement-of-flatness}]
Let us assume by contradiction that there exists a sequence $\eps_k \searrow 0$, and a sequence $(u_k)_k$ of viscosity solutions to the nonlocal one-phase problem for $K$ in $B_2$ with $0 \in \partial \{ u_k > 0\}$, such that
\begin{align*}
A(e_n)(x_n - \eps_k)_+^s \le u_k(x) \le A(e_n)(x_n + \eps_k)_+^s ~~ \text{ in } B_1,
\end{align*}
and moreover,
\begin{align*}
T_k := \tail([u_k-A(e_n)(x_n -\eps_k)_+^s]_- ; 1) + \tail([A(e_n)(x_n + \eps_k)_+^s - u_k]_- ; 1) \le \eps_k \delta_0,
\end{align*}
but so that the conclusion of the theorem does not hold true. Then, we can apply \autoref{lemma:linearized-problem} and deduce that the graphs $\Gamma_k$ of the domain variations $\tilde{u}_k$ of $u_k$ converge in the Hausdorff distance in $B_{1/2} \cap \{x_n \ge 0\}$ to the graph $\Gamma$ of a H\"older continuous function $\tilde{u}$, which solves in the viscosity sense
\begin{align*}
\begin{cases}
L((x_n)_+^{s-1} \tilde{u} \1_{B_{3/4}}) &= f ~~ \text{ in } B_{1/2} \cap \{ x_n > 0 \},\\
\partial_{\omega} \tilde{u} &= 0 ~~ \text{ on } B_{1/2} \cap \{ x_n = 0 \},
\end{cases}
\end{align*}
where
\begin{align*}
\omega = \frac{A^{1/s}(e_n)e_n - \nabla A^{1/s}(e_n)}{|(A^{1/s}(e_n)e_n - \nabla A^{1/s}(e_n))|} \in \mathbb{S}^{n-1}.
\end{align*}
Clearly, by the regularity of $A$ (see \autoref{prop:free-bound-cond}) there exists $\delta > 0$, depending only on $n,s,\lambda,\Lambda$, and $\Vert A \Vert_{C^{1+\beta}(\mathbb{S}^{n-1})}$, such that $\omega_n \ge \delta$.\\
Due to \autoref{lemma:reg-linearized}, and since $\tilde{u}(0) = 0$, and $| \tilde{u} | \le 1$ (and by \autoref{lemma:linearized-problem}(iv)) we deduce that there exists $a \in \R^{n}$ with $a \cdot \omega = 0$, and $|a| \le c$ for some constant $c > 0$, depending only on $n,s,\lambda,\Lambda,\gamma$, such that for any $\rho \in (0,1/2)$, and $\sigma \in [c_0\rho^{\gamma},1)$ for some $c_0 > 0$:
\begin{align*}
|\tilde{u}(x) - a \cdot x | \le \sigma\frac{\rho}{8} ~~ \text{ in } B_{2\rho} \cap \{ x_n > 0 \}.
\end{align*}
At this point $\eps_0$ depends on a lower bound for $\sigma$, and on $\rho$, so we need to make sure that we will not use this estimate for arbitrarily small $\sigma$ and $\rho$.
Therefore, by \autoref{lemma:linearized-problem}(ii), we get that for large enough $k$:
\begin{align*}
a \cdot x - \sigma\frac{\rho}{4} \le \tilde{u}_k(x) \le a \cdot x + \sigma\frac{\rho}{4} ~~ \text{ in } B_{2\rho} \cap \{ x_n > 0 \}.
\end{align*}
Thus, we are in a position to apply \autoref{lemma:conclusion-iof}, which yields that for any $\rho \in (0,1/2)$, and $\sigma \in [c_0\rho^{\gamma},1)$:
\begin{align}
\label{eq:iof-help-0}
A(\nu)(x \cdot \nu - \sigma\eps_k \rho /2)_+^s \le u_k(x) \le A(\nu)(x \cdot \nu + \sigma\eps_k \rho /2)_+^s ~~ \text{ in } B_{\rho},
\end{align}
where $\nu \in \mathbb{S}^{n-1}$ with $|\nu - e_n| \le c \eps_k$. 
In particular, we get for some $\rho_0\le 1/2$ to be chosen so small that $c_0 \rho_0^{\gamma} \le 1$, and upon choosing $\sigma = c_0 \rho_0^{\gamma}$:
\begin{align*}
A(\nu)(x \cdot \nu - \eps_k/2)_+^s \le (u_k)_{\rho_0}(x) \le A(\nu)(x \cdot \nu + \eps_k/2)_+^s ~~ \text{ in } B_{1},
\end{align*}
which proves the first claim of the theorem for $(u_k)_{\rho_0}$.\\
In order to get a contradiction, and thereby to conclude the proof, it remains to verify the tail estimate for $(u_k)_{\rho_0}$. To do so, we compute
\begin{align*}
\tail(|(u_k)_{\rho_0} - A(\nu)(x \cdot \nu)_+^s|; 1) &= \int_{\R^n \setminus B_1} \left| \frac{u_k(\rho_0 x) - A(\nu)(\rho_0 x \cdot \nu)_+^s}{\rho_0^s} \right| |x|^{-n-2s} \d x \\
&= \rho_0^{s} \int_{\R^n \setminus B_{\rho_0}} |u_k(x) - A(\nu)(x \cdot \nu)_+^s| |x|^{-n-2s} \d x.
\end{align*}

We will estimate this integral by splitting  it into three parts: 
\begin{align*}
\R^n \setminus B_{\rho_0} = (B_{1/2} \setminus B_{\rho_0}) \cup (B_1 \setminus B_{1/2}) \cup (\R^n \setminus B_1).
\end{align*}
To estimate the first part, we apply \eqref{eq:iof-help-0} with $\rho := |x| \in [\rho_0,1/2)$ and $\sigma = c_0 |x|^{\gamma}$:
\begin{align*}
\rho_0^{s} \int_{B_{1/2} \setminus B_{\rho_0}} & |u_k(x) - A(\nu)(x \cdot \nu)_+^s| |x|^{-n-2s} \d x \\
&\le  \rho_0^{s} \int_{B_{1/2} \setminus B_{\rho_0}} [ A(\nu)(x \cdot \nu + c\eps_k |x|^{1+\gamma})_+^s - A(\nu)(x \cdot \nu)_+^s] |x|^{-n-2s} \d x \\
&\quad + \rho_0^{s} \int_{B_{1/2} \setminus B_{\rho_0}} [A(\nu)(x \cdot \nu)_+^s - A(\nu)(x \cdot \nu - c\eps_k |x|^{1+\gamma})_+^s] |x|^{-n-2s} \d x.
\end{align*}
Note that for these integrals, we can compute in a similar way as in \eqref{eq:L1-exponent-gain}, if $\gamma \in (0,s)$:
\begin{align*}
 \rho_0^{s} & \int_{\R^n \setminus B_{\rho_0}} [(x \cdot \nu + c\eps_k |x|^{1+\gamma})_+^s - (x \cdot \nu)_+^s] |x|^{-n-2s} \d x \\
 &\le c \rho_0^{-n-s} \sum_{i = 1}^{\infty} 2^{-i(n+2s)} \int_{B_{2^{i+1} \rho_0} \setminus B_{2^i \rho_0}} [(x \cdot \nu + c\eps_k |x|^{1+\gamma})_+^s - (x \cdot \nu)_+^s] \d x \\
 &\le c \rho_0^{-1-s} \sum_{i = 1}^{\infty} 2^{-i(1+2s)} \int_{-2^{i+1} \rho_0}^{2^{i+1} \rho_0}[(x + c\eps_k (2^{i+1} \rho_0)^{1+\gamma})_+^s - (x)_+^s] \d x \\
 &\le c \rho_0^{-1-s} \sum_{i = 1}^{\infty} 2^{-i(1+2s)} [(2^{i+1} \rho_0)^s \eps_k (2^i\rho_0)^{1+\gamma}] \\
 &\le c \rho_0^{\gamma} \eps_k \left(\sum_{i = 1}^{\infty} 2^{-i(s-\gamma)} \right) \le c \rho_0^{\gamma} \eps_k,
\end{align*}
and analogously for the other term. Moreover, using that $|\nu - e_n| \le c \eps_k$ (see \autoref{lemma:conclusion-iof}), as well as the assumptions \eqref{eq:eps-flat-iof} and \eqref{eq:tail-smallness-iof}, we obtain for the second part

\begin{align*}
\rho_0^{s} \int_{B_{1} \setminus B_{1/2}} & |u_k(x) - A(\nu)(x \cdot \nu)_+^s| |x|^{-n-2s} \d x \\
&\le \rho_0^{s} \int_{B_{1} \setminus B_{1/2}} |u_k(x) - A(e_n)(x_n)_+^s| |x|^{-n-2s} \d x \\
&\quad + \rho_0^{s} \int_{B_{1} \setminus B_{1/2}} |A(e_n)(x_n)_+^s - A(\nu)(x \cdot \nu)_+^s| |x|^{-n-2s} \d x\\
&=: I_1 + I_2,
\end{align*}

and for the third part
\begin{align*}
\rho_0^{s} \int_{\R^n \setminus B_{1}} & |u_k(x) - A(\nu)(x \cdot \nu)_+^s| |x|^{-n-2s} \d x \\
&\le \rho_0^s T_{\eps_k} + c\rho_0^s \int_{\R^n \setminus B_{1}} |A(e_n)(x \cdot e_n)_+^s - A(\nu)(x \cdot \nu)_+^s| |x|^{-n-2s} \d x =: I_3 + I_4.
\end{align*}
For $I_1$, we use \eqref{eq:eps-flat-iof} and a similar computation as above, to deduce
\begin{align*}
I_1 &\le A(e_n)\rho_0^s \int_{B_1 \setminus B_{1/2}} |(x_n + \eps_k)_+^s - (x_n)_+^s| |x|^{-n-2s} \d x \\
&\quad + A(e_n)\rho_0^s \int_{B_1 \setminus B_{1/2}} |(x_n - \eps_k)_+^s - (x_n)_+^s| |x|^{-n-2s} \d x \\
&\le c \rho_0^s \eps_k.
\end{align*}
Moreover, by \eqref{eq:tail-smallness-iof}, we have $I_3 \le \rho_0^s \eps_k \delta_0$. For $I_2$ and $I_4$, we compute
\begin{align*}
I_2 + I_4 &= c\rho_0^s \int_{\R^n \setminus B_{1/2}} |A(e_n)(x_n)_+^s - A(\nu)(x \cdot \nu)_+^s| |x|^{-n-2s} \d x \\
&\le \rho_0^s \int_{\R^n \setminus B_{1/2}} |A(e_n)-A(\nu)| |x|^{-n-s} \d x + c\rho_0^s \int_{\R^n \setminus B_{1/2}} |(x_n)_+^s - (x \cdot \nu)_+^s| |x|^{-n-2s} \d x \\
&\le c \rho_0^s \eps_k + c\rho_0^s \int_{\R^n \setminus B_{1/2}} |(\cos(x,e_n))_+^s - (\cos(x,\nu))_+^s| |x|^{-n-s} \d x \\
&= c \rho_0^s \eps_k + c\rho_0^s \int_{1/2}^{\infty} \left( \int_{\mathbb{S}^{n-1}} |(\cos(\theta,e_n))_+^s - (\cos(\theta,\nu))_+^s| \d \theta \right)  r^{-1-s} \d r\\
&\le c \rho_0^s \eps_k + c\rho_0^s \int_{1/2}^{\infty} \eps_k r^{-1-s} \d r \le  c\rho_0^s \eps_k.
\end{align*}
where we used the Lipschitz continuity of $e \mapsto A(e)$ (see \eqref{eq:A-properties}), and that
\begin{align}
\label{eq:cos-estimate}
\int_{\mathbb{S}^{n-1}} |(\cos(\theta,e_n))_+^s - (\cos(\theta,\nu))_+^s| \d \theta \le c \eps_k.
\end{align}
Before explaining how to show \eqref{eq:cos-estimate}, note that a combination of all previous estimates yields 
\begin{align*}
\tail(|(u_k)_{\rho_0} - (x \cdot \nu)_+^s|); 1) \le c \rho_0^{\gamma} + I_1 + I_2 + I_3 + I_4 \le c (\rho_0^{\gamma} + \rho_0^s)\eps_k.
\end{align*}
Therefore, upon choosing $\rho_0$ so small that $c (\rho_0^{\gamma} + \rho_0^s) \le \delta_0/2$, we get the desired estimate for the tail of $(u_k)_{\rho_0}$. Note that the smallness of $\rho_0$ only depends on $n,s,\lambda,\Lambda,\delta_0,\gamma$, and $\Vert A \Vert_{C^{1+\gamma}(\mathbb{S}^{n-1})}$, which implies that $\eps_0$ from \autoref{lemma:conclusion-iof} only depends on these quantities.

Finally, we give more details on the proof of \eqref{eq:cos-estimate}. First, note that we can restrict the domain of integration to the two-dimensional hyper-surface $H$ that is spanned by $e_n,\nu$ and contained in $\mathbb{S}^{n-1}$.
Let $e_1,\nu_1,e_2,\nu_2,a \in H$ be such that 
\begin{align*}
0 = \cos(e_1,e_n) < \cos(\nu_1,e_n) < \cos(a,e_n) < \cos(e_n,e_n) = 1 > \cos(\nu,e_n) > \cos(e_2,e_n) = 0 > \cos(\nu_2,e_n),\\
\cos(e_1,\nu) < 0 = \cos(\nu_1,\nu) < \cos(e_n,\nu) < \cos(a,\nu)  < \cos(\nu,\nu) = 1 > \cos(e_2,\nu) > \cos(\nu_2,\nu) = 0,
\end{align*}
and $\cos(a,e_n) = \cos(a,\nu)$.  Note that since $|e_n - \nu| \le c \eps_k$, we also have $|e_1 - \nu_1| + |e_2 - \nu_2| \le c \eps_k$.\\
In the following, for $e \in H$, we will also denote the angle between $e$ and $e_1$ by $e$. This will be useful in the parametrization of the integrals. Then, we compute
\begin{align*}
\int_{\mathbb{S}^{n-1}} |(\cos(\theta,e_n))_+^s - (\cos(\theta,\nu))_+^s| \d \theta &\le c \int_{H} |(\cos(\theta,e_n))_+^s - (\cos(\theta,\nu))_+^s| \d \theta\\
&\le c \int_{e_1}^{\nu_2} |(\cos(\theta,e_n))_+^s - (\cos(\theta,\nu))_+^s| \d \theta\\
&= c\int_{e_1}^{a} (\cos(\theta,e_n))_+^s - (\cos(\theta,\nu))_+^s \d \theta \\
&\quad + c\int_a^{\nu_2} (\cos(\theta,\nu))_+^s - (\cos(\theta,e_n))_+^s \d \theta\\
&=: cJ_1 + cJ_2.
\end{align*}
For $J_1$, we compute
\begin{align*}
J_1 &= \int_{e_1}^{a} \cos(\theta,e_n)^s \d \theta - \int_{\nu_1}^a \cos(\theta,\nu)^s \d \theta = \int_{e_1}^{a} \cos(\theta,e_n)^s \d \theta - \int_{e_1}^{a + (\nu_1 - e_1)} \cos(\theta,e_n)^s \d \theta \\
&= \int_{a}^{a + (\nu_1 - e_1)} \cos(\theta,e_n)^s \d \theta \le |\nu_1 - e_1| \le c \eps_k,
\end{align*}
where we used $\cos(\theta,e_n)^s \le 1$. Similarly, for $J_2$, we obtain $J_2 \le c \eps_k$. This implies \eqref{eq:cos-estimate}, as desired.
\end{proof}

\subsection{Conclusion of the proof}
\label{subsec:iterate-improvement-of-flatness}

An iterative application of the improvement of flatness (see \autoref{thm:improvement-of-flatness}) implies smoothness of the free boundary near all points satisfying the flatness assumption in \autoref{thm:free-boundary-regularity}. The underlying argument is by now standard in the literature for local problems, however, here we give a detailed proof for anisotropic nonlocal problems following closely \cite{Vel23}, taking into account nonlocal tail term.

First, we establish the uniqueness of the blow-up limit and the decay rate of the blow-up sequence near free boundary points at which the viscosity solution is $\eps$-flat.

\begin{lemma}
\label{lemma:uniqueness-of-blowup}
Let $K \in C^{1-2s+\beta}(\mathbb{S}^{n-1})$ for some $\beta > \max\{0,2s-1\}$ and assume \eqref{eq:Kcomp}. Let $u$ be a viscosity solution to the nonlocal one-phase problem for $K$ in $B_2$. Then, there are $\eps_0, \delta_0 \in (0,1)$ and $c > 0$, depending only on $K$
such that if $\eps \in (0,\eps_0)$ is such that the following hold
\begin{align}
\label{eq:flat-iof2}
A(e_n)(x \cdot e_n -\eps)_+^s \le u(x) \le A(e_n)(x \cdot e_n + \eps)_+^s ~~ \forall x \in B_1,
\end{align}
and
\begin{align}
\label{eq:tail-smallness-iof2}
T_{\eps} := \tail([u-A(e_n)(x_n -\eps)_+^s]_- ; 1) + \tail([A(e_n)(x_n + \eps)_+^s - u]_- ; 1) \le \eps \delta_0,
\end{align}
then, for every $x_0 \in \partial \{ u > 0 \} \cap B_{\eps \delta_0}$, there is a unique $\nu_{x_0} \in \mathbb{S}^{n-1}$ such that for any $r \le \frac{1}{2}$:
\begin{align}
\label{eq:rate-of-convergence} 
\tail(|u_{r,x_0} - u_{x_0}|;1) + \Vert u_{r,x_0} - u_{x_0} \Vert_{L^{\infty}(B_1)} \le c r^{\gamma}, \quad \text{ where } \quad  u_{x_0}(x) = A(\nu_{x_0})(x \cdot \nu_{x_0})_+^s.
\end{align}
Here, $\gamma \in (0,1)$ is such that $2 = \rho_0^{-\gamma / s}$, where $\rho_0$ is the constant from \autoref{thm:improvement-of-flatness}.\\
Moreover, for any $y \in \overline{\{ u > 0 \}} \cap B_{1/2}$ it holds 
\begin{align}
\label{eq:non-deg-visc}
\Vert u \Vert_{L^{\infty}(B_r(y))} \ge c^{-1} r^s ~~ \forall r \in (0,1/2).
\end{align}
\end{lemma}

Note that \eqref{eq:non-deg-visc} is immediate for minimizers of $\cI$ due to the non-degeneracy in \autoref{thm:non-degeneracy}. However, since viscosity solutions are not non-degenerate in general, \eqref{eq:non-deg-visc} is an important result.

\begin{proof}
By scaling we have that $u_{r,x_0}$ is a viscosity solution to the nonlocal one-phase problem for $K$ in $B_2$ with $0 \in \partial \{ u_{r,x_0} > 0 \}$ for any $r \le 1/2$. Moreover, by \eqref{eq:flat-iof2}, we have
\begin{align*}
u_{1/2,x_0}(x) \le A(e_n)(x \cdot e_n + 2 x_0 \cdot e_n + 2\eps)_+^s \le A(e_n) (x \cdot e_n + 4\eps)_+^s ~~ \forall x \in B_1,
\end{align*}
and an analogous argument yields a corresponding lower bound. Moreover, by \eqref{eq:tail-smallness-iof2} and a computation similar to the one in \eqref{eq:L1-exponent-gain}, using that $|x_0| \le \eps \delta_0$, we deduce that for some constant $c_1 > 0$:
\begin{align*}
\tail([u_{1/2,x_0}-A(e_n)(x_n)_+^s]_- ; 1) &\le \tail([u-A(e_n)(x_n)_+^s]_- ; 1) \\
&\quad + A(e_n)\tail([(x_n - 2 x_0 \cdot e_n )_+^s -(x_n)_+^s]_- ; 1) \le c_1 \eps \delta_0.
\end{align*}
An analogous reasoning allows to estimate $\tail([A(e_n)(x_n)_+^s - u_{1/2,x_0}]_- ; 1) \le c_1 \eps \delta_0$.
Therefore, upon choosing $\delta_0$ as in \autoref{thm:improvement-of-flatness} and $\eps < \eps_0/c_1$, where $\eps_0$ is as in \autoref{thm:improvement-of-flatness}, we can apply \autoref{thm:improvement-of-flatness} to $u_{1/2,x_0}$, and then, iteratively to $u_{\rho_0^k/2,x_0} =: u_k$ for any $k \in \N$, where $\rho_0 > 0$ is the constant from \autoref{thm:improvement-of-flatness}. \\
Upon taking $\eps > 0$ even smaller, if necessary, this yields the existence of $\nu_k \in \mathbb{S}^{n-1}$ such that 
\begin{align*}
|\nu_k - \nu_{k+1}| \le C \eps_0 2^{-k}, \qquad A(\nu_k)(x \cdot \nu_k - \eps_0 2^{-k})_+^s \le u_k(x) \le A(\nu_k)(x \cdot \nu_k + \eps_0 2^{-k})_+^s ~~ \forall x \in B_1,
\end{align*}
and 
\begin{align*}
\tail([u_k -A(\nu_k)(x \cdot \nu_k -\eps_0 2^{-k})_+^s]_- ; 1) + \tail([A(\nu_k)(x \cdot \nu_k + \eps_0 2^{-k})_+^s - u_{k}]_- ; 1) \le 2^{-k} \eps_0 \delta_0.
\end{align*}
Note that the sequence $(\nu_k)_k$ is a Cauchy sequence since for any $1 < k < l$ it holds
\begin{align*}
|\nu_k - \nu_l| \le \sum_{i = k}^{l-1} |\nu_{i} - \nu_{i+1}| \le C \eps_0 \sum_{i = k}^{l-1} 2^{-i} \le 2 C \eps_0 2^{-k}.
\end{align*}
Thus, there exists $\nu_{x_0} \in \mathbb{S}^{n-1}$ such that $\nu_k \to \nu_{x_0}$ and $|\nu_k - \nu_{x_0}| \le 2 C \eps_0 2^{-k}$ for any $k \in \N$. Let us denote $u_{x_0}(x) = A(\nu_{x_0})(x \cdot \nu_{x_0})_+^s$. Then, we deduce that for any $x \in B_1$, using also the regularity of $[\nu \mapsto A(\nu)] \in C^{1,\beta}(\mathbb{S}^{n-1})$ from \eqref{eq:A-properties},
\begin{align}
\label{eq:uniqueness-help-1}
\begin{split}
|u_{x_0}(x) - u_{k}(x)| &\le |A(\nu_{x_0})(x \cdot \nu_{x_0})_+^s - A(\nu_k)(x \cdot \nu_k + \eps_0 2^{-k})_+^s| \\
& \quad + |A(\nu_{x_0})(x \cdot \nu_{x_0})_+^s - A(\nu_k)(x \cdot \nu_k - \eps_0 2^{-k})_+^s| \le c 2^{-ks}
\end{split}
\end{align}
for some $c > 0$, depending also on $n,s,\eps_0,\rho_0$, and $\Vert A \Vert_{C^{1+\beta}(\mathbb{S}^{n-1})}$. To conclude the proof, let us now fix $r \le 1/2$ and take $k \in \N$ such that $\rho_0^{k+1}/2 \le r \le \rho_0^{k}/2$. Observe that $u_{r,x_0}(x) = \left(\frac{2r}{\rho_0^k}\right)^{-s} u_{k}\left( \frac{2r}{\rho_0^k} x \right) = (u_k)_{\frac{2r}{\rho_0^k}}(x)$, and therefore we have
\begin{align*}
 A(\nu_k)\left(x \cdot \nu_k - \eps_0 2^{-k} \frac{\rho_0^k}{2r} \right)_+^s \le u_{r,x_0}(x) \le A(\nu_k) \left(x \cdot \nu_k + \eps_0 2^{-k} \frac{\rho_0^k}{2r} \right)_+^s ~~ \forall x \in B_1.
\end{align*} 
Thus, and using also that $\frac{\rho_0^{k}}{2r} \le \rho_0^{-1}$, we deduce for any $x \in B_1$: 
\begin{align*}
|u_{r,x_0}(x) - u_k(x) | \le | u_{r,x_0}(x) - A(\nu_k)(x \cdot \nu_k + \eps_0 2^{-k})_+^s | + | u_{r,x_0} - A(\nu_k)(x \cdot \nu_k - \eps_0 2^{-k})_+^s | \le c 2^{-ks}
\end{align*}
for some $c > 0$, depending also on $\rho_0$, but not on $k$.
Let us choose $\gamma > 0$ such that $\rho_0^{-\gamma / s} = 2$. Then, combining the previous estimate with \eqref{eq:uniqueness-help-1}, we deduce
\begin{align*}
\Vert u_{r,x_0} - u_{x_0} \Vert_{L^{\infty}(B_1)} \le c 2^{-ks} \le c \rho_0^{\gamma k} \le c r^{\gamma},
\end{align*}
as desired.\\
Note that by a similar argument based on the triangle inequality, but using the estimate for the tail of $u_k$ instead of the pointwise estimate in $B_1$, we can also prove that
\begin{align*}
\tail(|u_{r,x_0} - u_{x_0}|; 1) \le c 2^{-ks} \le c r^{\gamma}
\end{align*}
where $c > 0$, depends in addition also on $\delta_0$.
It remains to establish the non-degeneracy \eqref{eq:non-deg-visc}. It follows from \eqref{eq:rate-of-convergence}, as we will prove next.\\
First, we deduce from \eqref{eq:rate-of-convergence} that for any $x_0 \in \partial \{ u > 0 \}$ and $r \in (0,1)$
\begin{align}
\label{eq:visc-nondeg-help}
-cr^{s+\gamma} \le u(x_0 + rx) - A(\nu_{x_0})(rx \cdot \nu_{x_0})_+^s \le c r^{s+\gamma} ~~ \forall x \in B_1.
\end{align}
Consequently, 
\begin{align}
\label{eq:visc-nondeg-help-1}
\Vert u \Vert_{L^{\infty}(B_r(x_0))} \ge \Vert  A(\nu_{x_0})(rx \cdot \nu_{x_0})_+^s \Vert_{L^{\infty}(B_1)} - c r^{s+\gamma} \ge c_1 r^s - c r^{s+\gamma} \ge c r^s,
\end{align}
if $r \in (0,1)$ is small enough. 
Moreover, note that for any $x \in B_{1/2}$, we have
\begin{align*}
u(x) \ge c \dist(x ,\partial \{ u > 0 \})^s.
\end{align*}
Indeed, given $x \in B_{1/2}$, let us denote by $x_0$ the projection of $x$ to $\partial \{ u > 0 \}$, i.e, it holds $(x-x_0) \cdot \nu_{x_0} = |x-x_0| = \dist(x,\partial\{ u > 0 \})$, where we also used \eqref{eq:rate-of-convergence}. Then, from \eqref{eq:visc-nondeg-help} applied with $x_0 := x_0$, $r := \dist(x,\partial\{ u > 0\})$, and $x := (x - x_0)/r = \nu_{x_0}$ we deduce
\begin{align}
\label{eq:visc-nondeg-help-2}
u(x) \ge c \dist(x,\partial\{ u > 0 \})^s - c \dist(x,\partial\{ u > 0 \})^{s+\gamma} \ge c \dist(x,\partial\{ u > 0 \})^s,
\end{align}
if $x$ is close enough to $\partial \{ u > 0 \}$. A combination of \eqref{eq:visc-nondeg-help-1} and \eqref{eq:visc-nondeg-help-2} in the same way as in the proof of \cite[Proof of Theorem 4.8]{RoWe24b} yields \eqref{eq:non-deg-visc}, as we claimed. Indeed, if $y \in \{ u > 0 \}$ satisfies $\dist(y,\partial \{ u > 0 \}) \ge r/2$, then \eqref{eq:non-deg-visc} follows directly from \eqref{eq:visc-nondeg-help-2}. However, if $\dist(y,\partial \{ u > 0 \}) < r/2$, we apply \eqref{eq:visc-nondeg-help-1} to $B_{r/2}(x_0) \subset B_r(y)$, where $x_0 \in \partial \{ u > 0 \}$ is the projection of $y$ onto $\partial \{ u > 0 \}$.
\end{proof}

We are now in a position to establish \autoref{thm:free-boundary-regularity}. Note that this proof requires $u \in C^s$, and non-degeneracy in the sense of \eqref{eq:non-deg-visc}. This is the only time we need the $C^s$ regularity for viscosity solutions, which was proved in \autoref{lemma:visc-Cs}.

\begin{proof}[Proof of \autoref{thm:free-boundary-regularity}]
The proof is split into three steps.\\
Step 1: First, we claim that for every $x_0 \in \partial \{ u > 0 \} \cap B_{\eps \delta_0}$ the free boundary is flat near $x_0$, i.e., that there are $C > 0$ and $r_0 > 0$ such that for any $r \le r_0$ it holds
\begin{align}
\label{eq:flat-implies-smooth-claim-1}
\{ x \cdot \nu_{x_0} > C r^{\gamma/s} \} \cap B_1 \subset \{ u_{r,x_0} > 0\} \cap B_1, \qquad   \{ u_{r,x_0} > 0\} \cap \{ x \cdot \nu_{x_0} < - C r^{\gamma/s} \} \cap B_1 = \emptyset,
\end{align}
where $\eps_0,\delta_0,\gamma$, and $\nu_{x_0} \in \mathbb{S}^{n-1}$ are as in \autoref{lemma:uniqueness-of-blowup} and we assume that $\eps < \eps_0$.
The first inclusion follows immediately from \eqref{eq:rate-of-convergence} in \autoref{lemma:uniqueness-of-blowup}, which is applicable by assumption, and which (by the nonnegativity of $u$) implies that for $r \le 1/2$ and some $c > 0$:
\begin{align*}
u_{r,x_0}(x) \ge \left[A(\nu_{x_0})(x \cdot \nu_{x_0})_+^s -c r^{\gamma} \right]_+ ~~ \forall x \in B_1.
\end{align*}
We prove the second inclusion in \eqref{eq:flat-implies-smooth-claim-1} by contradiction. Let us assume that there is $y \in B_1$ such that $u_{r,x_0}(y) > 0$ and $y \cdot \nu_{x_0} < - C r^{\gamma/s}$. Then, by the non-degeneracy of $u$ (see \eqref{eq:non-deg-visc}), we deduce that for $\rho = C r^{\gamma/s}/2$
\begin{align}
\label{eq:flat-implies-smooth-help-1}
\Vert u_{2r,x_0} - u_{x_0} \Vert_{L^{\infty}(B_{\rho}(y/2))} = \Vert u_{2r,x_0} \Vert_{L^{\infty}(B_{\rho}(y/2))} \ge c( - r^{\gamma} + \rho^s) \ge c_1 (C/2)^{s} r^{\gamma}
\end{align}
for some $c_1 > 0$, where we used that $u_{x_0} \equiv 0$ in $B_{\rho}(y/2)$ by construction. Thus, if we choose $r_0 > 0$ so small that $C r_0^{\gamma} \le 1$, we deduce from \autoref{lemma:uniqueness-of-blowup}
\begin{align*}
\Vert u_{2r,x_0} - u_{x_0} \Vert_{L^{\infty}(B_{\rho}(y/2))} \le \Vert u_{2r,x_0} - u_{x_0} \Vert_{L^{\infty}(B_1)} \le c_2 r^{\gamma}
\end{align*}
for some $c_2 > 0$, which contradicts \eqref{eq:flat-implies-smooth-help-1} upon choosing $C > 0$ large enough. This proves \eqref{eq:flat-implies-smooth-claim-1}.

Step 2: Next, we claim that the map $x_0 \mapsto \nu_{x_0}$ is H\"older continuous in $\partial \{ u > 0 \} \cap B_{\rho}$ for $\rho := \eps \delta_0 \in (0,1)$, i.e., that there are $c > 0$ and $\alpha \in (0,s)$ such that
\begin{align}
\label{eq:flat-implies-smooth-claim-2}
|\nu_{x_0} - \nu_{y_0}| \le c |x_0 - y_0|^{\alpha} ~~ \forall x_0, y_0 \in \{ u > 0 \} \cap B_{\rho}.
\end{align}
To see this, we observe first that as an easy consequence of the $C^s$ regularity (see \autoref{lemma:visc-Cs}, using that $0 \in \partial \{ u > 0 \}$), it holds for $r \in (0,1)$ such that $r^s := |x_0 - y_0|^{s - \alpha}$:
\begin{align*}
\Vert u_{r,x_0} - u_{r,y_0}\Vert_{L^{\infty}(B_1)} &= r^{-s} \Vert u(x_0 + r  \cdot) - u(y_0 + r \cdot) \Vert_{L^{\infty}(B_1)} \\
&\le \Vert u \Vert_{C^{s}(B_1)} r^{-s}|x_0 - y_0|^s \le c |x_0 - y_0|^{\alpha},
\end{align*}
once $x_0,y_0 \in \{ u > 0 \} \cap B_{\eps \delta_0}$ and $r \le 1/2$.
Combining this estimate with \autoref{lemma:uniqueness-of-blowup}, which is applicable by assumption, and setting $\alpha = \frac{s\gamma}{s+\gamma}$, such that $r^{\gamma} = |x_0 - y_0|^{\alpha}$, we deduce
\begin{align*}
\Vert u_{x_0} - u_{y_0} \Vert_{L^{\infty}(B_1)} \le \Vert u_{x_0} - u_{r,x_0} \Vert_{L^{\infty}(B_1)}  + \Vert u_{r,x_0} - u_{r,y_0} \Vert_{L^{\infty}(B_1)}  + \Vert u_{r,y_0} - u_{y_0} \Vert_{L^{\infty}(B_1)} \le c |x_0 - y_0|^{\alpha}.
\end{align*}
Note that $r \le \frac{1}{2}$, which we required for the previous argument, follows directly from $(2\eps \delta_0)^{\frac{1}{s + \gamma}} \le \frac{1}{2}$, which can be achieved upon choosing $\eps > 0$ smaller, if necessary. From the previous estimate, we immediately deduce \eqref{eq:flat-implies-smooth-claim-2} by applying the following algebraic inequality with $v_1 = \nu_{x_0}$, $v_2 = \nu_{y_0}$:
\begin{align*}
|v_1 - v_2| \le c \Vert (v_1 \cdot x)_+^s - (v_2 \cdot x)_+^s \Vert_{L^{\infty}(B_1)} ~~ \forall v_1, v_2 \in \mathbb{S}^{n-1}.
\end{align*}
This algebraic inequality in turn follows from the corresponding one for $s = 1$ (see \cite[page 76]{Vel23}) and the Lipschitz regularity of $t \mapsto t^{1/s}$.

Step 3: Having at hand \eqref{eq:flat-implies-smooth-claim-1} and \eqref{eq:flat-implies-smooth-claim-2}, we can now conclude the proof by following the lines of \cite[Proposition 8.6]{Vel23}. Indeed, \eqref{eq:flat-implies-smooth-claim-1} yields that for any $\delta > 0$, there is $R > 0$ such that 
\begin{align}
\label{eq:unif-cone-1}
\begin{cases}
u > 0 ~~ \text{ in } \mathcal{C}^+_{\delta}(x_0,\nu_{x_0}) \cap B_R(x_0),\\
u = 0 ~~ \text{ in } \mathcal{C}^-_{\delta}(x_0,\nu_{x_0}) \cap B_R(x_0),
\end{cases}
 ~~ \forall x_0 \in \partial \{ u > 0 \} \cap B_{R},
\end{align} 
where we choose $R > 0$ such that $C R^{\gamma/s} \le \delta$, and define
\begin{align*}
\mathcal{C}^{\pm}_{\delta}(x_0,\nu_{x_0}) := \{ x \in \R^n :  \pm \nu_{x_0} \cdot (x - x_0) > \delta |x-x_0|\}.
\end{align*}
In fact, if $x \in \mathcal{C}^{\pm}_{\delta}(x_0,\nu_{x_0}) \cap B_R(x_0)$, then we have
\begin{align*}
\pm(x-x_0) \cdot \nu_{x_0} > \delta \ge C R^{\gamma/s} |x-x_0| \ge C |x-x_0|^{\gamma/s},
\end{align*}
and by \eqref{eq:flat-implies-smooth-claim-1} we have $u(x) = u_{|x-x_0|,x_0}\left(\frac{x-x_0}{|x-x_0|}\right) |x-x_0|^{-s} > 0$ (resp. $= 0$), as desired.\\
Let us now assume without loss of generality that $\nu_{0} = e_n$.
As a consequence, setting $\rho = R \sqrt{1 - \delta^2}$ for some $\delta \in (0,1)$ to be chosen small enough later, it turns out that the function
\begin{align*}
g(x') = \inf \{ t \in \R : u(x',\tau) > 0 ~~ \forall \tau \in (t,\rho) \}
\end{align*}
is well defined for any $x'\in B_{\rho}'$. Upon choosing $\rho > 0$ smaller, if necessary, we obtain that 
\begin{align}
\label{eq:unif-cone-2}
\begin{cases}
u > 0 ~~ \text{ in } \mathcal{C}^+_{2\delta}(x_0,e_n) \cap B_R(x_0),\\
u = 0 ~~ \text{ in } \mathcal{C}^-_{2\delta}(x_0,e_n) \cap B_R(x_0),
\end{cases}
~~ \forall x_0 \in \partial \{ u > 0 \} \cap B_{2\rho},
\end{align} 
which implies that $\partial \{ u >  0 \}$ satisfies the uniform cone condition in $B_{2\rho}$. To see this, in the light of \eqref{eq:unif-cone-1}, it suffices to prove that $\mathcal{C}^{\pm}_{2\delta}(x_0,e_n) \subset \mathcal{C}^{\pm}_{\delta}(x_0,\nu_{x_0})$. Indeed, given $x_0' \in B_{\rho}'$, we have that $x_0 = (x_0',g(x_0')) \in \partial \{ u > 0 \} \cap B_R$ and since $|g(x_0')| \le \delta |x_0'|$, we deduce $|x_0| \le \rho \sqrt{1+\delta^2} \le 2 \rho$, and thus
\begin{align*}
\nu_{x_0} \cdot (x-x_0) &= e_n \cdot (x - x_0) + (\nu_{x_0} - e_n) \cdot (x - x_0) \ge 2 \delta |x-x_0| - \delta |x-x_0| \ge \delta |x-x_0|,
\end{align*}
where we used that by \eqref{eq:flat-implies-smooth-claim-2} we have $|\nu_{x_0} - e_n| \le c \rho^{\alpha} \le \delta$, if $\rho$ is small enough, depending on $\delta$.\\
By the uniform cone condition \eqref{eq:unif-cone-2}, we deduce that the free boundary $\partial \{ u > 0 \}$ in $B_{\rho}' \times (-\rho,\rho)$ is given by the graph of $g$, i.e., by $\{ (x',t) : g(x') = t \}$, and that $g$ is Lipschitz continuous in $B_{\rho}'$. Moreover, we have by definition of $g$
\begin{align*}
(x-x_0) \cdot \nu_{x_0} = (x'- x_0') \cdot \nu_{x_0}' + (g(x') - g(x_0')) (\nu_{x_0})_n,
\end{align*}
and \eqref{eq:flat-implies-smooth-claim-1}, applied with $u(x) = u_{|x-x_0|,x_0}\left( \frac{x-x_0}{|x-x_0|} \right)$ (similar to the proof of \eqref{eq:unif-cone-1}), implies that
\begin{align*}
- c |x-x_0|^{1 + \frac{\gamma}{s}} \le (x - x_0) \cdot \nu_{x_0} \le c |x-x_0|^{1 + \frac{\gamma}{s}} ~~ \forall x'\in B_{\rho}', \text{ and } x = (x',g(x')).
\end{align*}
A combination of these two facts implies that $g$ is differentiable with $\nabla g(x_0') = (\nu_{x_0})'/(\nu_{x_0})_n$.
To see that $g \in C^{1,\alpha}$ in $B_{\rho}'$, we apply again \eqref{eq:flat-implies-smooth-claim-2}, which yields that $x_0 \mapsto \nu_{x_0}$ is in $C^{\alpha}$ and therefore $\nabla g \in C^{\alpha}$.  This proves the $C^{1,\alpha}$ regularity of the free boundary. Note that the $C^{1,\alpha}$ radius only depends on the constants from the previous results and on $\rho$. Then, as an application, from \cite[Proposition 2.7.8]{FeRo23} (and a standard truncation argument), together with the local boundedness estimate in \cite[Theorem 6.2]{Coz17} (using that $Lu \le 0$ in $B_1$ by \autoref{def:viscosity}), we deduce that
\begin{align*}
\Vert u / d^s \Vert_{C^{\alpha}\left( \overline{ \{ u > 0 \} } \cap B_{\rho} \right)} \le C \left(\Vert u \Vert_{L^{\infty}(B_1)} + \tail(u;1) \right) \le C \Vert u \Vert_{L^1_{2s}(\R^n)},
\end{align*}
as desired. Finally, let us recall that $\alpha = \frac{s\gamma}{s+\gamma}$, where $\gamma \in (0,1)$ was chosen such that $2 = \rho_0^{-s/\gamma}$, where $\rho_0 \in (0,1)$ was the constant from \autoref{thm:improvement-of-flatness}. Clearly, \autoref{thm:improvement-of-flatness} becomes a weaker statement, when $\rho_0$ is chosen smaller. Thus, we can choose $\gamma \in (0,s)$ as close to $s$, as we like. Hence, we can choose any $\alpha \in (0,\frac{s^2}{2s}) = (0,\frac{s}{2})$.
\end{proof}

\section{Free boundary regularity for the nonlocal one-phase problem}
\label{sec:free-boundary-reg}

In this section, we prove our main result (see \autoref{thm:main-C1alpha}) and its corollaries \autoref{cor:open-dense} and \autoref{cor:two-dim}.

The main tool to establish \autoref{thm:main-C1alpha} is the flatness implies $C^{1,\alpha}$ result from \autoref{thm:free-boundary-regularity}. It remains to prove that the assumption of \autoref{thm:free-boundary-regularity} holds true near flat free boundary points. This is the goal of Subsection \ref{subsec:close-to-half-space}. The following Subsections \ref{subsec:open-dense} and \ref{sec:two-dimensional-classification} are dedicated to the proofs of \autoref{cor:open-dense} and \autoref{cor:two-dim}.

\subsection{Flat free boundary points}
\label{subsec:close-to-half-space}

In this section, we prove that near every free boundary point satisfying one of the properties 
\begin{itemize}
\item[(i)] points near which the free boundary is flat,
\item[(ii)] points at which the blow-up is the half-space solution,
\item[(iii)] reduced boundary points (points at which measure-theoretic normal exists),
\item[(iv)] points that satisfy the interior ball condition,
\end{itemize}
the minimizer $u$ of $\mathcal{I}_{\Omega}$ is close to the half-space solution $A(\nu)(x \cdot \nu)_+^s$ for some $\nu \in \mathbb{S}^{n-1}$. This result is contained in the following proposition:

\begin{proposition}
\label{thm:close-to-halfspace}
Assume \eqref{eq:Kcomp}. Let $u$ be a minimizer of $\mathcal{I}_{\Omega}$ with $B_2 \subset \Omega$  and $0 \in \partial \{ u > 0 \}$. Then, for any $\eps \in (0,1)$ and $\delta_0 > 0$, there exists $\delta > 0$, such that if one of the following holds true
\begin{itemize}
\item[(i)] for some $\nu \in \mathbb{S}^{n-1}$ it holds
\begin{align*}
\{ x \cdot \nu \le - \delta \} \cap B_1 \subset \{ u = 0 \} \cap B_1 \subset \{ x \cdot \nu \le \delta \} \cap B_1 ,
\end{align*} 
\item[(ii)] for some $\nu \in \mathbb{S}^{n-1}$ it holds, up to a subsequence
\begin{align*}
u_r(x) \to u_0(x) := A(\nu)(x \cdot \nu)_+^s, ~~ \text{ as } r \to \infty, \text{ locally uniformly in } \R^n,
\end{align*}
\item[(iii)] $0 \in \partial^{\ast} \{ u > 0\}$, i.e., for some $\nu \in \mathbb{S}^{n-1}$ it holds, up to a subsequence
\begin{align*}
\{ u_r  > 0 \} \to \{ x \cdot \nu > 0 \}, ~~ \text{ as } r \to \infty, \text{ in } L^1_{loc}(\R^n),
\end{align*}
\item[(iv)] there exists a ball $B \subset \{ u > 0 \}$ with $\overline{B} \cap \partial \{ u > 0 \} = \{ 0 \}$,

\end{itemize}
we have for some $r > 0$:
\begin{align}
\label{eq:close-to-1d-sol}
A(\nu)(x \cdot \nu - \eps)_+^s \le u_r(x) \le A(\nu)(x \cdot \nu + \eps)_+^s ~~ \forall x \in B_1
\end{align}
and
\begin{align}
\label{eq:tail-small}
\tail((u_r - A(\nu)(x \cdot \nu - \eps)_+^s)_- ; 1) + \tail((A(\nu)(x \cdot \nu + \eps)_+^s - u_r)_- ; 1) \le \eps \delta_0.
\end{align}
\end{proposition}

The proof reveals that in case (i), $r$,$\delta$ only depend on $n,s,\lambda,\Lambda,\eps,\delta_0$. This result will be proved separately for each of the points (i), (ii), (iii), (iv). However, note that we will show it for (iii), (iv) by reducing those cases to (ii), and for (ii) by reducing this case to (i) (for a rescaling of $u$).

Once \autoref{thm:close-to-halfspace} is established, \autoref{thm:main-C1alpha} is immediate after combination with \autoref{thm:free-boundary-regularity} and \autoref{lemma:min-visc}.

The following two lemmas are adaptations of \cite[Lemma 7.2]{DeSa12}. They are crucial ingredients in our proof of \autoref{thm:close-to-halfspace}.

\begin{lemma}
\label{lemma:aux-flat-implies-closeness}
Assume \eqref{eq:Kcomp}. Let $\delta > 0$ and $\nu \in \mathbb{S}^{n-1}$. There exist $c,C > 0$, depending only on $n,s,\lambda,\Lambda$ (but not on $\delta$), such that for any $R \ge c \delta^{-\frac{1}{s}}$ and any $u$ satisfying the following properties in the viscosity sense
\begin{align*}
\begin{cases}
Lu &= 0 ~~ \qquad\qquad\qquad \text{ in } \{x \cdot \nu > 0 \} \cap B_R,\\
u(x) &\ge A(\nu)(x \cdot \nu)_+^s - \delta ~~ \forall x \in B_{R},\\
u &\ge 0 ~~ \qquad\qquad\qquad \text{ in } \R^n,
\end{cases}
\end{align*} 
it holds:
\begin{align*}
u(x) \ge (1-C\delta) A(\nu)(x \cdot \nu)_+^s ~~ \forall x \in B_{1}.
\end{align*}
\end{lemma}

\begin{proof}
Without loss of generality, we assume that $\nu = e_n$. We use the notation $d^s(x) := (x_n)_+^s$. Let $R > 4$ to be chosen suitably later on in the proof. We define $q$ to be the solution to
\begin{align*}
\begin{cases}
L q &= 1 ~~ \text{ in } \{ x_n > 0 \} \cap B_{R/2},\\
q &= 1 ~~ \text{ in } \{ x_n > 0 \} \cap (\R^n \setminus B_{R/2}),\\
q &= 0 ~~ \text{ in } \{ x_n \le 0 \} \cap B_2,\\
q &\ge 0 ~~ \text{ in }\{ x_n \le 0 \} \cap (\R^n \setminus B_2),
\end{cases}
\end{align*}
where we choose some nonnegative boundary data in $\{ x_n \le 0 \} \cap (\R^n \setminus B_{2})$ in order to ensure existence of $q$ (see \cite[Theorem 3.2.27]{FeRo23}).
By \cite[Proposition 2.6.4]{FeRo23} applied to $\{ x_n > 0 \}$ in $B_2$, we deduce
\begin{align}
\label{eq:q-reg}
|q| \le C A(e_n) d^s ~~ \text{ in } \{x_n > 0\} \cap B_1.
\end{align}
Note that $C > 0$ depends only on $n,s,\lambda,\Lambda$, but not on $R$ (due to the maximum principle).
Next, we observe that  for $x \in B_{R/2}$ and $y \in \R^n \setminus B_{R}$ it holds $|y| \le |x| + |x-y| \le 2|x-y|$, which implies
\begin{align*}
L(d^s \mathbbm{1}_{B_{R}})(x) = \int_{\R^n \setminus B_R} d^s(y)K(x-y) \d y \le c_1 \int_{\R^n \setminus B_R} |y|^{-n-s} \d y \le c_2 R^{-s} ~~ \forall x \in \{ x_n > 0 \} \cap B_{R/2}.
\end{align*}
Therefore, we have for $R \ge (c_2 A(e_n)/\delta)^{1/s}$:
\begin{align*}
L(A(e_n) d^s\mathbbm{1}_{B_{R}} - \delta q) \le c_2 A(e_n) R^{-s} - \delta \le 0 = Lu ~~ \text{ in } \{ x_n > 0 \} \cap B_{R/2}.
\end{align*}
Moreover, clearly 
\begin{align*}
A(e_n)d^s\mathbbm{1}_{B_{R}} - \delta q \le 0 &\le u ~~ \text{ in } \{ x_n \le 0 \},\\
A(e_n)d^s\mathbbm{1}_{B_{R}} - \delta q = -\delta &\le u ~~ \text{ in } \{ x_n > 0 \} \cap (\R^n \setminus B_{R}),\\
A(e_n)d^s\mathbbm{1}_{B_{R}} - \delta q = A(e_n)d^s - \delta &\le u ~~ \text{ in } \{ x_n > 0 \} \cap (B_{R} \setminus B_{R/2}).
\end{align*}
Thus, by the comparison principle, and using also \eqref{eq:q-reg} we deduce
\begin{align*}
(1 - C\delta)A(e_n) d^s  \le A(e_n) d^s - \delta q \le u ~~ \text{ in } \{ x_n > 0 \} \cap B_1,
\end{align*}
as desired.
\end{proof}

We also require an analogous result which provides an upper bound for $u$. Although the proof is in principle analogous, the nonlocality of our problem requires us to use in addition the nonlocal Harnack inequality to control the tails of $u$:

\begin{lemma}
\label{lemma:aux-flat-implies-closeness-2}
Assume \eqref{eq:Kcomp}. Let $\delta, \eps \in (0,1)$, and $\nu \in \mathbb{S}^{n-1}$. There exist $c,C > 0$, depending only on $n,s,\lambda,\Lambda$ (but not on $\delta$), such that for any $R \ge c \max \{ \delta^{-\frac{1}{s}} ,  \eps \} $, and any $u$ satisfying the following properties in the viscosity sense
\begin{align*}
\begin{cases}
Lu &\le 0 ~~ \qquad\qquad\qquad \text{ in } B_R,\\
Lu &= 0 ~~ \qquad\qquad\qquad \text{ in } \{ x \cdot \nu > \eps \} \cap B_R,\\
u &= 0 ~~ \qquad\qquad\qquad \text{ in } \{ x \cdot \nu \le - \eps \} \cap B_R,\\
u(x) &\le A(\nu)(x \cdot \nu)_+^s + \delta ~~ \forall x \in B_{R},\\
u &\ge 0 ~~ \qquad\qquad\qquad \text{ in } \R^n,
\end{cases}
\end{align*} 
it holds:
\begin{align*}
u(x) \le (1+C\delta) A(\nu)(x \cdot \nu + \eps)_+^s ~~ \forall x \in B_{1}.
\end{align*}
\end{lemma}

\begin{proof}
Without loss of generality, we assume that $\nu = e_n$.
Let us assume that $R > 8 \eps$ and $R \ge \delta^{-\frac{1}{s}}$. Then, there exists a ball $B \subset \{ x_n > \eps \} \cap B_R$ of radius $R/8$ such that $\dist(B,\R^n \setminus B_R) \ge R/8$ and $\dist(B, \{ x_n = \eps \}) \ge R/8 $. Note that $L u = 0$ in $B$ by assumption. Thus, using the tail estimate (see \cite[Theorem 1.9]{KaWe24}) and the upper bound for $u$ in $B_R$ from the assumption, we compute for $x \in B_{R/2}$:
\begin{align*}
L(u \1_{B_R})(x) &\le \int_{\R^n \setminus B_R } u(y) K(x-y) \d y \\
&\le c_1 R^{-2s} \tail(u ; R) \le c_1 R^{-2s} \inf_{B} u \le c_1 (R^{-s} + R^{-2s} \delta) \le c_1 R^{-s}.
\end{align*}
Moreover, by a computation analogous to the one in \autoref{lemma:aux-flat-implies-closeness}, we have $|L(A(e_n)d^s \1_{B_R})| \le c_2 R^{-s}$ in $\{ x_n \ge - \eps \} \cap B_{R/2}$, where we denote $d^s(x) = (x_n + \eps)_+$.\\
Next, let us define $q$ to be the solution to
\begin{align*}
\begin{cases}
L q &= 1 ~~ \text{ in } \{ x_n > - \eps  \} \cap B_{R/2},\\
q &= 1 ~~ \text{ in } \R^n \setminus B_{R/2},\\
q &= 0 ~~ \text{ in } \{ x_n \le - \eps \} \cap B_2 ,\\
q &\ge 0 ~~ \text{ in } \{ x_n \le - \eps \} \cap (B_{R/2} \setminus B_{2}).
\end{cases}
\end{align*}
By \cite[Proposition 2.6.4]{FeRo23} we deduce
\begin{align}
\label{eq:q-reg-2}
|q| \le C A(e_n) d^s ~~ \text{ in } \{ x_n > - \eps \} \cap B_1.
\end{align}
Note that by the previous computations we have, once $R$ is chosen so large that $(c_1 + c_2) R^{-s} \le \delta$,
\begin{align*}
L((u - A(e_n)d^s)\1_{B_R}) \le (c_1 + c_2) R^{-s} \le \delta = L(\delta q) ~~ \text{ in } \{ x_n > - \eps \} \cap B_{R/2}.
\end{align*}
Thus, by the comparison principle (same arguments as in the proof of \autoref{lemma:aux-flat-implies-closeness}), we deduce
\begin{align*}
(u - A(e_n) d^s )\1_{B_R} \le \delta q ~~ \text{ in } \R^n.
\end{align*}
Therefore, using \eqref{eq:q-reg-2}, we deduce
\begin{align*}
u \le A(e_n) d^s + \delta q \le (1+C \delta) A(e_n) d^s ~~ \text{ in } \{ x_n \ge - \eps \} \cap B_1,
\end{align*}
as desired.
\end{proof}

\subsubsection*{Flatness implies closeness to half-space solution}

We prove that flatness of the free boundary implies closeness of $u$ to the half-space solution (see \autoref{thm:close-to-halfspace}(i)).

This was proved for the fractional Laplacian in \cite[Lemma 2.10]{DSS14} and in \cite[Lemma 7.9]{DeSa12}, but is new for general nonlocal operators.

\begin{proof}[Proof of \autoref{thm:close-to-halfspace}(i)]
Without loss of generality, we assume that $\nu = e_n$. Let $\eps > 0$ and $\delta_0 > 0$ be given. Assume by contradiction that there exist sequences of minimizers $(u_k)_k$, of homogeneous jumping kernels $(K_k)_k$ satisfying \eqref{eq:Kcomp}, and positive numbers $(\delta_k)_k$ with $\delta_k \searrow 0$, such that (i) holds for every $k$, but the conclusion fails, i.e., for every $r$, $u_{k,r}$ violates \eqref{eq:close-to-1d-sol} or \eqref{eq:tail-small}. \\
First, let us extract a subsequence $r_k := \delta_k^{1/2}$ and deduce from \autoref{lemma:scaling-blow-up} that up to a subsequence, it holds $u_{k,r_k} \to u_{\infty} \in H^s(B_R) \cap L^1_{2s}(\R^n)$ locally uniformly and in $H^s(B_R)$ and $L^1_{2s}(\R^n)$ such that $u_{\infty} \ge 0$ in $\R^n$ and minimizes $\mathcal{I}_{B_R}$ in $B_R$ for any $R > 0$ and some homogeneous kernel $K_{\infty}$ satisfying \eqref{eq:Kcomp}. 
As a consequence of (i), we have
\begin{align}
\label{eq:consequence-i}
\{ x_n \le - \delta_k r_k^{-1} \} \cap B_{r_k^{-1}} \le \{ u_{k,r_k} = 0 \} \cap B_{r_k^{-1}} \subset \{ x_n \le \delta_k r_k^{-1} \} \cap B_{r_k^{-1}},
\end{align}
and since by definition of $r_k$ it holds $r_k^{-1} \delta_k \to 0$, we obtain $ \{u_{\infty} > 0\} = \{ x \cdot e_n > 0 \}$, and therefore by \autoref{lemma:min-visc} and \autoref{prop:free-bound-cond}:
\begin{align*}
\begin{cases}
L_{K_{\infty}} u_{\infty} &= 0 ~~ \qquad \text{ in } \{ x_n > 0 \},\\
u_{\infty} &= 0 ~~ \qquad \text{ in }  \{ x_n \le 0 \},\\
\frac{u_{\infty}}{d^s} &= A(e_n) ~~ \text{ in }  \{ x_n = 0 \}.
\end{cases}
\end{align*}
Thus, by the Liouville theorem (see \cite[Theorem 2.7.2]{FeRo23}), it must be $u_{\infty} = A(e_n)(x_n)_+^s$.
By the locally uniformly convergence, we deduce that for any $R > 1$ and $\eta \in (0,1)$, there exists $k_0 \in \N$ such that for any $k \ge k_0$:
\begin{align}
\label{eq:loc-unif-conv-consequence}
|u_{k,r_k} - u_{\infty}| \le \eta ~~ \text{ in } B_R.
\end{align}
Thus, we have
\begin{align*}
u_{k,r_k}(x) \ge A(e_n)(x_n)_+^s - \eta \ge A(e_n)(x_n - \delta_k r_k^{-1})_+^s - \eta ~~ \forall x \in B_R.
\end{align*}
As a consequence of \eqref{eq:consequence-i}, and \autoref{lemma:scaling-blow-up}, choosing $k$ so large that $r_k^{-1} \ge R$, we have $\{ x_n - \delta_k r_k^{-1} > 0 \} \cap B_{R} \subset \{ u_{k,r_k} > 0 \} \cap B_{R}$. This allows us to apply \autoref{lemma:aux-flat-implies-closeness} (with $R := R$ and $\delta := \eta$).  Note that due to \eqref{eq:loc-unif-conv-consequence} the relation $R \ge c\delta^{-\frac{1}{s}} = c \eta^{- \frac{1}{s}}$ holds true once $\eta$ is chosen small enough, which is possible, simply by choosing $k$ large enough. Thus, we deduce
\begin{align*}
u_{k,r_k}(x) \ge (1-C\eta) A(e_n) (x \cdot e_n - \delta_k r_k^{-1})_+^s \ge A(e_n) (x \cdot e_n - c(\delta_k r_k^{-1} + \eta^{\frac{1}{s}}))_+^s ~~ \forall x \in B_{1},
\end{align*}
where we used that 
\begin{align*}
(x \cdot e_n)(1 - C \eta)^{\frac{1}{s}} \ge x \cdot e_n - c \eta^{\frac{1}{s}}, \qquad - \delta_k r_k^{-1}(1 - C \eta)^{\frac{1}{s}} \ge  - \frac{1}{2} \delta_k r_k^{-1}
\end{align*}
for some constant $c > 0$, depending only on $C,s$, once $\eta > 0$ is chosen small enough. Choosing $k$ so large that $c(\delta_k r_k^{-1} + \eta^{\frac{1}{s}}) < \eps$, we have verified \eqref{eq:close-to-1d-sol} for $u_{k,r_k}$. Moreover, if we take any $R > 0$, note that by choosing $k > 0$ even larger, depending on $R$, a rescaled version of \autoref{lemma:aux-flat-implies-closeness} implies that the previous estimate does not only hold true in $B_1$, but even in $B_R$. Taking $R > 0$ large enough, depending on $\eps,\delta_0$, we can thereby also get 
\begin{align*}
\tail((u_{k,r_k} - A(e_n)(x_n - \eps)_+^s)_- ; 1) < \eps \delta_0/2,
\end{align*}
i.e., the first estimate in \eqref{eq:tail-small}. \\
An analogous chain of arguments based on \autoref{lemma:aux-flat-implies-closeness-2} (applied with $R := R, \delta := \eta, \eps := \delta_k r^{-1}$) yields a corresponding upper bound for $u_{k,r_k}$ and also the second estimate in \eqref{eq:tail-small}, a contradiction. This proves the desired result.
\end{proof}

\subsubsection*{Closeness to half-space solutions when blow-up is a half-space solution}

We prove that minimizers are close to the half-space solution near boundary points at which the blow-up is the half-space solution (see \autoref{thm:close-to-halfspace}(ii)).

\begin{proof}[Proof of \autoref{thm:close-to-halfspace}(ii)]
By \autoref{cor:blowups}(iv), we know that up to a subsequence, it holds 
\begin{align*}
\overline{\{ u_{r} > 0 \}} \to \overline{\{ u_{0} > 0 \}} = \{ x \cdot \nu \ge 0 \}
\end{align*}
locally in $B_R$ for any $R > 0$ in the Hausdorff-sense. Thus, for any $\delta > 0$ there exists $r > 0$ such that
\begin{align*}
\{ x \cdot \nu \le - \delta \} \cap B_1 \subset \{u_r = 0\} \cap B_1 \subset \{ x \cdot \nu \le \delta \} \cap B_1.
\end{align*}
Therefore, the desired result follows by application of \autoref{thm:close-to-halfspace}(i) to $u_r$.
\end{proof}

\subsubsection*{Closeness to half-space solutions at reduced boundary points}

We prove that minimizers are close to the half-space solution near reduced boundary points (see \autoref{thm:close-to-halfspace}(iii)).

The proof of \autoref{thm:close-to-halfspace}(iii) is a direct consequence of the following lemma:

\begin{lemma}
\label{lemma:reduced-blow-up}
Assume \eqref{eq:Kcomp}. Let $u$ be a minimizer of $\mathcal{I}_{\Omega}$ with $B_2 \subset \Omega$ and $0 \in \partial^{\ast} \{ u > 0\}$, where $\nu \in \mathbb{S}^{n-1}$ denotes the measure theoretic inward normal to $\{ u > 0 \}$ at $0$. Then, up to a subsequence,
\begin{align*}
u_r \to u_0 = A(\nu)(x \cdot \nu)_+^s ~~ \text{ locally uniformly in } \R^n. 
\end{align*}
\end{lemma}

\begin{proof}
By \autoref{cor:blowups}, we deduce that $u_r \to u_0$ locally uniformly, up to a subsequence, and that $u_0$ minimizes $\cI_{B_R}$ for any $R$. Therefore, 
\begin{align*}
\begin{cases}
L u_0 &= 0 ~~ \text{ in } \{ u_0 > 0 \},\\
u_0 &\ge 0 ~~ \text{ in } \R^n.
\end{cases}
\end{align*}
Moreover, since $0 \in \partial^{\ast}\{ u > 0\}$, we have $\{ u_r > 0 \} \to \{ x \cdot \nu > 0 \}$ locally in $L^1(\R^n)$. This implies that $\{ u_0 > 0 \} = \{ x \cdot \nu > 0\}$. Indeed, if this property did not hold true, we could find a sequence $(x_r) \subset \partial \{ u_r > 0\}$ with $x_r \to x \in \{ x \cdot \nu > 0 \}$. But then, we would have
\begin{align*}
\frac{|B_{\dist(x_r,\{x \cdot \nu = 0\})}(x_r) \cap \{ u_r > 0\}|  }{|B_{\dist(x_r,\{x \cdot \nu = 0\})}(x_r)|} \to 1,
\end{align*}
contradicting \autoref{thm:density-est}.
Finally, since $\partial \{ u_0 > 0 \} = \{ x \cdot \nu = 0\}$, we deduce from the Liouville theorem in the half-space (see \cite[Theorem 2.7.2]{FeRo23}) that
\begin{align*}
u_0(x) = \kappa (x \cdot \nu)_+^s.
\end{align*}
Finally, we can apply \autoref{prop:free-bound-cond} and deduce
\begin{align*}
\frac{u_0}{d^s} = A(\nu) ~~ \text{ on } \partial \{ u_0 > 0 \},
\end{align*}
which implies that $\kappa = A(\nu)$. This concludes the proof.
\end{proof}

\begin{proof}[Proof of \autoref{thm:close-to-halfspace}(iii)]
By \autoref{lemma:reduced-blow-up}, we have verified (ii). Thus, the desired result follows from \autoref{thm:close-to-halfspace}(ii), which we already proved before.
\end{proof}

\subsubsection*{Interior ball condition implies closeness to half-space solution}

In this section, we prove that near any point $x_0 \in \partial \{ u > 0 \}$ which can be touched by a ball from the interior, the solution is close to the half-space solution (see \autoref{thm:close-to-halfspace}(iv)).

The proof of \autoref{thm:close-to-halfspace}(iv) is a direct consequence of the following lemma:

\begin{lemma}
\label{lemma:interior-ball-blowup}
Assume \eqref{eq:Kcomp}. Let $u$ be a minimizer of $\mathcal{I}_{\Omega}$ with $B_2 \subset \Omega$ and assume that there exists a ball $B \subset \{ u > 0 \}$ with $\overline{B} \cap \partial\{ u > 0 \} = \{ 0 \}$. Then, there exists $\nu \in \mathbb{S}^{n-1}$ such that
\begin{align*}
u_r \to u_0 = A(\nu)(x \cdot \nu)_+^s ~~ \text{ locally uniformly in } \R^n. 
\end{align*}
\end{lemma}

\begin{proof}[Proof of \autoref{thm:close-to-halfspace}(iv)]
Without loss of generality, we assume that $\nu = e_n$, where $\nu \in \mathbb{S}^{n-1}$ denotes the normal vector of $\partial B$ at zero, inward to $\{ u > 0\}$. Then, clearly, $B \subset \{x_n > 0 \}$. Then, by \autoref{lemma:interior-ball-blowup-viscosity} there is $\alpha \ge 0$ such that for any $x \in B \cap \{d_B(x) \ge |x|/2 \}$ (non-tangential region inside $B$) near $0$ it holds
\begin{align}
\label{eq:claim-interior-ball}
u(x) = \alpha (x_n)_+^s + o(|x|^s).
\end{align}
Let us explain how \eqref{eq:claim-interior-ball} implies the desired result. Note that with the aid of the non-degeneracy (see \autoref{thm:non-degeneracy}), we can deduce that $\alpha > 0$. Indeed, $B \subset \{ u > 0 \}$ yields the existence of a sequence $(x_k)_k \subset B \cap \{d_B \ge |x|/2 \}$ such that $x_k \to 0$ and $\dist(x_k,\partial \{ u > 0 \}) = (x_k)_n$, and by \autoref{thm:non-degeneracy} this implies $u(x_k) \ge c ((x_k)_n)_+^s$ for some constant $c > 0$. Thus $u(x) \not= o(|x|^s)$ near $0$, and it follows $\alpha > 0$.\\
Moreover, note that since $u \in C^s(B_1)$ (due to \autoref{thm:or}) with $u(0) = 0$, we have that $u \le C(x_n)_-^s$ in $\{ x_n \le 0 \}$ near $0$ for some $C > 0$, and therefore by application of \autoref{lemma:one-sided-expansion} to $v \1_{\{ x_n \le 0 \}}$ in $\{ x_n \le 0 \}$, we obtain that for $x \in \{ x_n \le 0 \}$ near $0$ it holds
\begin{align*}
u(x) = \beta (x_n)_-^s + o(|x|^s).
\end{align*}
Altogether, and using also \autoref{cor:blowups}, this implies that the blow-up sequence $(u_r)_r$ converges to $u_0$ locally uniformly, where $u_0$ is defined as follows
\begin{align*}
u_0(x) = \alpha (x_n)_+^s + \beta (x_n)_-^s ~~ \forall x \in \R^n.
\end{align*}
Here we used that for any $x \in \{ x_n > 0 \}$ there is $r_0 > 0$ such that $rx \in  B \cap \{d_B(x) \ge |x|/2 \}$ for any $r \le r_0$, so that the convergence in $\{ x_n > 0 \}$ follows from the claim \eqref{eq:claim-interior-ball}.\\
Since by \autoref{cor:blowups}, $u_0$ is a global minimizer of $\mathcal{I}$, \autoref{thm:density-est} implies that we cannot have $\alpha, \beta > 0$, so it must be $\beta = 0$. Thus, $u_0(x) = \alpha (x_n)_+^s$ and by \autoref{prop:free-bound-cond}, we deduce that $\alpha = A(e_n)$, as desired.
\end{proof}

\begin{proof}[Proof of \autoref{thm:close-to-halfspace}(iv)]
By \autoref{lemma:interior-ball-blowup}, we have verified (ii). Thus, the desired result follows from the proof of \autoref{thm:close-to-halfspace}(ii).
\end{proof}

The proof of \autoref{thm:main-C1alpha} is now immediate.

\begin{proof}[Proof of \autoref{thm:main-C1alpha}]
By \autoref{thm:close-to-halfspace}(i) and \autoref{lemma:min-visc}, we can apply \autoref{thm:free-boundary-regularity} to $u_r$ for some $r > 0$, depending only on $n,s,\lambda,\Lambda$, once $\delta \in(0,1)$ is small enough. Then, the desired result follows immediately from \autoref{thm:free-boundary-regularity} after a suitable rescaling. 
\end{proof}

\subsection{Smoothness in an open dense set}
\label{subsec:open-dense}

In this section we prove \autoref{cor:open-dense}.

\begin{proof}[Proof of \autoref{cor:open-dense}]
Let us define
\begin{align*}
\mathcal{O} := \left\{ x_0 \in \partial \{ u > 0 \} \cap \Omega : \exists \text{ ball } B \subset \{ u > 0 \} ~~ \text{ s.t. } \overline{B} \cap \partial \{ u > 0 \} = \{ x_0 \} \right\}.
\end{align*}
It is easy to see that the set $\mathcal{O}$ is open. To see that $\mathcal{O}$ is dense in $\Omega \cap \partial \{ u > 0 \}$, it clearly suffices to assume that $\Omega \subset B_2$ and to prove that $\mathcal{O}$ is dense in $\partial \{ u > 0 \} \cap B_1$. To show this, let $x_0 \in B_1 \cap \partial \{ u > 0 \}$ and $\eps > 0$ be arbitrary. Our goal is to find $y_0 \in \mathcal{O}$ with $|x_0 - y_0| < \eps$. To do so, take $x \in \{ u > 0 \}$ with $|x_0 - x| < \eps/2$. Since $\{ u > 0 \}$ is open, there exists a radius $\delta > 0$ such that $B_{\delta}(x) \subset \{ u > 0 \}$. Let us set
\begin{align*}
\delta_0 := \sup \{ \delta > 0 : B_{\delta}(x) \subset \{ u > 0 \} \}
\end{align*}
and observe that $\delta_0 \in (0,\eps/2)$. Moreover, by construction, $B_{\delta_0}(x) \subset \{ u > 0 \}$ and there exists $y_0 \in \overline{B_{\delta_0}(x)} \cap \partial \{ u > 0 \}$. Clearly $|x_0 - y_0| \le |x_0 - x| + |x - y_0| < \eps$. Moreover, if necessary, by shifting  and shrinking the ball $B_{\delta_0}(x)$, we can guarantee that $y_0$ is the only point in the intersection, so that $y_0 \in \mathcal{O}$, as desired.\\
We have shown that $\mathcal{O}$ is an open, dense set. Due to \autoref{thm:close-to-halfspace}(iv), we can apply \autoref{thm:free-boundary-regularity} for any $x_0 \in \mathcal{O}$ to a rescaling $u_r$ of $u$ (where $r$ depends on $u,x_0$), which concludes the proof.
\end{proof}

\subsection{Classification in two dimensions}
\label{sec:two-dimensional-classification}

The goal of this section is to prove \autoref{cor:two-dim}. The idea of the proof is to classify all global minimizers of $\mathcal{I}$ in dimension $n = 2$. We will show that all global minimizers are one-dimensional. In particular, this yields a characterization of blow-ups and therefore full regularity of the nonlocal one-phase problem in dimension $n=2$ (see \autoref{cor:two-dim}).

\begin{theorem}
\label{thm:two-dimensional-classification}
Let $n=2$. Let $K \in C^{2}(\mathbb{S}^{1})$ and assume \eqref{eq:Kcomp}. Let $u$ be a minimizer of $\cI_{B_R}$ for any $R > 0$, and assume that $0 \in \partial \{ u > 0 \}$. Then, there exists $e \in \mathbb{S}^{1}$ such that
\begin{align*}
u(x) = A(e) (x \cdot e)_+^s.
\end{align*}
\end{theorem}

With the aid of \autoref{thm:two-dimensional-classification}, the proof of \autoref{cor:two-dim} becomes straightforward.

\begin{proof}[Proof of \autoref{cor:two-dim}]
From \autoref{thm:two-dimensional-classification}, we deduce that for any $x_0 \in \Omega \cap \partial\{ u > 0\}$ it holds $u_{r,x_0} \to A(e_{x_0})(x \cdot e)_+^s$ for some $\nu_{x_0} \in \mathbb{S}^{n-1}$. Thus, by \autoref{thm:close-to-halfspace}(ii), we can apply \autoref{thm:free-boundary-regularity} for any $x_0  \in B_1 \cap \partial\{ u > 0\}$ to a rescaling $u_r$ of $u$ (where $r$ depends on $u,x_0$).
\end{proof}

The rest of this section is dedicated to the proof of \autoref{thm:two-dimensional-classification}.
The proof of the classification of blow-ups in dimension $n=2$ from \cite{AlCa81} does not seem to work in our setup. The proofs for the thin one-phase problem (see \cite{DeSa15}, \cite{EKPSS21}) establish classification of homogeneous minimizers (minimal cones), which implies the classification of blow-ups by the Allen-Weiss monotonicity formula. Since such formula is not available for our general class of operators, in this paper, we follow a different strategy, inspired by \cite{CSV19} and \cite{FiSe19}.

Let us define
\begin{align*}
\phi_R(x) = 
\begin{cases}
1, ~~ &\text{ if } |x| \le \sqrt{R},\\
2\left(1 - \frac{\log |x|}{\log R} \right), ~~ &\text{ if } \sqrt{R} \le |x| \le R,\\
0 ~~ &\text{ if } |x| \ge R.
\end{cases}
\end{align*}

For $\nu \in \mathbb{S}^{n-1}$ and $t \in [-1,1]$, we set
\begin{align*}
\Psi_{R,t}(x) = x + t \phi_R(x) \nu, \qquad u_{R,t}(x) = u(\Psi^{-1}_{R,t}(x)).
\end{align*}

First, we have the following lemma, which generalizes \cite[Lemma 2.1]{CSV19} in the sense that the bound on the right hand side only contains the $L^1_{2s}(\R^n)$ norm of $u$ (respectively $\tail(u)$) instead of the energy in the whole space. Therefore, the following lemma does not require $u \in V^s(B_R | \R^n)$ (if one interprets the energies on the left hand side as being all written under the same integral).

\begin{lemma}
\label{lemma:two-dim-energy-est-1}
Let $n \ge 2$. Let $K \in C^{2}(\mathbb{S}^{n-1})$ and assume \eqref{eq:Kcomp}. Let $R \ge 4$, and $u \in V^s(B_{2R} | B_{3R}) \cap L^1_{2s}(\R^n)$. Then, for all $t \in (-1,1)$:
\begin{align*}
\cI_{B_R}(u_{R,t}) + \cI_{B_R}(u_{R,-t}) - 2 \cI_{B_R} (u) \le \frac{C t^2}{\log R} \mathcal{S}_R,
\end{align*}
where $C > 0$ depends only on $n,s,\lambda,\Lambda$, and $\Vert K \Vert_{C^2(\mathbb{S}^{n-1})}$, and we denote
\begin{align*}
\mathcal{S}_R := \left(\sup_{\rho \in [1,R]} \frac{\cE_{B_{2\rho} \times B_{2\rho}}(u,u)}{\rho^2} + \sup_{\rho \in [1,R]} \frac{ R^{-2s} \Vert u \Vert_{L^2(B_{\rho})}^2}{\rho^2} + \sup_{\rho \in [1,R]} \frac{R^{-2s}  \tail(u;2R) \Vert u \Vert_{L^{1}(B_{\rho})}}{\rho^2} \right)
\end{align*}
\end{lemma}

\begin{proof}
First, by following the same arguments as in the proof of \cite[Lemma 2.1]{CSV19}, we deduce
\begin{align*}
& \cE_{(B_{2R} \times B_{2R}) \setminus (B_R^c \times B_R^c)}(u_{R,t},u_{R,t}) + \cE_{(B_{2R} \times B_{2R}) \setminus (B_R^c \times B_R^c)}(u_{R,-t},u_{R,-t}) - 2 \cE_{(B_{2R} \times B_{2R}) \setminus (B_R^c \times B_R^c)}(u,u) \\
&= \cE_{B_{2R} \times B_{2R}}(u_{R,t},u_{R,t}) + \cE_{B_{2R} \times B_{2R}}(u_{R,-t},u_{R,-t}) - 2 \cE_{B_{2R} \times B_{2R}}(u,u) \le \frac{C t^2}{\log R} \sup_{\rho \in [1,R]} \frac{\cE_{B_{2\rho} \times B_{2\rho}}(u,u)}{\rho^2}.
\end{align*}
The only difference to \cite{CSV19} is the slightly different domain of integration, however, it does not change any of the arguments. Moreover, the first identity comes from the fact that $u_{R,\pm t} \equiv u$ in $\R^n \setminus B_R$, which leads to cancellations.
Next, we observe
\begin{align*}
\cE_{B_{R} \times B_{2R}^c}(u_{R,\pm t},u_{R,\pm t}) &- \cE_{B_{R} \times B_{2R}^c}(u,u) \\
&= \int_{B_R} \int_{\R^n \setminus B_{2R}} \big[ (u_{R,\pm t}(x) - u_{R,\pm t}(y))^2 - (u(x) - u(y))^2 \big] K(x-y) \d y \d x \\
& = \int_{B_R} \int_{\R^n \setminus B_{2R}} \big[ u_{R,\pm t}^2(x) - u^2(x) - 2u(y)(u_{R,\pm t}(x) - u(x))\big] K(x-y) \d y \d x \\
&= \int_{B_R} \int_{\R^n \setminus B_{2R}} \big[ u_{R,\pm t}^2(x)-2u_{R,\pm t}(y)u_{R,\pm t}(x) \big] K(x-y) \d y \d x \\
&\quad - \int_{B_R} \int_{\R^n \setminus B_{2R}} \big[ u^2(x)-2u(y)u(x) \big] K(x-y) \d y \d x.
\end{align*}
By summing up the previous identity, once for $u_{R,+t}$ and once for $u_{R,-t}$, and doing the same change of coordinates as in \cite{CSV19}, and using \cite[(2.11),(2.13)]{CSV19}, we obtain
\begin{align*}
\cE_{B_{R} \times B_{2R}^c} &(u_{R,t},u_{R,t}) + \cE_{B_{R} \times B_{2R}^c}(u_{R,-t},u_{R,-t}) - 2 \cE_{B_{R} \times B_{2R}^c}(u,u) \\
&= \int_{B_R} \int_{\R^n \setminus B_{2R}} \big[ u^2(x)-2u(y)u(x) \big] e(x,y) \d y \d x \\
&\le \int_{B_R} u^2(x) \left( \int_{\R^n \setminus B_{2R}} |e(x,y)| \d y \right) \d x + 2 \int_{B_R} |u(x)| \left( \int_{\R^n \setminus B_{2R}} |u(y)| |e(x,y)| \d y \right) \d x\\
&= I_1 + I_2,
\end{align*}
where $e(x,y)$ satisfies the following estimate for some $C > 0$, depending only on $n,s,\lambda,\Lambda$, and $\Vert K \Vert_{C^2(\mathbb{S}^{n-1})}$:
\begin{align*}
|e(x,y)| \le \frac{C t^2}{(\log R)^2 \max \{ R , \rho^2 \} } |x-y|^{-n-2s} \qquad \forall x,y \in \R^n \setminus B_{\rho}.
\end{align*}
Hence, by splitting the domain of integration into suitable annuli, taking $k \in \N$ such that $\log_2 R \le 2k < \log_2 R + 2$ and $\theta^{2k} = R$ we obtain
\begin{align*}
I_1 &\le \frac{C t^2}{(\log R)^2} R^{-2s} \left( \int_{B_R} \frac{u^2(x)}{\max\{ R, |x|^2 \}} \d x \right) \\
&\le \frac{C t^2}{(\log R)^2} R^{-2s} \left( R^{-1} \int_{B_{\sqrt{R}}} u^2(x) \d x + \sum_{i = k+1}^{2k} \theta^{-2(i-1)} \int_{B_{\sigma^i} \setminus B_{\sigma^{i-1}}} u^2(x) \d x \right) \\
&\le \frac{C t^2}{(\log R)^2} R^{-2s} \sup_{\rho \in [1,R]} \frac{\Vert u \Vert_{L^2(B_{\rho})}^2}{\rho^2} \left( 1 + \sum_{i = k+1}^{2k} \frac{\theta^{2i}}{\theta^{2(i-1)}} \right) \\
&\le \frac{C t^2}{\log R} R^{-2s} \sup_{\rho \in [1,R]} \frac{\Vert u \Vert_{L^2(B_{\rho})}^2}{\rho^2} ,
\end{align*}
where we used that $1 + \sum_{i = k+1}^{2k} \frac{\theta^{2i}}{\theta^{2(i-1)}} = (k+1)\theta^2 \le C (\log R)^{-1}$ by construction. Analogously for $I_2$, 
\begin{align*}
I_2 &\le \frac{C t^2}{(\log R)^2} R^{-2s} \tail(u;2R) \left( \int_{B_R} \frac{|u(x)|}{\max\{ R , |x|^2\}} \d x \right)\\
&\le \frac{C t^2}{\log R} R^{-2s} \tail(u;2R) \sup_{\rho \in [1,R]} \frac{\Vert u \Vert_{L^1(B_{\rho})}}{\rho^2}.
\end{align*}
Altogether, summing up all the aforementioned estimates, and using that
\begin{align*}
(B_R^c \times B_R^c)^c = \big[ (B_{2R} \times B_{2R}) \setminus (B_R^c \times B_R^c) \big] \cup \big[ B_{R} \times B_{2R}^c \big] \cup \big[ B_{2R}^c \times B_{R} \big],
\end{align*}
we have shown
\begin{align*}
\cE_{(B_R^c \times B_R^c)^c}(u_{R,t},u_{R,t}) + \cE_{(B_R^c \times B_R^c)^c}(u_{R,-t},u_{R,-t}) - 2 \cE_{(B_R^c \times B_R^c)^c}(u,u) \le \frac{C t^2}{\log R} \mathcal{S}_R.
\end{align*}

Moreover, by the same arguments as in \cite[Theorem 5.5]{DeSa15}, it holds
\begin{align*}
|\{ u_{R,\pm t} > 0 \} \cap B_R| = \int_{ \{ u > 0 \} \cap B_R} (1 \pm t \partial_{\nu} \phi_R(x)) \d x.
\end{align*}
This identity implies
\begin{align*}
|\{ u_{R,t} > 0 \} \cap B_R| + |\{ u_{R,-t} > 0 \} \cap B_R| - 2 |\{ u > 0 \} \cap B_R| = 0,
\end{align*}
and therefore we immediately obtain the desired result.
\end{proof}

As a consequence, we deduce the following lemma:

\begin{lemma}
\label{lemma:two-dim-energy-est-2}
Let $n \ge 2$. Let $K \in C^{2}(\mathbb{S}^{n-1})$ and assume \eqref{eq:Kcomp}. Let $R \ge 8$, and $u$ be a minimizer of $\cI_{B_R}$. Then, for any $t \in (-1,1)$:
\begin{align*}
\int_{B_R} \int_{B_R} (u(x) - u_{R,t}(x))_+ (u(y) - u_{R,t}(y))_- \frac{\d y \d x}{|x-y|^{n+2s}} \le \frac{C t^2}{\log R} \mathcal{S}_R,
\end{align*}
where $C > 0$ depends only on $n,s,\lambda,\Lambda$, and $\Vert K \Vert_{C^2(\mathbb{S}^{n-1})}$. In particular, for any $x_0 \in \R^n$ such that $B_1(x_0) \subset B_{R/4}$,
\begin{align*}
\left(\int_{B_{1}(x_0)} \frac{(u(x) - u(x + t \nu))_+}{t} \d x \right) \left( \int_{B_{1}(x_0)} \frac{(u(y) - u(y + t\nu))_-}{t} \d y \right) \le \frac{C}{\log R} \mathcal{S}_R.
\end{align*}
\end{lemma}

\begin{proof}
Since $u$ is a minimizer of $\mathcal{I}_{B_R}$ in $B_R$, and $u \equiv u_{R,\pm t}$ in $\R^n \setminus B_R$ by construction, we have
\begin{align*}
\cI_{B_R}(u) \le \cI_{B_R}(u_{R,\pm t}), \qquad \cI_{B_R}(u) \le \cI_{B_R}( u \wedge u_{R,t}), \qquad \cI_{B_R}(u) \le \cI_{B_R}( u \vee u_{R,t}),
\end{align*}
and therefore \autoref{lemma:two-dim-energy-est-1} implies
\begin{align*}
\cI_{B_R}(u_{R,t}) \le \cI_{B_R}(u) + \frac{C t^2}{\log R} \mathcal{S}_R.
\end{align*}
Moreover, we recall the following algebraic identity from \cite[Proof of Lemma 3.4]{RoWe24b}
\begin{align}
\label{eq:min-max-alg-identity}
\begin{split}
\big((w_1 \wedge \phi_1) - (w_2 \wedge \phi_2) \big)^2 &+ \big((w_1 \vee \phi_1) - (w_2 \vee \phi_2) \big)^2 \\
&= (w_1 - w_2)^2 + (\phi_1 - \phi_2)^2 -2(w_1 - \phi_1)_+(w_2 - \phi_2)_- .
\end{split}
\end{align}
This yields:
\begin{align*}
\cI_{B_R}(u \wedge u_{R,t}) + \cI_{B_R}(u \vee u_{R,t}) &+2 \iint_{(B_R^c \times B_R^c)^c} (u(x) - u_{R,t}(x))_+ (u(y) - u_{R,t}(y))_- K(x-y) \d y \d x \\
&= \cI_{B_R}(u) + \cI_{B_R}(u_{R,t}).
\end{align*}
By combination of all the previous facts, we deduce
\begin{align*}
2 \iint_{(B_R^c \times B_R^c)^c} & (u(x) - u_{R,t}(x))_+ (u(y) - u_{R,t}(y))_- K(x-y) \d y \d x \\
&= \cI_{B_R}(u) + \cI_{B_R}(u_{R,t}) - \cI_{B_R}(u \wedge u_{R,t}) - \cI_{B_R}(u \vee u_{R,t}) \\
&\le \cI_{B_R}(u_{R,t}) - \cI_{B_R}(u \vee u_{R,t}) \le \frac{C t^2}{\log R} \mathcal{S}_R.
\end{align*}
This yields the first claim. To obtain the second claim, note that in $B_{1}(x_0) \subset B_{R/4}$, we have $u_{R,t}(x) = u(x + t \nu)$. Thus, making the domain of integration smaller, and estimating $|x-y|^{-n-2s} \ge c$ for some $c > 0$ in $B_{1}(x_0) \times B_{1}(x_0)$, we obtain the second claim after division by $t^2$. The proof is complete.
\end{proof}

By combination of the optimal regularity for minimizers of $\mathcal{I}$ with the previous lemma, we deduce

\begin{lemma}
\label{lemma:two-dim-energy-est-3}
Let $n = 2$. Let $K \in C^{2}(\mathbb{S}^{1})$ and assume \eqref{eq:Kcomp}. Let $R \ge 8$, and $u$ be a non-trivial minimizer of $\cI_{B_{2R}}$ in $B_{2R}$ such that $0 \in \partial \{ u > 0 \}$. Then, for any $x_0 \in \R^n$ such that $B_1(x_0) \subset B_{R/4}$ and any $ t \in (-1,1) \setminus \{ 0 \}$, it holds
\begin{align*}
\left(\int_{B_{1}(x_0)} \frac{(u(x) - u(x + t \nu))_+}{t} \d x \right) & \left( \int_{B_{1}(x_0)} \frac{(u(y) - u(y + t\nu))_-}{t} \d y \right) \le \frac{C}{\log R},
\end{align*}
where $C > 0$ depends only on $n,s,\lambda,\Lambda$, and $\Vert K \Vert_{C^2(\mathbb{S}^{1})}$.
In particular if $u$ is a minimizer of $\mathcal{I}_{B_R}$ in $B_R$ for any $R > 0$, then it holds for any $x_0 \in \R^n$ and any $t \in (-1,1) \setminus \{ 0 \}$:
\begin{align*}
\left(\int_{B_{1}(x_0)} \frac{(u(x) - u(x + t \nu))_+}{t} \d x \right) & \left( \int_{B_{1}(x_0)} \frac{(u(y) - u(y + t \nu))_-}{t} \d y \right) = 0.
\end{align*}
\end{lemma}

\begin{proof}
We observe that as a consequence of \autoref{lemma:energy-bound} and using that $n = 2$, we have
\begin{align*}
\mathcal{S}_R &= \left(\sup_{\rho \in [1,R]} \frac{\cE_{B_{2\rho} \times B_{2\rho}}(u,u)}{\rho^2} + \sup_{\rho \in [1,R]} \frac{ R^{-2s} \Vert u \Vert_{L^2(B_{\rho})}^2}{\rho^2} + \sup_{\rho \in [1,R]} \frac{R^{-2s}  \tail(u;2R) \Vert u \Vert_{L^{1}(B_{\rho})}}{\rho^2} \right) \\
&\le C \sup_{\rho \in [1,R]} \rho^{n-2} + C R^{-2s} \sup_{\rho \in [1,R]} \rho^{2+2s-2} + C R^{-s} \sup_{\rho \in [1,R]} \rho^{s} \le C.
\end{align*}
Therefore, by \autoref{lemma:two-dim-energy-est-2} we deduce:
\begin{align*}
\left(\int_{B_{1}(x_0)} \frac{(u(x) - u(x + t \nu))_+}{t} \d x \right) \left( \int_{B_{1}(x_0)} \frac{(u(y) - u(y + t\nu))_-}{t} \d y \right) \le \frac{C}{\log R} \mathcal{S}_R \le \frac{C}{\log R}.
\end{align*}
This implies the first result. The second claim follows by taking the limit $R \to \infty$.
\end{proof}

We are now in a position to conclude the proof of the two-dimensional classification result.

\begin{proof}[Proof of \autoref{thm:two-dimensional-classification}]
As a consequence of \autoref{lemma:two-dim-energy-est-3}, we deduce that for any $\nu \in \mathbb{S}^{1}$, any $x_0 \in \R^n$, and any $ t \in (-1,1) \setminus \{ 0 \}$ it holds
\begin{align*}
\text{either} \qquad \frac{u(x) - u(x + t \nu)}{t} \ge 0 ~~ \forall x \in B_1(x_0), \qquad \text{or} \qquad \frac{u(x) - u(x + t \nu)}{t} \le 0 ~~ \forall x \in B_1(x_0).
\end{align*}
Thus, by varying $x_0$, we get that for any $\nu \in \mathbb{S}^{1}$ and $t \in (-1,1)$ it holds
\begin{align*}
\text{either} \qquad u(x) \ge u(x + t \nu) ~~ \forall x \in \R^2 , \qquad \text{or} \qquad u(x) \le u(x + t \nu) ~~ \forall x \in \R^2.
\end{align*}
Next, by varying $t$, and using the continuity of $u$, we deduce that 
\begin{align*}
\text{either} \qquad u(x) \ge u(x + t \nu) ~~ \forall x \in \R^2, ~~ \forall t \in \R , \qquad \text{or} \qquad u(x) \le u(x + t \nu) ~~ \forall x \in \R^2,~~ \forall t \in \R,
\end{align*}
i.e., $u$ is monotone in every coordinate direction. This implies that $u$ is one-dimensional and monotone, i.e., there exist $e \in \mathbb{S}^{1}$ and $\phi : \R \to \R$ such that $u(x) = \phi(x \cdot e)$.
Since $u$ is a non-trivial, but $0 \in \partial \{ u > 0 \}$, we must have $\{ u > 0 \} = \{ x \cdot e > 0 \}$ (up to a rotation). Then, since $u$ is a minimizer of $\mathcal{I}_{\Omega}$, in the view of \autoref{lemma:aux} and the Liouville theorem in the half-space (see \cite[Theorem 2.7.2]{FeRo23}), we have that
\begin{align*}
u(x) = \kappa (x \cdot e)_+^s
\end{align*}
for some $\kappa > 0$. Finally, by \autoref{prop:free-bound-cond}, we deduce that $\kappa = A(e)$, as desired.
\end{proof}

\end{document}